\documentclass[letterpaper,12pt]{article}
\usepackage[top=1in,bottom=1.25in, left=1in, right=1in]{geometry}
\usepackage{hyperref}
\usepackage{amsmath}
\usepackage{booktabs}
\usepackage[font=footnotesize]{caption}
\usepackage{subcaption}
\usepackage{algorithm}
\usepackage{algpseudocode}
\usepackage{amssymb,amsthm}
\usepackage{array}
\usepackage{enumerate}
\usepackage{enumitem}
\usepackage{graphicx}
\usepackage[dvipsnames]{xcolor}
\usepackage{natbib}
\bibpunct{(}{)}{;}{a}{}{,}

\title{Minimax rates for latent position estimation in the generalized random dot product graph}
\author{Hao Yan and Keith D. Levin\\
Department of Statistics, University of Wisconsin--Madison}

\begin{document}
\newtheorem{theorem}{Theorem}
\newtheorem{lemma}{Lemma}
\newtheorem{proposition}{Proposition}
\newtheorem{observation}{Observation}
\newtheorem{remark}{Remark}
\newtheorem{example}{Example}
\newtheorem{definition}{Definition}
\newtheorem{corollary}{Corollary}
\newtheorem{assumption}{Assumption}

\renewcommand{\AA}[1]{A^{(#1)}}
\newcommand{\HH}[1]{H^{(#1)}}
\newcommand{\nub}[1]{(\nu_{#1} + b_{#1}^2)}

\newcommand{\ssigma}[1]{\rho^{(#1)}}
\newcommand{\mxbernsymbol}{\eta}

\renewcommand{\Pr}{\mathbb{P}}
\newcommand{\E}{\mathbb{E}}
\newcommand{\N}{\mathbb{N}}
\newcommand{\R}{\mathbb{R}}
\newcommand{\Z}{\mathbb{Z}}

\newcommand{\bbB}{\mathbb{B}}
\newcommand{\bbO}{\mathbb{O}}
\newcommand{\bbP}{\mathbb{P}}
\newcommand{\bbR}{\mathbb{R}}

\newcommand{\Od}{\bbO_d}
\newcommand{\Opq}{\bbO_{p,q}}

\newcommand{\Rnonneg}{\R_{\ge 0}}

\newcommand{\calA}{\mathcal{A}}
\newcommand{\calC}{\mathcal{C}}
\newcommand{\calD}{\mathcal{D}}
\newcommand{\calL}{\mathcal{L}}
\newcommand{\calN}{\mathcal{N}}
\newcommand{\calP}{\mathcal{P}}
\newcommand{\calR}{\mathcal{R}}
\newcommand{\calS}{\mathcal{S}}
\newcommand{\calU}{\mathcal{U}}
\newcommand{\calV}{\mathcal{V}}
\newcommand{\calX}{\mathcal{X}}

\newcommand{\Stiefel}{\calS}

\newcommand{\Ahat}{\hat{A}}
\newcommand{\Ohat}{\hat{O}}
\newcommand{\Phat}{\hat{P}}
\newcommand{\Rhat}{\hat{R}}
\newcommand{\Xhat}{\hat{X}}
\newcommand{\dhat}{\hat{d}}
\newcommand{\uhat}{\hat{u}}
\newcommand{\what}{\hat{w}}
\newcommand{\rhohat}{\hat{\rho}}
\newcommand{\sigmahat}{\hat{\sigma}}
\newcommand{\tauhat}{\hat{\tau}}
\newcommand{\nuhat}{\hat{\nu}}
\newcommand{\chat}{\hat{c}}

\newcommand{\ccheck}{\check{c}}
\newcommand{\Xcheck}{\check{X}}
\newcommand{\Aml}{\mathring{A}}
\newcommand{\Vml}{\mathring{V}}
\newcommand{\Xml}{\mathring{X}}
\newcommand{\wml}{\mathring{w}}

\newcommand{\Abar}{\bar{A}}
\newcommand{\Pbar}{\bar{P}}
\newcommand{\Xbar}{\bar{X}}
\newcommand{\cbar}{\bar{c}}
\newcommand{\nubar}{\bar{\nu}}
\newcommand{\bbar}{\bar{b}}
\newcommand{\vbar}{\bar{v}}

\newcommand{\Atilde}{\tilde{A}}
\newcommand{\UAtilde}{U_{\Atilde}}
\newcommand{\SAtilde}{S_{\Atilde}}
\newcommand{\Stilde}{{\tilde{S}}}
\newcommand{\Otilde}{\tilde{O}}
\newcommand{\Ptilde}{\tilde{P}}
\newcommand{\Xtilde}{\tilde{X}}
\newcommand{\dtilde}{\tilde{d}}
\newcommand{\wtilde}{\tilde{w}}
\newcommand{\rhotilde}{\tilde{\rho}}
\newcommand{\calDtilde}{\tilde{\calD}}
\newcommand{\calXtilde}{\tilde{\calX}}

\newcommand{\dH}{d_{\text{H}}}
\newcommand{\dHtilde}{\dtilde_{\text{H}}}

\newcommand{\bB}{\mathbf{B}}
\newcommand{\bZ}{\mathbf{Z}}

\newcommand{\Xstar}{X^*}
\newcommand{\bstar}{b^*}
\newcommand{\nustar}{\nu^*}
\newcommand{\kappastar}{\kappa_{\star}}
\newcommand{\lambdastar}{\lambda_{\star}}
\newcommand{\nmin}{n_{\min}}

\newcommand{\wopt}{\wml}

\newcommand{\SA}{S_A}
\newcommand{\SM}{S_M}
\newcommand{\SP}{S_P}
\newcommand{\UP}{U_P}
\newcommand{\UA}{U_A}
\newcommand{\UM}{U_M}

\newcommand{\RDPG}{\operatorname{RDPG}}
\newcommand{\GRDPG}{\operatorname{GRDPG}}
\newcommand{\ASE}{\operatorname{ASE}}
\newcommand{\KL}{\operatorname{KL}}
\newcommand{\VAR}{\operatorname{Var}}
\newcommand{\COV}{\operatorname{Cov}}
\newcommand{\wtd}{\operatorname{wtd}}
\newcommand{\unif}{\operatorname{unif}}
\newcommand{\Bernoulli}{\operatorname{Bern}}
\newcommand{\supp}{\operatorname{supp}}
\newcommand{\rank}{\operatorname{rank}}
\newcommand{\diag}{\operatorname{diag}}
\newcommand{\sign}{\operatorname{sign}}
\newcommand{\tr}{\operatorname{tr}}
\newcommand{\tti}{2,\infty}

\newcommand{\dtildetti}{\dtilde_{\tti}}

\newcommand{\Phatwtd}{\Ptilde}
\newcommand{\Phatunif}{\Pbar}

\newcommand{\onevec}{\vec{\mathbf{1}}}
\newcommand{\ivec}{\vec{i}}
\newcommand{\jvec}{\vec{j}}
\newcommand{\indicator}{\mathbb{I}}
\newcommand{\indic}{\indicator}

\newcommand{\marginal}[1]{\marginpar{\raggedright\scriptsize #1}}
\newcommand{\as}{\text{~~~a.s.}}
\newcommand{\whp}{\text{~~~w.h.p.}}
\newcommand{\inlaw}{\xrightarrow{\calL}}
\newcommand{\inprob}{\xrightarrow{P}}
\newcommand{\iid}{\stackrel{\text{i.i.d.}}{\sim}}
\newcommand{\eqdist}{\stackrel{d}{=}}
\newcommand{\Perm}{\Pi}


\newcommand{\mH}{\mathbf{H}}
\newcommand{\mJ}{\mathbf{J}}
\newcommand{\mK}{\mathbf{K}}
\newcommand{\mU}{\mathbf{U}}
\newcommand{\mV}{\mathbf{V}}
\newcommand{\mW}{\mathbf{W}}
\newcommand{\mA}{\mathbf{A}}
\newcommand{\mB}{\mathbf{B}}
\newcommand{\mI}{\mathbf{I}}
\newcommand{\mP}{\mathbf{P}}
\newcommand{\mX}{\mathbf{X}}
\newcommand{\mY}{\mathbf{Y}}
\newcommand{\mZ}{\mathbf{Z}}
\newcommand{\mQ}{\mathbf{Q}}
\newcommand{\mR}{\mathbf{R}}
\newcommand{\mG}{\mathbf{G}}
\newcommand{\mzero}{\mathbf{0}}
\newcommand{\mLambda}{\boldsymbol{\Lambda}}
\newcommand{\mPi}{\boldsymbol{\Pi}}
\newcommand{\mXhat}{\hat{\mathbf{X}}}
\newcommand{\vx}{\mathbf{x}}
\newcommand{\vy}{\mathbf{y}}
\newcommand{\ve}{\mathbf{e}}
\newcommand{\vu}{\mathbf{u}}
\newcommand{\vv}{\mathbf{v}}
\newcommand{\vo}{\mathbf{0}}
\newcommand{\vw}{\mathbf{w}}
\newcommand{\va}{\mathbf{a}}
\newcommand{\vz}{\mathbf{z}}
\newcommand{\vmu}{\mathbf{\mu}}
\newcommand{\tmU}{\tilde{\mU}}
\newcommand{\tmV}{\tilde{\mV}}
\newcommand{\mSig}{\boldsymbol{\Sigma}}
\newcommand{\mDelta}{\boldsymbol{\Delta}}
\newcommand{\mGamma}{\boldsymbol{\Gamma}}
\newcommand{\tmSig}{\tilde{\boldsymbol{\Sigma}}}
\newcommand{\tmLambda}{\tilde{\boldsymbol{\Lambda}}}
\newcommand{\tildV}{\tilde{V}}
\newcommand{\tmDelta}{\Tilde{\boldsymbol{\Delta}}}
\newcommand{\mM}{\mathbf{M}}
\newcommand{\tmM}{\Tilde{\mM}}

\maketitle
\abstract{
Latent space models play an important role in the modeling and analysis of network data. Under these models, each node has an associated latent point in some (typically low-dimensional) geometric space, and network formation is driven by this unobserved geometric structure. The random dot product graph (RDPG) and its generalization (GRDPG) are latent space models under which this latent geometry is taken to be Euclidean. These latent vectors can be efficiently and accurately estimated using well-studied spectral embeddings. In this paper, we develop a minimax lower bound for estimating the latent positions in the RDPG and the GRDPG models under the two-to-infinity norm, and show that a particular spectral embedding method achieves this lower bound. We also derive a minimax lower bound for the related task of subspace estimation under the two-to-infinity norm that holds in general for low-rank plus noise network models, of which the RDPG and GRDPG are special cases. The lower bounds are achieved by a novel construction based on Hadamard matrices.
}

\section{Introduction}

Networks encoding relations among entities are a common form of data in a broad range of scientific disciplines.
In neuroscience, networks encode the strength of connections among brain regions \citep{sporns_complex}.
In biology, networks encode which pairs of genes or proteins are co-expressed or are involved in the same pathways \citep{PPIsurvey}.
In the social sciences, networks arise naturally in the form of social network data \citep{Granovetter1973,TraMucPor2012,legramanti2022extended}.

Network embeddings are a broadly popular tool for exploring and analyzing network data.
These methods seek to represent the vertices of a network in a lower-dimensional (typically Euclidean) space, in such a way that the geometry of these embeddings reflects some network structure of interest.
Most commonly, these embeddings arise either via spectral methods \citep{RohChaYu2011,SusTanFisPri2012,tang2018limit}, which construct embeddings from the leading eigenvalues and eigenvectors of the adjacency matrix, or via representation learning methods \citep{node2vec,LinSusIsh2021}.
Embeddings are especially appropriate in settings where we believe that data is well-approximated by a latent space network model \cite{hoff2002latent}.
Under these models, each vertex has an associated latent variable (often a point in Euclidean space), and network formation is driven by these latent variables, with pairs of vertices more likely to form edges if their latent variables are ``similar'' according to some measure (e.g., proximity in space).
Examples of such models include Hoff models \citep{hoff2002latent,JMLR:v21:17-470}, random geometric graphs \citep{penrose2003random}, graph root distributions \citep{lei2021network} and graphons \citep{Lovasz2012}, to name just a few.

Among these latent space models is the random dot product graph \citep[RDPG;][]{young2007random,RDPGsurvey} and its generalization \citep[GRDPG;][]{Rubin_2017}.
Under this model, each node $v$ has an associated low-dimensional vector $\vx_v \in \R^d$, called its {\em latent position}.
Conditional on these latent positions, the probability of two nodes $u$ and $v$ sharing an edge is given by the inner product of the associated vectors $ \vx_u^T\vx_v $. 
Although the RDPG is simple and widely applicable, one limitation of the model is that it can only produce graphs whose expected adjacency matrices are positive semidefinite.
To overcome this drawback, \cite{Rubin_2017} introduced the generalized random dot product graph (GRDPG), which allows this expected adjacency matrix to be indefinite.
This model includes many classical models as special cases, including the stochastic block model \citep{HolLasLei1983}, degree corrected stochastic block model \citep{KarNew2011} and mixed membership stochastic block model \citep{AirBleFieXin2008}. 

Under the RDPG and GRDPG, the most basic inferential problem involves estimation of the latent positions based on an observed network.
Once estimates of the latent positions are obtained, they can be used in many downstream tasks such as clustering \citep{SusTanFisPri2012, LyzSusTanAthPri2014}, graph hypothesis testing \citep{tang2017semiparametric, tang2017nonparametric}, and bootstrapping \citep{levin2019bootstrapping}.
A widely-used approach to estimating the latent positions in the RDPG is the adjacency spectral embedding \citep[ASE;][]{SusTanFisPri2012}.
The consistency of the ASE has been established previously under both the spectral \citep{SusTanPri2014} and two-to-infinity \citep{LyzSusTanAthPri2014} norms and the asymptotic distributional behavior of this estimate was further explored in \cite{AthLyzMarPriSusTan2016,LevAthTanLyzPri2017}.
For other related approaches to estimating the latent positions under the RDPG, see~\cite{tang2018limit,xie2019optimal,WuXie2022,XieXu2023}. 
The latent positions of the GRDPG can also be estimated consistently using a slight modification of the ASE \citep{Rubin_2017}, with similar asymptotic distributional behavior to that established in previous work for the RDPG \citep{AthLyzMarPriSusTan2016,tang2018limit,LevAthTanLyzPri2017}.

These previous results suggest that the estimation rate, as measured in two-to-infinity norm, obtained by the ASE and related methods should be optimal, perhaps up to logarithmic factors.
In this paper, we show that this is indeed the case (see Theorem~\ref{thm:main:main}), establishing minimax lower bounds for estimation of the latent positions in a class of low-rank network models that includes both the RDPG and GRDPG.
This matches estimation upper bounds previously established in the literature \cite{LyzSusTanAthPri2014,Rubin_2017}, up to logarithmic factors, and is in accord with previous work by \cite{xie2019optimal} establishing the minimax rate under the Frobenius norm for the RDPG model.
Our proof is based on a novel construction using Hadamard matrices, which may be of interest to researchers in subspace estimation.
Indeed, as a corollary of our main result, we obtain minimax bounds for the closely related problem of singular subspace estimation in low-rank network models.
Previous results along these lines include \cite{cai2018rate}, who established a lower bound under Gaussian noise, and \cite{Zhou2021}, who provided a lower bound for random bipartite graphs under the spectral norm and Frobenius norm.

\paragraph{Notation.}
For a vector $\vx$, we use $\| \vx \|_2$ to denote its Euclidean norm.
For a matrix $\mA$, $\|\mA\|,\|\mA\|_F$ and $\|\mA\|_{2, \infty}$ denote the spectral, Frobenius, and two-to-infinity (see Equation~\eqref{eq:tti:definition}) norms, respectively.
We use $\mA_{ij}$ to denote the element in the $i$-th row and $j$-th column of the matrix $\mA$.
For a sequence of matrices, we use subscripts $\mA_{1}, \mA_{2}, \ldots, \mA_{n}$ to index them if we do not need to specify an element of them.
To specify the $(i, j)$ entry of a sequence of matrices, we use the notation  $\mA^{(1)}_{ij}, \mA^{(2)}_{ij}, \ldots, \mA^{(n)}_{ij}$.
Similarly, we use subscripts $\vx_{1}, \vx_2, \ldots, \vx_n$ to index a sequence of vectors.
We use letters $C$ and $c$ to denote constants, not depending on the problem size $n$, whose specific values may change from line to line.
$\bbO_d$ denotes the set of all $d\times d$ orthogonal matrices.
$\mI_d$ denotes the $d\times d$ identity matrix. 
$\mzero$ denotes a matrix of all zeros.
For a positive integer $n$, we let $[n] = \{1, 2, \ldots, n\}$.
We denote the standard basis in $\R^n$ by $\ve_1, \ve_2, \ldots, \ve_n$, where the components of $\ve_i$ are all zero, save for the $i$-th component, which is equal to 1.
We make use of standard use of Landau notation.
Thus, for positive sequences $(a_n)$ and $(b_n)$, if there exists a constant $C$ such that $a_n \leq C b_n$ for all suitably large $n$, then we write $a_n=O\left(b_n\right)$ or $a_n \lesssim b_n$, and we write $b_n = \Omega(a_n)$.
We write $a_n=\Theta\left(b_n\right)$ to denote that $a_n=O\left(b_n\right)$ and $b_n=O\left(a_n\right)$.
If $a_n / b_n \rightarrow 0$ as $n \rightarrow \infty$, then we write $a_n=o\left(b_n\right)$ and $b_n = \omega(a_n)$.

\section{Low-rank Models and Embeddings}
\label{sec:setup}

We are concerned in this paper with low-rank network models, in which the expected value of the adjacency matrix, perhaps conditional on latent variables, is of low rank.
These models are exemplified by the RDPG, where conditional on the latent positions, the adjacency matrix has expectation given by the Gram matrix of the latent positions.

\begin{definition}[RDPG; \cite{young2007random,RDPGsurvey}]
\label{def:RDPG}
Let $F$ be a distribution on $\bbR^d$ such that for all $\vx,\vy \in \supp F$, $0 \le \vx^T \vy \le 1$.
Let $\vx_1,\vx_2,\dots,\vx_n \in \bbR^{d}$ be drawn i.i.d.\ according to $F$, and collect them in
the rows of $\mX \in \bbR^{n \times d}$.
Conditional on $\mX$, generate a symmetric adjacency matrix $\mA \in \{0,1\}^{n \times n}$ according to
$\mA_{ij} \sim \Bernoulli( \vx_i^T \vx_j )$ independently over all $1 \le i < j \le n$.
Then we say that $\mA$ is the adjacency matrix of a {\em random dot product graph} (RDPG), and write
$(\mA,\mX) \sim \RDPG(F,n)$.
For a fixed choice of $\mX$, we write $\mA \sim \RDPG(\mX)$ and say that the resulting network is
distributed as a {\em conditional RDPG} with latent positions $\mX$.
\end{definition}

As defined, the (conditional) expected adjacency matrix $\E[ \mA \mid \mX ]$ is always positive semidefinite under the RDPG, restricting the range of network structures it can express.
The generalized RDPG (GRDPG) resolves this issue.

\begin{definition}[GRDPG; \cite{Rubin_2017}]
\label{def:GRDPG}
Let $d = p+q$ where $p,q \ge 0$ are integers, and define the matrix
\begin{equation} \label{eq:def:Ipq}
\mI_{p,q} = \diag( 1,1,\dots,1,-1,\dots,-1 ).
\end{equation}
Suppose that $F$ is a distribution on $\bbR^d$ such that $0 \le \vx^T \mI_{p,q} \vy \le 1$ for all $\vx,\vy \in \supp F$.
Draw $\vx_1,\vx_2,\dots,\vx_n \in \bbR^{d}$ i.i.d.\ according to $F$, and collect them in the rows of $\mX \in \bbR^{n \times d}$.
Conditional on $\mX$, generate a symmetric adjacency matrix $\mA \in \{0,1\}^{n \times n}$ according to
$\mA_{ij} \sim \Bernoulli( \vx_i^T \mI_{p,q} \vx_j )$ independently over all $1 \le i < j \le n$.
We say that $\mA$ is the adjacency matrix of a {\em generalized random dot product graph (GRDPG)} with signature $(p,q)$, and write
$(\mA,\mX) \sim \GRDPG(F,p,q,n)$.
For a fixed $\mX$ and signature $(p,q)$, we write $\mA \sim \GRDPG(\mX, p, q)$ and say that the resulting network is
distributed as a {\em conditional RDPG} with latent positions $\mX$ and signature $(p,q)$.
\end{definition}

We can naturally extend the conditional versions of these models to a generic ``low-rank plus noise'' network model, in which the expected adjacency matrix is low-rank.

\begin{definition}[Low rank network model]
\label{def:lrpn}
Let $d = p+q$ for non-negative integers $p$ and $q$, and let $\mI_{p,q}$ be as defined in Equation~\eqref{eq:def:Ipq}.
Let $\mX \in \bbR^{n \times d}$ be such that $\mP = \mX \mI_{p,q} \mX^T$ has all its entries between $0$ and $1$.
Given $\mP$, generate a symmetric binary adjacency matrix $\mA \in \{0,1\}^{n \times n}$ according to $\mA_{ij} \sim \Bernoulli( \mP_{ij} )$, independently over all $1 \le i < j \le n$.
We say that the resulting network is distributed according to a low-rank plus noise model with expectation $\mP$.
\end{definition}

Under both the RDPG and GRDPG as well as under their generalization in Definition~\ref{def:lrpn}, we have
\begin{equation*}
  \E[ \mA \mid \mX ] = \mP = \mX \mI_{p,q} \mX^T
\end{equation*}
for $\mI_{p,q}$ as in Equation~\eqref{eq:def:Ipq}.
Note that we recover the RDPG by taking $q=0$.
Under these models, the matrix $\mX \in \bbR^{n \times d}$ is a natural inferential target.
The aim of this paper is to establish the limits on estimating this low-rank part $\mX$ under network models like those in Definitions~\ref{def:RDPG},~\ref{def:GRDPG} and~\ref{def:lrpn}.

For non-negative integers $p,q$, define the set
\begin{equation} \label{eq:def:calX}
\calX_n^{(p, q)} = \{ \mX \in \R^{n \times d} : 0 \le \mX \mI_{p,q} \mX^T \le 1 \},
\end{equation}
where the inequality is meant entry-wise, so that for each $1 \le i < j \le n$, the element $(\mX \mI_{p,q} \mX^T)_{i,j}$ is a probability.
That is, the set $\calX_n^{(p,q)}$ corresponds to the collection of all possible collections of $n$ latent positions whose indefinite inner products under a signature $(p,q)$ are valid probabilities.
In other words, any $\mX \in \calX^{(p,q)}_n$ is a potential collection of latent positions under Definition~\ref{def:GRDPG} or~\ref{def:lrpn}.

When $p = d$, the GRDPG model recovers the random dot product graph (RDPG) model as a special case.
As such, we define
\begin{equation} \label{eq:def:calXd}
\calX^d_n = \{ \mX \in \R^{n \times d} : 0 \le \mX \mX^T \le 1 \}.
\end{equation}

To establish estimation rates for network latent positions (i.e., elements of the set $\calX^{(p,q)}_n$ or $\calX^d_n$), we must endow the set with a distance.
One such distance, surely the most studied in the context of network modeling, derives from the $(\tti)$-norm.
Given two matrices $\mX,\mY \in \bbR^{n \times d}$, this norm is defined according to
\begin{equation}
\label{eq:tti:definition}
\| \mX - \mY \|_{\tti} = \max_{i \in [n]} \| \mX_i - \mY_i \|_2,
\end{equation}
where $\|\cdot\|_2$ is the standard Euclidean norm in $\bbR^d$ and $\mX_{i} \in \bbR^d$ denotes the $i$-th row of $\mX \in \bbR^{n \times d}$, viewed as a column vector.

We will use this norm to construct a distance on the set $\calX^d_n$, once we account for a non-identifiability inherent to latent space models \citep{ShaAst2017}.
Observe that for any  orthogonal transformation $\mW \in \bbO_d$, we have $\mX \mX^T = \mX \mW (\mX \mW)^T$.
As a result, given an adjacency matrix $\mA$ generated from an RDPG, we can only hope to estimate a particular $\mX \in \calX^d_n$ up to such an orthogonal transformation.
We thus endow $\calX^d_n$ with an equivalence relation $\sim$, writing $\mX \sim \mY$ if $\mY = \mX\mW$ for some $\mW \in \bbO_d$.
Our notion of recovering the rows of the true $\mX$ up to orthogonal rotation yields a natural notion of distance on these equivalence classes.

\begin{definition} \label{def:dtildetti}
Let $\calXtilde^d_n$ denote the quotient set of $\calX^d_n$ by $\sim$.
Denoting elements of $\calXtilde^d_n$ by $[\mX]$ for any class representative $\mX \in \calX^d_n$, define a distance on $\calXtilde^d_n$ by
\begin{equation*}
\dtildetti\left( [\mX],[\mY] \right)
= \min_{\mW \in \bbO_d} \| \mX - \mY \mW \|_{\tti}.
\end{equation*}
\end{definition}

\begin{observation}
\label{obs:main:1}
$\dtildetti$ is a distance on $\calXtilde^d_n$.
\end{observation}
\begin{proof}
Symmetry and non-negativity of $\dtildetti$ are immediate from the definition and invariance of the $(\tti)$-norm under right-multiplication by elements of $\bbO_d$.
Similarly, it follows by definition that $\dtildetti([\mX],[\mY] )=0$ if and only if $[\mX]=[\mY]$.

To establish the triangle inequality, note that for $[\mX],[\mY],[\mZ] \in \calXtilde^d_n$, we have
\begin{equation*} \begin{aligned}
\dtildetti\left( [\mX],[\mY] \right)
&= \min_{\mW \in \bbO_d} \| \mX - \mY \mW \|_{\tti}
= \min_{\mW \in \bbO_d, \mW' \in \bbO_d} \| \mX - \mZ \mW' + \mZ \mW' - \mY \mW \|_{\tti} \\
&\le \min_{\mW \in \bbO_d, \mW' \in \bbO_d} \| \mX - \mZ \mW' \|_{\tti}
        +  \| \mZ \mW' - \mY \mW \|_{\tti} \\
&= \min_{\mW \in \bbO_d} \| \mX - \mZ \mW \|_{\tti}
        + \min_{\mW \in \bbO_d} \| \mZ - \mY \mW \|_{\tti} \\
&= \dtildetti\left( [\mX],[\mZ] \right) + \dtildetti\left( [\mZ],[\mY] \right),
\end{aligned} \end{equation*}
where we have used the fact that the $(\tti)$-norm is invariant under
right-multiplication by an orthogonal matrix.
\end{proof}

Under the GRDPG and other low-rank network models (i.e., Definitions~\ref{def:GRDPG} and~\ref{def:lrpn}), a similar non-identifiability occurs, but its structure is complicated by the presence of the matrix $\mI_{p,q}$.
Analogous to the orthogonal group $\bbO_d$, we denote the indefinite orthogonal group by
\begin{equation*}
\Opq
= \{\mQ \in \R^{d\times d} : \mQ \mI_{p,q} \mQ^T = \mI_{p, q}\}.
\end{equation*}
For any matrix $\mQ \in \Opq$ and any $\mX \in \calX_{n}^{(p,q)}$, we have $\mX\mI_{p,q}\mX^T = \mX\mQ\mI_{p,q}(\mX\mQ)^T$.
As a result, under the GRDPG, the conditional distribution of $\mA$ remains unchanged if we replace $\mX$ with $\mX\mQ$ for any $\mQ \in \Opq$.
Thus, we also consider an equivalence relation $\sim$ on $\calX_n^{(p, q)}$, whereby for $\mX,\mY \in \calX_n^{(p,q)}$, we write $\mX \sim \mY$ if and only if $\mY = \mX\mQ$ for some $\mQ \in \Opq$. 
Lemma~\ref{lem:setup:equiv} shows that the equivalence classes under this relation correspond precisely to the matrices $\mX \in \calX_n^{(p,q)}$ that give rise to the same distribution over networks.
A proof can be found in Appendix~\ref{sec:apx:setup:equiv:proof}

\begin{lemma} \label{lem:setup:equiv}
For $\mX, \mY \in \calX_n^{(p,q)}$, define respective probability matrices $\mP_{\mX} = \mX \mI_{p,q} \mX^T$ and $\mP_{\mY} = \mY \mI_{p,q} \mY^T$.
Then $\mX \sim \mY$ if and only if $\mP_{\mX}=\mP_{\mY}$.
\end{lemma}

In light of Lemma~\ref{lem:setup:equiv}, our equivalence relation can also be understood as $\mX \sim \mY$ if and only if $\mP_\mX = \mP_{\mY}$.
Under this equivalence relation, we denote by $\calXtilde_{n}^{(p, q)}$ the set of equivalence classes of $\calX_n^{(p,q)}$ under $\sim$.
When it is clear from the context, we also use $[\mX]$ to denote the element of $\calXtilde_n^{(p, q)}$ corresponding to the equivalence class of $\mX \in \calX^{(p,q)}_n$. 

In order to show minimax results for estimation of the latent positions in the GRDPG model and related low-rank network models, we first need to fix a notion of distance over the parameter set $\calXtilde_n^{(p, q)}$.
To account for non-identifiability in the GRDPG, it is natural to follow Definition~\ref{def:dtildetti} and define the distance between $[\mX]$ and $[\mY]$ according to
\begin{equation} \label{eq:ind-dist}
    \inf_{\mQ_1, \mQ_2 \in \Opq} \left\|\mX \mQ_1 - \mY \mQ_2\right\|_{\tti}.
\end{equation}
Unfortunately, this definition is not necessarily a valid distance.
To see a simple example, consider the case when $n=1$, $p=1$ and $q=1$. 
For any $\vx_0 = (x_{0,1}, x_{0,2}) \in \mR^2$ such that $x_{0,1}^2 - x_{0,2}^2 = r$, we observe that $\mQ \vx_0$ moves $\vx_0$ along the curve $C_r: x_1^2 - x_2^2 = r$.
Notice that for all $r \in \R$, $C_r$ shares a common asymptote $l: x_1 - x_2 = 0$.
Therefore, it follows that for any $\vx$ and $\vy \in \R^2$, 
\begin{equation*}
    \inf_{\mQ_1, \mQ_2 \in \mathcal{O}_{1,1}} \left\|\mQ_1 \vx - \mQ_2\vy\right\|_{2} = 0. 
\end{equation*}
Furthermore, the quantity defined in Equation~\eqref{eq:ind-dist} may not satisfy the triangle inequality.
We include an example for $n=2, p=1$ and $q = 1$ in Section \ref{sec:apx:dist}.
Instead, we must take a slightly more careful route to define a distance on $\calXtilde_n^{(p,q)}$.

We begin by noting that for any $\mX \in \calX_n^{(p,q)}$, Sylvester's law of inertia implies that $\mP_{\mX} = \mX \mI_{p,q} \mX^T$, has $p$ positive eigenvalues, $q$ negative eigenvalues and the remaining $n-p-q$ eigenvalues are zero.
Thus, we can always decompose $\mP_{\mX}$ as
\begin{equation*}
\mP_{\mX} = \mU_{\mX} \mLambda_{\mX}^{1/2} \mI_{p, q} \mLambda_{\mX}^{1/2} \mU_{\mX}^T,
\end{equation*}
where $\mU_{\mX} \in \R^{n\times d}$ is a matrix with orthonormal columns and $\mLambda_{\mX}\in \R^{d\times d}$ is a diagonal matrix with positive on-diagonal entries.
By Lemma~\ref{lem:setup:equiv}, we have $\mU_{\mX} \mLambda_{\mX}^{1/2} \in [\mX]$, since $\mU_{\mX} \mLambda_{\mX}^{1/2}$ and $\mX$ both produce the same probability matrix $\mP_{\mX}$. 
In light of this, we can define a distance $\dtilde_{2,\infty}$ on $\calXtilde_n^{(p,q)}$ according to
\begin{equation}
\label{eq:setup:dtildetti}
\dtildetti\left( [\mX],[\mY] \right)
= \min_{\mW \in \bbO_d\cap \Opq} \left\| \mU_{\mX}\mLambda_{\mX}^{1/2} - \mU_{\mY}\mLambda_{\mY}^{1/2} \mW \right\|_{\tti}.
\end{equation}
The reader may notice that we have used the same notation $\dtildetti$ as in Definition~\ref{def:dtildetti}.
This can be done without risk of confusion:
when $p = d$ and $q = 0$, since $\mU_{\mX}\mLambda_{\mX}^{1/2} \in [\mX]$ and $\mU_{\mY}\mLambda_{\mY}^{1/2} \in [\mY]$, there exist $\mW_{\mX}, \mW_{\mY} \in \Od$ such that $\mX\mW_{\mX} = \mU_{\mX}\mLambda_{\mX}^{1/2}$ and $\mY\mW_{\mY} = \mU_{\mY}\mLambda_{\mY}^{1/2}$. As a result, since $\Opq = \bbO_d$ when $p=d$, we have
\begin{equation*} \begin{aligned}
\dtildetti\left( [\mX],[\mY] \right)
&= \min_{\mW \in \bbO_d} \left\| \mU_{\mX}\mLambda_{\mX}^{1/2} - \mU_{\mY}\mLambda_{\mY}^{1/2} \mW \right\|_{\tti}\\
&= \min_{\mW \in \bbO_d} \left\| \mX\mW_{\mX} - \mY\mW_{\mY} \mW \right\|_{\tti}\\
&= \min_{\mW \in \bbO_d} \left\| \mX - \mY\mW \right\|_{\tti},
\end{aligned} \end{equation*}
which is precisely our definition of $\dtildetti$ for the RDPG as given in Definition~\ref{def:dtildetti}.

\begin{observation}
$\dtildetti$ is a distance on $\calXtilde_{n}^{(p, q)}$.
\end{observation}
\begin{proof}
    Symmetry of $\dtildetti$ follows from the fact that $\mW^T \in \bbO_d\cap \Opq$ whenever $\mW \in \bbO_d\cap \Opq$, and non-negativity is immediate from the fact that $\| \cdot \|_{\tti}$ is a norm.
    The triangle inequality follows from the same argument as given in Observation~\ref{obs:main:1}.
    
    It remains to show that
    \begin{equation} \label{eq:dtilde:zero:iff}
    \dtildetti\left( [\mX],[\mY] \right) = 0
    ~~~\text{ if and only if }~~~
    [\mX] = [\mY].
    \end{equation}

    Toward this end, suppose that $\dtildetti\left( [\mX],[\mY] \right) = 0$.
    Since $\bbO_d \cap \Opq$ is compact, there exists $\mW^\star \in \bbO_d \cap \Opq$ such that
    \begin{equation*}
    \left\| \mU_{\mX}\mLambda_{\mX}^{1/2} - \mU_{\mY}\mLambda_{\mY}^{1/2} \mW^\star \right\|_{2,\infty} = 0,
    \end{equation*}
    that is to say, $\mU_{\mX}\mLambda_{\mX}^{1/2} = \mU_{\mY}\mLambda_{\mY}^{1/2} \mW^\star$. 
    We therefore have
    \begin{equation*} \begin{aligned}
	\mP_{\mX}
	&= \mX\mI_{p,q}\mX^T
	= \mU_{\mX}\mLambda_{\mX}^{1/2} \mI_{p, q} \mLambda_{\mX}^{1/2} \mU_{\mX}^T \\
	&=  \mU_{\mY}\mLambda_{\mY}^{1/2} \mW^\star\mI_{p, q}\mW^{\star T} \mLambda_{\mY}^{1/2}\mU_{\mY}^T
        = \mU_{\mY}\mLambda_{\mY}^{1/2} \mI_{p, q} \mLambda_{\mY}^{1/2} \mU_{\mY}^T \\
	&= \mY\mI_{p,q}\mY^T = \mP_{\mY},
    \end{aligned} \end{equation*}
    and Lemma~\ref{lem:setup:equiv} implies that $\mX \sim \mY$.
    
    To show the other direction of the equivalence in Equation~\eqref{eq:dtilde:zero:iff}, let $\mX,\mY \in \calX^{(p,q)}_n$ be representatives of $[\mX],[\mY] \in \calXtilde^{(p,q)}_n$, respectively, and suppose that $[\mX] = [\mY]$.
    We will show there exists a matrix $\mW^\star \in \bbO_d \cap \Opq$ such that $\mU_{\mX}\mLambda_{\mX}^{1/2} = \mU_{\mY}\mLambda_{\mY}^{1/2} \mW^\star$, whence it will follow that $\dtildetti([\mX],[\mY]) = 0$.    
    Recall that we associate to $\mX$ and $\mY$ the probability matrices
    \begin{equation*} \begin{aligned}
    \mP_{\mX} &= \mX \mI_{p,q} \mX^T 
    = \mU_{\mX} \mLambda_{\mX}^{1/2}
        \mI_{p,q} \mLambda_{\mX}^{1/2} \mU_{\mX}^T~~~\text{ and }\\
    \mP_{\mY} &= \mY \mI_{p,q} \mY^T
    = \mU_{\mY} \mLambda_{\mY}^{1/2}
        \mI_{p,q} \mLambda_{\mY}^{1/2} \mU_{\mY}^T,
    \end{aligned} \end{equation*}
    where $\mU_{\mX}, \mU_{\mY} \in \bbR^{n \times d}$ both have orthonormal columns and $\mLambda_{\mX}, \mLambda_{\mY} \in \bbR^{d \times d}$ are diagonal and positive definite.

    Since $[\mX] = [\mY]$, by Lemma~\ref{lem:setup:equiv} there exists $\mQ \in \Opq$ such that 
    \begin{equation}
    \label{eq:Q:mid-step}
        \mU_{\mX} \mLambda_{\mX}^{1/2} = \mU_{\mY} \mLambda_{\mY}^{1/2} \mQ. 
    \end{equation}
    There also exists a $\mW \in \Od$ such that $\mU_\mX = \mU_{\mY} \mW$, since $\mU_\mX$ and $\mU_{\mY}$ corresponds to the same singular subspaces. We also have a permutation matrix $\mPi$ such that $\mLambda_{\mX}^{1/2}\mI_{p,q}\mLambda_{\mX}^{1/2} = \mPi\mLambda_{\mY}^{1/2}\mI_{p,q}\mLambda_{\mY}^{1/2} \mPi^T$. The presence of $\mI_{p,q}$ forces $\mPi$ to be of the form 
    \begin{equation*}
    \mPi = \begin{bmatrix}
    \mPi_p & 0 \\
    0 & \mPi_q
    \end{bmatrix},
    \end{equation*}
    where $\mPi_p \in \R^{p\times p}$ and $\mPi_q \in \R^{q\times q}$ are permutation matrices. 
    Hence, $\mPi \in \Opq\cap\Od$ and we also have that $\mLambda_{\mX}^{1/2} = \mPi\mLambda_{\mY}^{1/2}\mPi^T$. It follows from Equation~\eqref{eq:Q:mid-step} that
    \begin{equation*}
        \mQ\mPi = \mLambda_{\mY}^{-1/2}\mW\mPi\mLambda_{\mY}^{1/2}.
    \end{equation*}
    Denote $\mV = \mW\mPi$ for ease of notation.
    Since $\mQ\mPi\in\Opq$, we have
    \begin{equation*}
    \mLambda_{\mY}^{-1/2}\mV\mLambda_{\mY}^{1/2}\mI_{p,q}\mLambda_{\mY}^{1/2}\mV^T\mLambda_{\mY}^{-1/2} = \mI_{p,q}.
    \end{equation*}
    Rearranging and using the fact that diagonal matrices commute,
    \begin{equation*}
    \mV\mLambda_{\mY}\mI_{p,q} = \mLambda_{\mY}\mI_{p,q}\mV.
    \end{equation*}
    Therefore, for any $i, j \in [d]$, we have
    $
    \mV_{ij}(\mLambda_{\mY}\mI_{p,q})_{jj} = \mV_{ij}(\mLambda_{\mY}\mI_{p,q})_{ii}.
    $
    If $\mV_{ij} \neq 0$, we have $(\mLambda_\mY\mI_{p,q})_{jj} = (\mLambda_\mY\mI_{p,q})_{ii}$ and thus $(\mLambda_\mY^{1/2}\mI_{p,q})_{jj} = (\mLambda_\mY^{1/2}\mI_{p,q})_{ii}$. Otherwise, we have
    \begin{equation*}
    \mV_{ij}(\mLambda_{\mY}^{1/2}\mI_{p,q})_{jj} = \mV_{ij}(\mLambda_{\mY}^{1/2}\mI_{p,q})_{ii} = 0.
    \end{equation*}
    Hence, $\mV_{ij}(\mLambda_{\mY}^{1/2}\mI_{p,q})_{jj} = \mV_{ij}(\mLambda_{\mY}^{1/2}\mI_{p,q})_{ii}$ always holds and
    it follows that 
    \begin{equation*}
    \mV\mLambda_{\mY}^{1/2}\mI_{p,q} = \mLambda_{\mY}^{1/2}\mI_{p,q}\mV.
    \end{equation*}
    Thus, we have 
    \begin{equation*}
    \mQ\mPi~\mI_{p,q}
    = \mLambda_{\mY}^{-1/2}\mV\mLambda_{\mY}^{1/2}\mI_{p,q}
    = \mLambda_{\mY}^{-1/2}\mLambda_{\mY}^{1/2}\mI_{p,q}\mV
    = \mI_{p,q}\mV.
    \end{equation*}
    Moving $\mPi~\mI_{p,q}$ to the right hand side, we have $\mQ = \mI_{p,q}\mV~\mI_{p,q}\mPi^T$, implying that $\mQ$ is an orthogonal matrix, whence $\mQ \in \Opq \cap \Od$. Taking $\mW^\star = \mQ$ completes the proof.
\end{proof}

The minimax risk for estimating $\mX \in \calX^{(p,q)}_n$ under the $(\tti)$-norm after accounting for the equivalence structure encoded in $\calXtilde^{(p,q)}_n$ is given by \citep{Tsybakov2009}
\begin{equation*}
\inf_{\mXhat} \sup_{\mX \in \calX_{n}^{(p,q)}}
	\E \dtildetti\left( [\mXhat], [\mX] \right)
=
\inf_{\mXhat} \sup_{\mX \in \calX_{n}^{(p,q)}}
        \E \min_{\mW \in \bbO_d\cap\Opq} \left\| \hat{\mU}_{\mXhat}\hat{\mLambda}^{1/2}_{\mXhat} - \mU_{\mX}\mLambda^{1/2}_{\mX} \mW \right\|_{\tti},
\end{equation*}
where the infimum is over all estimators $\mXhat$.
Our goal in the remainder of this paper is to lower-bound this minimax risk.

\section{Main Results}
\label{sec:results}

We consider estimation (up to orthogonal non-identifiability) of a low-rank matrix $\mX = \mU \mLambda^{1/2}$, where $\mU$ is an element of the Stiefel manifold of all $d$-frames in $\R^d$,
\begin{equation*}
\Stiefel_d(\R^n) = \left\{\mU\in \R^{n\times d}: \mU^T \mU = \mI_d\right\}.
\end{equation*}
The structure of $\mLambda$ plays a crucial role in the estimation of $\mX$.
When the smallest eigenvalues of $\E[\mA \mid \mX]$ are especially close to zero, it is hard to distinguish the $d$ ``signal'' eigenvalues of $\mA$ from the ``noise'' associated with the remaining $n-d$ eigenvalues.
As such, we consider a particular structure on $\mLambda = \diag(\lambda_1,\lambda_2,\dots,\lambda_d)$.
Assuming without loss of generality that $\lambda_1 \geq \lambda_2 \geq \dots \geq \lambda_d$ and defining the condition number $\kappa = \kappa(\mLambda) = \lambda_1/ \lambda_d$, this spectral structure is captured by membership in the set
\begin{equation*} \begin{aligned}
\calC(\kappastar, \lambdastar)
= \Big\{\mLambda = \diag(\lambda_1,\lambda_2,\dots,\lambda_d) \in \R^{d\times d}: &~ \kappa(\mLambda) \leq \kappastar, \lambda_d \geq \lambdastar > 0\Big\}.
\end{aligned} \end{equation*}

With this notation in hand, we can state our main result.

\begin{theorem} \label{thm:main:main}
With the sets $\Stiefel_d(\R^n)$ and $\calC(\kappa_{\star},\lambda_{\star})$ as defined above, define the set
\begin{equation*}
\calP(\kappastar, \lambdastar, p, q) = \left\{\left(\mU, \mLambda\right): \mU\in \Stiefel_d(\R^n), \mLambda \in \mathcal{C}(\kappastar, \lambdastar), \mU\mLambda^{1/2} \in \mathcal{X}_n^{(p,q)}\right\}.
\end{equation*}
If $\kappa_{\star} = o\left(n\right)$, $\kappa_{\star} \geq 3d$ and $3\kappastar\lambdastar \leq n$, then
\begin{equation} \label{eq:main:minimax}
    \inf_{\left(\hat{\mU}, \hat{\mLambda} \right)}~
    \sup_{(\mU, \mLambda) \in \calP(\kappastar, \lambdastar, p, q)}~
        \E~
        \dtildetti\left( 
        \left[\hat{\mU}\hat{\mLambda}^{1/2}\right],
                \left[\mU\mLambda^{1/2}\right] \right)
        ~\gtrsim~ \sqrt{\frac{\kappa_{\star}(\lambda_\star \wedge \log n )}{n}}. 
\end{equation}
\end{theorem}
\begin{proof}
Our main tool is a standard packing argument \citep[see Theorem 2.7 in][]{Tsybakov2009}.
The main technical hurdle is constructing a collection of elements of $\Stiefel_d(\R^n)$ all of which produce valid elements of $\calP(\kappastar, \lambdastar, p, q)$ when paired with  a particular choice of $\mLambda$.
Our construction is based on stacking Hadamard matrices to form $\mU \in \Stiefel_d(\R^n)$.
In particular, we require very different constructions depending on the growth rate of the condition number $\kappastar$, and we divide our proof of Theorem~\ref{thm:main:main} into two cases accordingly.
Details are given in the Appendix.
\end{proof}

As a remark, we note that the factor $3$ in the conditions $\kappastar \geq 3d$ and $3\kappastar\lambdastar \leq n$ can each be relaxed to $(1+\epsilon)$ and $(2+\epsilon)$, respectively, for any constant $\epsilon > 0$.
Details are provided in the Appendix.

\subsection{Illustrative Examples and Applications}
\label{subsec:main:examples}

We now apply our main result to some well-studied special cases from the network modeling literature, starting with the GRDPG.
The assumption in Theorem~\ref{thm:main:main} that $\kappastar = \Omega(d)$ is a natural one for the RDPG and GRDPG setting.
To see this, we first state Lemma~\ref{lem:main:cond}.
\begin{lemma} \label{lem:main:cond}
Assume that $\mP = \mX \mX^T$, where the row vectors $\vx_1, \vx_2, \ldots, \vx_n \in \R^d$ of $\mX$ are independent identically distributed random vectors and let $\mDelta = \E\left[\vx_1 \vx_1^T\right]$.
For $n$ sufficiently large, it holds with probability at least $1 -  2n^{-1}$ that
    \begin{equation*}
    \frac{\lambda_1(\mDelta)-\delta}{\lambda_d(\mDelta)+\delta}\leq \kappa(\mP) \leq \frac{\lambda_1(\mDelta)+\delta}{\lambda_d(\mDelta)-\delta},
    \end{equation*}
    where $\delta = 4\sqrt{\frac{\log d}{n}} + \frac{8\log d}{3n}$. 
\end{lemma}
\begin{proof}
Applying the definition of $\kappa$ and using basic properties of eigenvalues,
\begin{equation*}
\kappa(\mP) = \kappa(\mX \mX^T)
= \frac{\lambda_1\left(\mX^T \mX/n\right)}
	{\lambda_d\left(\mX^T \mX/n\right)}.
\end{equation*}
Since $\mP$ is a probability matrix, for any $i \in [n]$, we have $0 \le \vx_i^T\vx_i \leq 1$, and 
\begin{equation*} 
\left\|\vx_i \vx_i^T - \mDelta\right\|
\leq \left\|\vx_i \vx_i^T\right\| + \left\|\mDelta\right\|
\leq \left\|\vx_i\right\|_2^2 + \E \| \vx_i \|_2^2
\leq 2. 
\end{equation*}
Similarly, we also have 
    $
    \left\|\E\left[\left(\vx_i \vx_i^T - \mDelta \right)\left(\vx_i \vx_i^T - \mDelta \right)\right]\right\| \leq 4.
    $
Therefore, by a matrix version of Bernstein's inequality \citep[see Corollary 3.3 in][]{chen2021spectral}, with probability at least $1 - 2 n^{-1}$, we have 
\begin{equation*}
\left\|\frac{1}{n} \mX^T\mX - \mDelta\right\|
= \left\|\frac{1}{n}\sum_{i=1}^n \left(\vx_i\vx_i^T - \mDelta\right)\right\|
\leq 4\sqrt{\frac{\log d}{n}} + \frac{8\log d}{3n}. 
\end{equation*}
Hence, by Weyl's inequality, it follows that with probability at least $1 - 2n^{-1}$, 
\begin{equation*}
    \left|\lambda_1(\mDelta) - \lambda_1\left(\frac{1}{n} \mX^T\mX\right)\right| \leq \delta
~~~\text{ and }~~~
    \left|\lambda_d(\mDelta) - \lambda_d\left(\frac{1}{n} \mX^T\mX\right)\right| \leq \delta,
\end{equation*}
    where we set $\delta = 4\sqrt{\frac{\log d}{n}} + \frac{8\log d}{3n}$.
    Rearranging the inequalities completes the proof. 
\end{proof}

Put simply, Lemma~\ref{lem:main:cond} implies that under the RDPG, when $n$ is sufficiently large, we have $\kappa(\mP) \approx \kappa(\mDelta)$. Without loss of generality, we assume that $\vx_1, \vx_2, \ldots, \vx_n$ are sampled from a distribution whose support is a subset of $\bbB_{d}(1) \cap \R^d_{+}$, where $\bbB_{d}(1)$ is the unit ball in $\R^d$.
 Denote the covariance matrix as $\mSig = \E\left(\vx_1 - \vmu\right)\left(\vx_1 - \vmu\right)^T$. Notice that for any $\ell \in [d]$, $\vx_{1,\ell}^2 \leq \|\vx_{1}\|_2^2 \leq 1$, hence $\vx_{1,\ell}^2 \leq \vx_{1,\ell}$ and we have 
\begin{equation*} \begin{aligned}
    \vmu_\ell = \E  \vx_{1, \ell} \geq \E  \vx_{1, \ell}^2 = \mu_{\ell}^2 + \mSig_{\ell\ell}. 
\end{aligned} \end{equation*}
this implies that $\vmu_{\ell} \geq \mSig_{\ell\ell}$.  

If $\mSig = \gamma\mI_d$ for some $\gamma > 0$, then $\kappa(\mDelta) = \gamma^{-1}\vmu^T\vmu + 1 \geq \gamma d+1$, and hence $\kappa(\mP) = \Omega_{\bbP}(d)$.
One sufficient condition for this is that each element of $\vx_1$ be drawn i.i.d.
For example, if the entries of $\vx_{1}$ are generated i.i.d.~from the uniform distribution over $[0, 1/\sqrt{d}]$, then $\kappa(\mDelta) = 3d + 1$. 
As another example, if we sample $\vx_1, \vx_2, \ldots, \vx_n$ uniformly from $\bbB_{d}(1) \cap \R^d_{+}$, then one can show that $\kappa(\mDelta) = (2d + \pi - 2)/(\pi - 2) > d$.

The case for the GRDPG is more complicated, owing to replacing the RDPG's inner product with an indefinite inner product.
We first state Lemma~\ref{lem:main:grdpg-cond}, which allows us to relate the spectrum of the indefinite matrix $\mP =  \mX \mI_{p,q}\mX^T$ to the spectrum of the positive semidefinite $\mDelta = \E \vx_1 \vx_1^T$.

\begin{lemma} \label{lem:main:grdpg-cond}
    Assume that $\mP = \mX \mI_{p,q}\mX^T$, where the row vectors $\vx_1, \vx_2, \ldots, \vx_n \in \R^d$ of $\mX$ are i.i.d.~random vectors with second moment matrix $\mDelta = \E \vx_1 \vx_1^T$.
If there exists $0 < \delta < 1$ such that
    \begin{equation}
    \label{eq:cov-concen}
        \left\|\mX^T\mX/n - \mDelta\right\| \leq \delta \|\mDelta\|,
    \end{equation}
    then for a suitably chosen constant $C > 0$, we have 
\begin{equation*}
    \frac{\lambda_1\left(\mI_{p,q}\mDelta\right)  - C\sqrt{\delta} \|\mDelta\|}{\lambda_d\left(\mI_{p,q}\mDelta\right) + C\sqrt{\delta} \|\mDelta\|} \leq \kappa(\mP) \leq \frac{\lambda_1\left(\mI_{p,q}\mDelta\right) + C\sqrt{\delta} \|\mDelta\|}{\lambda_d\left(\mI_{p,q}\mDelta\right) - C\sqrt{\delta} \|\mDelta\|}.
\end{equation*}
\end{lemma}
\begin{proof}
Since $\mDelta$ is a symmetric positive semidefinite matrix, its square root $\mDelta^{1/2}$ is well-defined, as is that of $\tmDelta = \mX^T \mX/n$.
Using basic spectral properties,
\begin{equation}  \label{eq:grdpgkappa}
\kappa(\mP) = \kappa(\mX\mI_{p,q} \mX^T)
= \left|\frac{\lambda_1\left(\mI_{p,q}\mX^T \mX/n\right)}{\lambda_d\left(\mI_{p,q}\mX^T \mX/n\right)}\right|
= \left|\frac{\lambda_1\left(\tmDelta^{1/2}\mI_{p,q}\tmDelta^{1/2}\right)}{\lambda_d\left(\tmDelta^{1/2}\mI_{p,q}\tmDelta^{1/2}\right)}\right|.
\end{equation}
Applying the triangle inequality and basic properties of the spectral norm, we have
\begin{equation*} \begin{aligned}
 \left\|\tmDelta^{1/2}\mI_{p,q}\tmDelta^{1/2}
	- \mDelta^{1/2}\mI_{p,q}\mDelta^{1/2}\right\|
&\leq \left\|\tmDelta^{1/2}\mI_{p,q}\left(\tmDelta^{1/2} - \mDelta^{1/2}\right)\right\| \\
&~~~~~~~~~~~~+ \left\|\mDelta^{1/2}\mI_{p,q}\left(\tmDelta^{1/2} - \mDelta^{1/2}\right)\right\|\\
&\leq \left(\left\|\tmDelta^{1/2}\right\| + \left\|\mDelta^{1/2}\right\|\right) \left\|\tmDelta^{1/2} - \mDelta^{1/2}\right\|\\
        &\leq 2 \left\|\mDelta^{1/2}\right\|\left\|\tmDelta^{1/2} - \mDelta^{1/2}\right\| + \left\|\tmDelta^{1/2} - \mDelta^{1/2}\right\|^2. 
\end{aligned} \end{equation*}
Since $\tmDelta$ and $\mDelta$ are both positive semidefinite matrices, by Theorem X.1.1 in \cite{Bhatia1997}, we have 
\begin{equation*}
    \left\|\tmDelta^{1/2} - \mDelta^{1/2}\right\| \leq \left\|\tmDelta - \mDelta\right\|^{1/2} \text{ and } \left\|\mDelta^{1/2}\right\| = \left\|\mDelta\right\|^{1/2}.
\end{equation*}
Therefore, using the fact that $\delta \in (0,1)$, we obtain
\begin{equation*} \begin{aligned}
\left\|\tmDelta^{1/2}\mI_{p,q}\tmDelta^{1/2}
	- \mDelta^{1/2}\mI_{p,q}\mDelta^{1/2}\right\|
&\leq 2\left\|\mDelta\right\|^{1/2}\left\|\tmDelta - \mDelta\right\|^{1/2}
	+ \left\|\tmDelta - \mDelta\right\|\\
    &\leq (2\sqrt{\delta} + \delta)\|\mDelta\|
	\leq 3\sqrt{\delta}\|\mDelta\|. 
\end{aligned} \end{equation*}
Applying Weyl's inequality, it follows that
\begin{equation*} \begin{aligned}
\left|\lambda_1\left(\tmDelta^{1/2}\mI_{p,q}\tmDelta^{1/2}\right)
-\lambda_1\left(\mDelta^{1/2}\mI_{p,q}\mDelta^{1/2}\right)\right|
&\leq C\sqrt{\delta} \|\mDelta\|
\end{aligned} \end{equation*}
and
\begin{equation*} \begin{aligned}
\left|\lambda_d\left(\tmDelta^{1/2}\mI_{p,q}\tmDelta^{1/2}\right)
- \lambda_d\left(\mDelta^{1/2}\mI_{p,q}\mDelta^{1/2}\right)\right| &\leq C\sqrt{\delta} \|\mDelta\|.
\end{aligned} \end{equation*}
Applying these two bounds to Equation~\eqref{eq:grdpgkappa}, it follows that
\begin{equation*} 
\kappa(\mP)
\geq \frac{\lambda_1\left(\mDelta^{1/2}\mI_{p,q}\mDelta^{1/2}\right) - C\sqrt{\delta} \|\mDelta\|}{\lambda_d\left(\mDelta^{1/2}\mI_{p,q}\mDelta^{1/2}\right) + C\sqrt{\delta} \|\mDelta\|}\\
= \frac{\lambda_1\left(\mI_{p,q}\mDelta\right) - C\sqrt{\delta} \|\mDelta\|}{\lambda_d\left(\mI_{p,q}\mDelta\right) + C\sqrt{\delta} \|\mDelta\|},
\end{equation*}
and
\begin{equation*}
    \kappa(\mP)
\leq 
 \frac{\lambda_1\left(\mI_{p,q}\mDelta\right) + C\sqrt{\delta} \|\mDelta\|}{\lambda_d\left(\mI_{p,q}\mDelta\right) - C\sqrt{\delta} \|\mDelta\|},
\end{equation*}
completing the proof.
\end{proof}

For many distributions, Equation~\eqref{eq:cov-concen} holds with high probability for small choices of $\delta$.
As an example, suppose that for some constant $K \geq 1$, $\|\vx_i\|_2 \leq K(\E \|\vx_i\|_2^2)^{1/2}$ almost surely. Then 
\begin{equation*}
\| \frac{ \mX^T\mX }{ n } - \mDelta \| \leq C \left(\sqrt{\frac{K^2 d(\log d+\log n)}{n}} + \frac{K^2 d(\log d + \log n)}{n}\right) \|\mDelta\|
\end{equation*}
holds with probability at least $1 - 2n^{-1}$. See Theorem 5.6.1 and Exercise 5.6.4 in \cite{vershynin2018high}. 

As another example, if the first $p$ entries of $\vx_1$ are independently drawn from the uniform distribution over the interval $[1/(2\sqrt{p}), 1/\sqrt{p}]$ and the last $q$ entries are independently drawn from the uniform distribution over the interval $[0, 1/(2\sqrt{q})]$, then one can show that $\kappa(\mI_{p,q}\mDelta) \geq 5d - \frac{1}{2}$ and we can show that $\|\vx_i\|_2 \leq 3(\E \|\vx_i\|_2^2)^{1/2}$ almost surely, so that $\kappa(\mP) = \Omega_{\bbP}(d)$.

On the other hand, if we treat $d$ as a constant with respect to $n$, then $\kappa(\mP) = O_{\bbP}(1)$ and Theorem~\ref{thm:main:main} implies the following corollary.

\begin{corollary} \label{cor:results:latent}
Under the GRDPG, with latent dimension $d$ fixed with respect to $n$, suppose that the latent position matrix $\mX \in \R^{n\times d}$ satisfies $2d \leq \kappa(\mX\mI_{p,q}\mX^T) = O(1)$ and $\lambda_d(\mX\mI_{p,q}\mX^T) \geq \lambda_{\star}$.
Then
\begin{equation*} 
\inf _{\hat{\mX}} \sup _{\mX \in \calX_n^{(p, q)}}
	\E \tilde{d}_{2, \infty}([\hat{\mX}],[\mX])
	\gtrsim \sqrt{\frac{\lambda_\star \wedge \log n}{n}}. 
    \end{equation*}
\end{corollary}

Under the RDPG, \cite{xie2019optimal} derived a similar minimax lower bound for estimation in Frobenius norm, rather than $(\tti)$-norm, under the setting where the latent dimension is a constant.
For the sake of comparison, we restate their lower bound using our notation.
\begin{theorem}[Theorem 2 in \cite{xie2019optimal}] \label{thm:results:xie2019}
Let $\mA \sim \RDPG(\mX)$ for some $\mX \in \calX_{n,d}$, where $d$ is a constant with respect to $n$.
Let $\hat{\mX}$ be an estimator of the latent position matrix $\mX$ satisfying $\|\hat{\mX}\|_{\mathrm{F}} \lesssim \sqrt{n}$ with probability one.
Then
\begin{equation*} 
\inf_{\hat{\mX}} \sup _{\mX \in \calX^d_n} \E \left\{\frac{1}{n} \inf _{\mW \in \bbO_d}\|\widehat{\mX}-\mX \mW \|_{\mathrm{F}}^2\right\} \gtrsim \frac{1}{n}.
\end{equation*}
\end{theorem}
Directly applying Theorem~\ref{thm:results:xie2019} in the RDPG setting and using the fact that
\begin{equation} \label{eq:results:tti-vs-frob}
    \|\mY\|_{\tti} \geq \frac{ \|\mY\|_F }{ \sqrt{n} }
\end{equation}
for any $\mY \in \R^{n\times d}$, we obtain a lower bound of $O(n^{-1/2})$.
This has a gap of order $\lambda_{\star}\wedge\sqrt{\log n}$ compared to our result in Corollary~\ref{cor:results:latent}.
Further, we note that the techniques used in \cite{xie2019optimal} are specialized to the RDPG, and it is not obvious how to adapt their strategy to the more general setting considered here.

\subsection{Singular Subspace Estimation}
\label{subsec:main:subspace}

For a matrix $\mP = \mU\mLambda\mU^T$, instead of estimating the latent positions, singular subspace estimation aims to estimate the matrix $\mU \in \R^{n\times d}$. There is a vast literature on singular subspace estimation, and we refer the interested reader to the recent survey by \cite{chen2021spectral}.
\cite{vu2013minimax} derives a minimax lower bound for subspace estimation for sparse high-dimensional principal component analysis (PCA), and \cite{cai2021optimal} provides a more general framework to establish lower bounds in structured PCA problems.
We note that PCA is distinct from the low-rank network models considered here, and that these two papers consider estimation in the Frobenius or spectral norm in the presence of Gaussian noise, while we are concerned with estimation under the $(\tti)$-norm with Bernoulli-distributed noise.
To the best of our knowledge, the prior work closest to the present manuscript is that by \cite{Zhou2021}, where the authors obtain minimax lower bounds for singular subspace estimation of random bipartite graphs.
A few existing works address minimax lower bounds for singular subspace estimation under the $(\tti)$-norm.
\cite{cai2021subspace} provides a lower bound under the $(\tti)$-norm for subspace recovery in an incomplete low-rank matrix setting.
Lower bounds can also be found in \cite{joshua2022estimating}, derived from lower bounds on the spectral norm.
Below, we discuss why such approaches result in lower bounds weaker than those proved in the present work.

As a corollary to Theorem~\ref{thm:main:main}, we also obtain a minimax lower bound for singular subspace estimation.
The proof uses the same construction as Theorem~\ref{thm:main:main}, and thus details are omitted.

\begin{corollary} \label{cor:main:subspace}
Under the same setup as Theorem \ref{thm:main:main}, we have
\begin{equation} \label{eq:main:subspace-minimax}
\inf_{\hat{\mU}} \sup_{\mU \in \Stiefel_d(\R^n)}
\E \min_{\mW \in \bbO_d} \left\| \hat{\mU}- \mU \mW \right\|_{\tti}
\gtrsim \sqrt{ \frac{\kappa_{\star}(\lambda_\star\wedge\log n)}
		{\lambda_\star n} }. 
\end{equation}
\end{corollary}
We note that the minimum in Equation~\eqref{eq:main:subspace-minimax} is taken over $\Od$ rather than $\Od\cap \Opq$, since our proof of Theorem~\ref{thm:main:main} only makes use of the fact that $\mW \in \Od$, while the restriction to $\Od\cap \Opq$ is necessary to ensure that our distance on $\calXtilde^{(p,q)}_n$ is well-defined.

We remark that lower bounds for subspace estimation derived from the Frobenius norm or the spectral norm cannot be optimal in the $(\tti)$-norm setting.
These lower bounds use the fact that for any $\mU \in \R^{n\times d}$,
\begin{equation} \label{eq:results:tti-vs-op}
    \|\mU\|_{\tti} \geq \frac{1}{\sqrt{n}} \|\mU\|
\end{equation}
Taking $\hat{\mU} = \mzero$ to be our estimator, we have
\begin{equation*} \begin{aligned}
    \inf_{\hat{\mU}} \sup_{\mU \in \Stiefel_d(\R^n)}
        \E \min_{\mW \in \bbO_d} \left\| \hat{\mU}- \mU \mW \right\| &\leq \sup_{\mU \in \Stiefel_d(\R^n)} \E \min_{\mW \in \bbO_d} \left\|\mU \mW \right\| = 1\\
\end{aligned} \end{equation*}
or 
\begin{equation*}
\begin{aligned}
    \inf_{\hat{\mU}} \sup_{\mU \in \Stiefel_d(\R^n)}
        \E \min_{\mW \in \bbO_d} \left\| \hat{\mU}- \mU \mW \right\|_F &\leq \sqrt{d},
\end{aligned}
\end{equation*}
where this second bound follows from Equation~\eqref{eq:results:tti-vs-frob}.
It follows that any lower bound on the $(\tti)$-norm minimax rate can be no larger than $O(\sqrt{d/n})$ if we derive it from the Frobenius norm or the spectral norm through Equation~\eqref{eq:results:tti-vs-frob} or Equation~\eqref{eq:results:tti-vs-op}, respectively.
Comparing this with Equation~\eqref{eq:main:subspace-minimax}, our lower bound in Corollary~\ref{cor:main:subspace} improves on this rate by a factor of order $\sqrt{(\lambdastar \wedge \log n) \kappastar / \lambdastar}$ if $d$ is bounded by a constant.  

\subsection{Upper bounds} \label{subsec:upper}

In order to see the tightness of our lower bounds in Theorem~\ref{thm:main:main} and Corollary~\ref{cor:main:subspace}, we now consider upper bounds on the $(\tti)$-norm estimation error in different asymptotic regimes.
Before doing so, we must introduce the concept of average node degree and sparsity of a network.

For a node in a network, its degree is defined as the number of edges connected to it. For a random network with $n$ nodes generated from a probability matrix $\mP$, the $i$-th node has an expected degree of $\sum_{j=1}^n \mP_{ij}$. We define the average node degree of a network as the expected degree of each nodes averaging over the entire network, which is given by $n^{-1}\sum_{i=1}^n\sum_{j=1}^n \mP_{ij}$. If the average node degree grows as $\Theta(n)$, we are in the dense network regime. Random networks generated by the GRDPG model are dense networks. In applications, networks are observed to be sparse: the average node degree grows as $o(n)$. To incorporate the sparse regime into the GRDPG model, we scale the probability matrix $\mP$ by a sparsity factor $\rho_n \in (0,1]$, so that the probability matrix becomes $\rho_n\mP$, and its average node degree grows as $\Theta(n \rho_n)$. When $\rho_n = 1$, we recover the dense regime. Allowing $\rho_n \to 0$ as $n \to \infty$ produces sparse networks.

For latent position estimation under the GRDPG model, Theorem 3 in \cite{Rubin_2017} established an upper bound on the estimation errors of the ASE under $(\tti)$-norm.
We restate this result here.

\begin{theorem}[Theorem 3 in \cite{Rubin_2017}]
\label{thm:main:upper}
There exists a universal constant $c>1$ and a matrix $\mW_{\star} \in \bbO_d \cap \Opq$ such that, provided the sparsity factor satisfies $n\rho_n = \omega\{\log^{4c} n\}$, 
\begin{equation*}
\left\|\hat{\mU}\hat{\mLambda}^{1/2} \mW_{\star}-\mU\mLambda^{1/2}\right\|_{\tti}
=O_{\bbP} \left(\frac{\log ^c n}{n^{1 / 2}}\right).
\end{equation*}
\end{theorem}
In the setting of Theorem~\ref{thm:main:upper}, the condition number of the probability matrix satisfies $\kappa = O(1)$ and $\lambda_d = \Omega( n \rho_n ) = \omega( \log  n )$.
Applying Theorem~\ref{thm:main:main}, the lower bound in Equation~\eqref{eq:main:minimax} implies that the minimax estimation rate should be $n^{-1/2} \log^{1/2} n$, which matches the upper bound up to a polylogarithmic factor.
This also suggests the near optimality of the ASE in the GRDPG model. Note that Theorem~\ref{thm:main:upper} also applies to the RDPG model since the latter is a special case of the GRDPG model. 

For singular subspace estimation of low-rank plus noise models like that in Definition~\ref{def:lrpn}, an upper bound for the estimation error of the truncated SVD estimator $\hat{\mU}$ is given by Theorem 4.2 in \cite{chen2021spectral}. Adapted to our setting, Theorem 4.2 in \cite{chen2021spectral} states that there exists a matrix $\mW_{\star} \in \Od$, such that 
\begin{equation} \label{eq:main:sub-upper}
    \left\|\hat{\mU} \mW_{\star}-\mU\right\|_{\tti} \lesssim \frac{\kappa \sqrt{\rho_n \mu }+ \sqrt{\rho_n \log n}}{\lambda_d},
\end{equation}
where $\mu= n \left\|\mU\right\|_{\tti}/d$ is the incoherence parameter of the probability matrix $\mP$.
Notice that we always have $\mu \geq 1$.
Under the GRDPG, both $\mu$ and $\kappa$ are bounded by constants, and $\lambda_1/n = O(\rho_n)$.
Hence, the lower bound in Equation~\ref{eq:main:subspace-minimax} ensured by Theorem~\ref{thm:main:main} also matches the upper bound in Equation~\eqref{eq:main:sub-upper} up to a constant.

More generally, by the Perron-Frobenius theorem, for any probability matrix $\mP$, we have
\begin{equation*}
\lambda_1 \geq \min_{i\in[n]} \sum_{j = 1}^n \mP_{ij},
\end{equation*}
Hence, if we assume that $\mP = \rho_n \mP_0$ for some probability matrix $\mP_0$ with entries strictly bounded between 0 and 1, then $\lambda_1 = \Theta(n\rho_n)$, and our lower bound in Equation~\eqref{eq:main:subspace-minimax} can be rewritten as $\Omega( \sqrt{\rho_n(\lambda_d\wedge\log n)}/\lambda_d )$.
In this setting, if we further assume that $\mu = O(\log n)$, the upper bound in Equation~\eqref{eq:main:sub-upper} becomes $ O( \sqrt{ \rho_n (\lambda_d \wedge \log n) }/\lambda_d ) $, and we see that there is a $O(\kappa)$ gap (up to log factors) between the upper bound derived by \cite{chen2021spectral} and our lower bound in Corollary~\ref{cor:main:subspace}.
We study this gap through simulations in Section~\ref{sec:exp} (see Figure~\ref{fig:kappa_log} and Table~\ref{table:kappa}).
Based on those experiments, we conjecture that the upper bound in \cite{chen2021spectral} can be improved to match our lower bound (up to logarithmic factors), but we leave further exploration of this point for future work.

\section{Experiments}
\label{sec:exp}

In this section, we compare our theoretical lower bounds from Section~\ref{sec:results} with empirical estimation performance obtained by the ASE which according to existing results (e.g., Theorem~\ref{thm:main:upper}), matches this lower bound up to logarithmic factors.
Recall that for a pair of estimates $(\hat{\mU}, \hat{\mLambda})$, the $(\tti)$-norm between it and the ground truth $(\mU_0, \mLambda_0)$ is given by
\begin{equation} \label{eq:exp:tti}
\min_{\mW \in \Od \cap \Opq} \|\hat{\mU} \hat{\mLambda}^{1/2} \mW - \mU_0 \mLambda_0^{1/2} \|_{\tti}.
\end{equation}
Finding the exact minimizer of Equation~\eqref{eq:exp:tti} is non-trivial.
Instead, we approximate it by first solving a similar Procrustes problem under the Frobenius norm,
\begin{equation} \label{eq:exp:frob}
\min_{\mW \in \Od \cap \Opq} \|\hat{\mU} \hat{\mLambda}^{1/2} \mW - \mU_0 \mLambda_0^{1/2} \|_{F}
\end{equation}
and then plugging in the minimizer to the $(\tti)$-distance.
In practice, the minimizer under the Frobenius norm provides a good approximation to the exact minimizer.
As a matter of fact, we note that the matrix $\mW^\star$ in Theorem \ref{thm:main:upper} is the same minimizer of the Procrustes problem under the Frobenius norm, and therefore, the same upper bound for latent position estimation error still holds when $\kappa = O(1)$.
For details, we refer the reader to the proof of Theorem 3 in \cite{Rubin_2017}.
In general, approximating the problem in Equation~\eqref{eq:exp:tti} with the minimizer of Equation~\eqref{eq:exp:frob} serves as a valid upper bound for Equation~\eqref{eq:exp:tti}, and if it matches the lower bound, Equation~\eqref{eq:exp:tti} will as well.

Recall from Section~\ref{subsec:upper} that when $\kappa = O(1)$, our minimax lower bounds in Theorem~\ref{thm:main:main} and Corollary~\ref{cor:main:subspace} match the corresponding upper bounds up to logarithmic factors.
On the other hand, when $\kappa = \omega(1)$, as discussed in Section~\ref{subsec:upper}, there is no matching upper bound to our lower bound.
Rather, the best upper bound of which we are aware has a $O(\kappa\sqrt{\mu/\log n})$ gap with our minimax lower bound.
In light of this, we consider two different asymptotic regimes, both under the sparse GRDPG as discussed in Section~\ref{subsec:upper}.
In the first, we fix $\kappa$ to be a constant and vary the growth rate of the sparsity factor $\rho_n$.
In the second, where $\kappa = \omega(1)$, we fix the sparsity $\rho_n$ to be a constant and vary the growth rate of $\kappa$.
To emphasize the dependence of $\kappa$ on $n$, we also write $\kappa$ as $\kappa_n$ below. 


In both asymptotic regimes, we consider networks generated from a GRDPG with latent position dimension $d = 3$, and signature $(p,q)=(2,1)$.
The probability matrix $\mP_0 \in [0, 1]^{n\times n}$ is set to be $\mP_0 = \rho_n \mU_0 \mLambda_0 \mU_0$, where $\mU_0 \in \R^{n\times d}$ is constructed according to Equation~\eqref{eq:part-ii-U0} with suitably chosen constants and 
\begin{equation*}
\mLambda_0 = \diag \left(\frac{n}{3}, \frac{n}{3\kappa_n}, -\frac{n}{3\kappa_n}\right).
\end{equation*}
We vary $n$ from $9,000$ to $20,000$ with a step size of $1000$. 
In the setting where $\kappa = O(1)$, we fix $\kappa_n = 6$ and vary
\begin{equation}\label{eq:rhovals}
\rho_n \in \left\{0.2, 20n^{-1/2}, 90n^{-2/3}, 190n^{-3/4}, 300n^{-4/5}, 400n^{-5/6}, 1800n^{-1}\right\},
\end{equation}
where the constants are chosen so that all the $\rho_n$ are approximately equal to $0.2$ when $n = 9000$. 
In the second setting, where $\kappa = \omega(1)$, we fix $\rho_n = 0.9$ and vary 
\begin{equation} \label{eq:kappavals}
\begin{aligned}
\kappa_n \in \Bigg\{ & \frac{1207}{500} n^{1 / 10},\frac{971}{1000} n^{1 / 5},\frac{391}{1000} n^{3 / 10},\frac{157}{1000} n^{2 / 5},  \\
&~~~~~~~~~\frac{63}{1000} n^{1 / 2},\frac{1}{40} n^{3 / 5}, \frac{1}{100} n^{7 / 10}, \frac{1}{250} n^{4 / 5}\Bigg\}.
\end{aligned} \end{equation}
 The constants here are chosen to satisfy that all $\kappa_n$ are approximately equal to $6$ when $n = 9000$.
For each combination of $(n, \rho_n, \kappa_n)$, we generate $240$ Monte Carlo trials when we keep $\kappa_n = 6$ and $200$ trials when we keep $\rho_n = 0.9$.
We approximate their latent position and subspace estimation errors as described by Algorithm~\ref{alg:general}. 

\begin{algorithm}
\caption{Simulation procedure for expected adjacency matrix $\mP_0 = \mU_0 \mLambda^{1/2}_0 \mI_{p,q} \mLambda^{1/2}_0 \mU_0$ with signature $(p,q)$, based on $M$ Monte Carlo trials. We assume access to a function $\operatorname{TopEig}( \mA, k )$ for obtaining the top $k$ eigenvalues and eigenvectors of a matrix $\mA$.}\label{alg:general}
\begin{algorithmic}
\Require $\mP_0 \in \R^{n\times n}, d=p+q, M$. 
\For{$1 \leq i \leq M$}
    \State Sample an adjacency matrix $\mA_i$ from $\mP_0$.
    \State $(\hat{\mU}_{i, p}, \hat{\mLambda}_{i, p}) \gets \operatorname{TopEig}( \mA_i, p );~~~(\hat{\mU}_{i, q}, \hat{\mLambda}_{i, q}) \gets \operatorname{TopEig}( -\mA_i, q )$ 
    \State $\hat{\mU}_i \gets (\hat{\mU}_{i, p}, \hat{\mU}_{i, q})$;~~~ $\hat{\mLambda}_i \gets \diag(\hat{\mLambda}_{i, p}, \hat{\mLambda}_{i, q})$ .  
    \State $\mW_{1i} \gets \arg \min _{\mW \in \Od \cap \Opq} \|\hat{\mU}_i \hat{\mLambda}_i^{1/2} \mW - \mU_0 \mLambda_0^{1/2} \|_F$, 
    \State $\mW_{2i} \gets \arg \min _{\mW \in \Od \cap \Opq} \|\hat{\mU}_i \mW - \mU_0\|_F$.
    \State $\ell_{1i} \gets \|\hat{\mU}_i \hat{\mLambda}_i^{1/2} \mW_{1i} - \mU_i \mLambda_i^{1/2}\|_{2, \infty}$; $\quad \ell_{2i} \gets \|\hat{\mU}_i \mW_{2i} - \mU_i\|_{2, \infty}.$ 
\EndFor
\State $\ell_{\text{latent}} \gets \frac{1}{M}\sum_{i=1}^M \ell_{1i}$; $\quad \ell_{\text{subspace}} \gets \frac{1}{M}\sum_{i=1}^M \ell_{2i}.$\\
\Return $\ell_{\text{latent}}$, $\ell_{\text{subspace}}$. 
\end{algorithmic}
\end{algorithm}

Figure~\ref{fig:rho_log} shows the results when we fix $\kappa_n = 6$ and vary $\rho_n$.
The left subplot shows the estimation errors for the latent positions as a function of the number of vertices $n$.
We see that the lines by and large overlap one another, indicating that the growth rate of $\rho_n$ has little effect on the latent position estimation error rate, in agreement with what our lower bounds suggest.
The right subplot shows the estimation error for subspace recovery, again as a function of the number of vertices $n$.
Examining the different lines in the plot, we see that as the growth rate of $\rho_n$ gets smaller, the estimation error has a slower convergence rate, as suggested by our lower bound.
Of course, our lower bounds make predictions about the precise slope these lines should have, a point we explore in more detail below (see Table~\ref{table:rho}and discussion thereof).

\begin{figure}
    \centering
    \includegraphics[width=\textwidth]{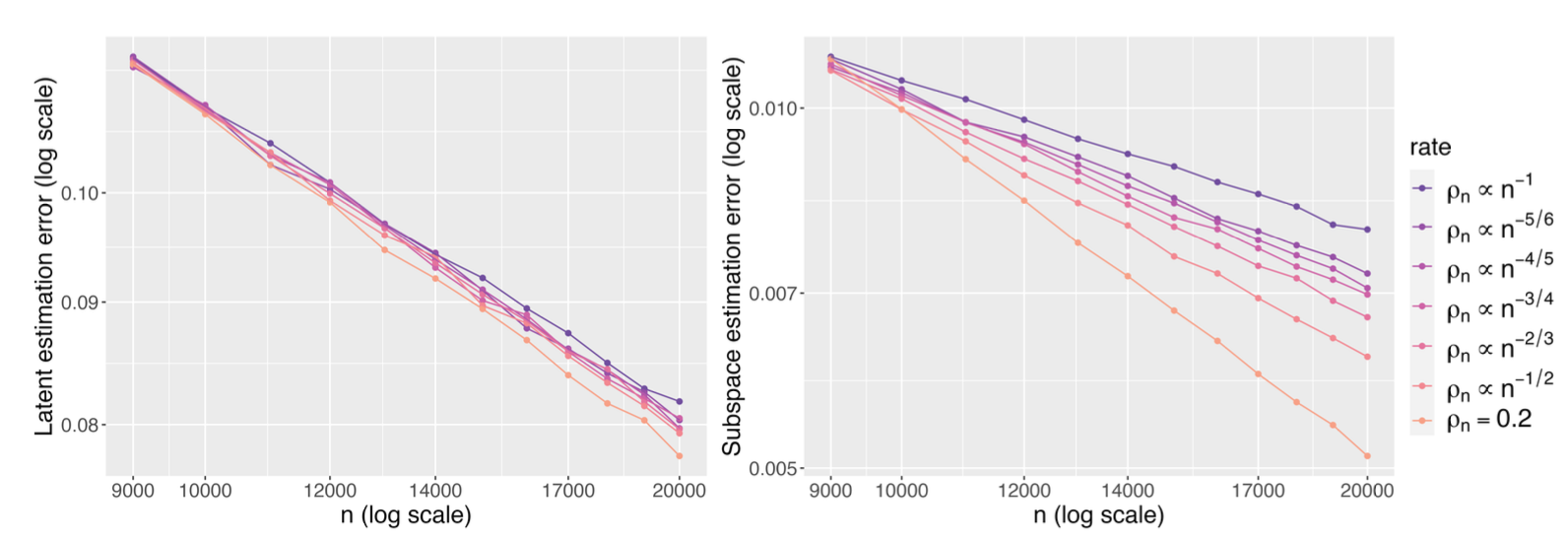}
    \caption{$\log$-$\log$ plots of latent positions estimation errors (left) and subspace estimation errors (right) as a function of the number of vertices $n$ when $\kappa_n = 6$. The $x$-axis displays the number of vertices $n$ in log scale and the $y$-axis displays the estimation error in log scale. Lines connect $(n,\rho_n)$ pairs that have the same scaling of $\rho_n$ with $n$. Lines with darker colors correspond to sparser networks while lighter colors correspond to denser networks with $\rho_n$ varying as in Equation~\eqref{eq:rhovals}.}
    \label{fig:rho_log}
\end{figure}

Figure~\ref{fig:kappa_log} shows the results of the same experiment when we fix $\rho_n = 6$ and vary $\kappa_n$, once again showing estimation error as a function of the number of vertices $n$.
The left subplot shows the estimation error for the latent positions while the right subplot shows the log-estimation error for the subspaces.
In both subplots, the estimation error has a slower convergence rate as the growth rate of $\kappa_n$ gets larger, again in agreement with our lower bounds in Theorem~\ref{thm:main:main} and Corollary~\ref{cor:main:subspace}.

\begin{figure}
    \centering
    \includegraphics[width=\textwidth]{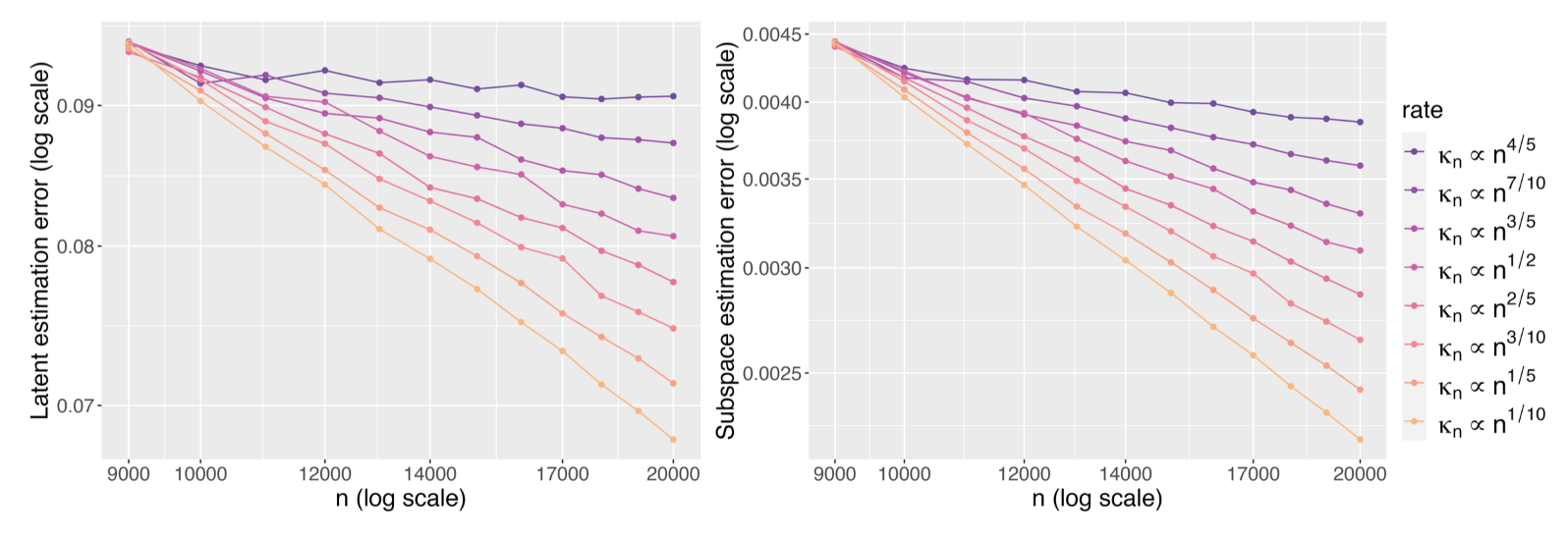}
    \caption{$\log$-$\log$ plots of latent positions estimation error (left) and subspace estimation error (right) as a function of the number of vertices $n$ when $\rho_n = 0.9$ and the condition number $\kappa_n$ varies. The $x$-axis displays number of vertices $n$ on a log scale and the $y$-axis displays the estimation error on log scale. Lines connect $(n,\kappa_n)$ pairs that have the same scaling of $\kappa_n$ with $n$. Lines with darker colors correspond to networks generated with larger $\kappa_n$ while lighter colors correspond to smaller $\kappa_n$ with $\kappa_n$ varies as in Equation~\eqref{eq:kappavals}. }
    \label{fig:kappa_log}
\end{figure}
 
The plots in Figures~\ref{fig:rho_log} and~\ref{fig:kappa_log} suggest a roughly log-log linear relationship between the estimation error and the number of vertices $n$.
Given a pair $(\rho_n, \kappa_n)$, if the estimation error is of order $n^{\alpha}$, then the log estimation error should be of order $\alpha\log n$.
Therefore, the slope of a line in the log-log plot provides an estimate of the exponent of the growth rate of the estimation error.
To better compare the growth rate obtained from the simulations against our lower bounds in Theorem~\ref{thm:main:main} and Corollary~\ref{cor:main:subspace}, regression the log estimation errors against $\log n$ for each $(\rho_n, \kappa_n)$-pair in our simulation.
That is, we fit a linear model to the points in each line in Figures~\ref{fig:rho_log} and~\ref{fig:kappa_log}.
The estimated slopes are listed in Tables~\ref{table:rho} and~\ref{table:kappa} in the columns labeled ``latent rate'' and ``subspace rate''.
We wish to compare these estimation rates against our theoretical lower bounds from Theorem~\ref{thm:main:main} and Corollary~\ref{cor:main:subspace}.
We note that these lower bounds include logarithmic factors, which have no bearing on the predicted slope of the lines in Figures~\ref{fig:rho_log} and~\ref{fig:kappa_log} when $n$ tends to infinity, but may lead to appreciably different lower bounds for finite $n$.
To account for this, we fit a second linear model, this time regression the logarithm of our minimax lower bound against $\log n$.
The estimated slopes are listed in Tables~\ref{table:rho} and~\ref{table:kappa} in the columns labeled ``latent lower'' and ``subspace lower''.
As an example, if we exclude the $\log n$ factor from our minimax lower bound in Theorem~\ref{thm:main:main}, then the ``latent lower'' column of Table~\ref{table:rho} would be all be equal to $-0.5$, since our lower bound becomes $\Omega(n^{-1/2})$.
In comparison, fitting a linear model to the lower bounds with logarithmic terms included yields a fitted slope of $-0.447$, in better agreement with the observed estimation rate.

Examining Tables~\ref{table:rho} and~\ref{table:kappa}, we see that the estimated error rates are close to the rates suggested by our lower bounds. 
We note, however, that for most $(n, \rho_n, \kappa_n)$ triples, the estimated error rates are slightly larger than predicted by the lower bounds. One reason for this might be that the ASE method is minimax optimal up to logarithmic factors. Since the minimax lower bounds are obtained for estimators that minimize the worst case risk, it might be the case that the ASE method is near optimal in some of the worst cases and the logarithmic factors in its rate will affect the estimated rate in finite sample cases, therefore making the estimated rate slightly larger. It is also possible that randomness in our simulations still has some significant effect on our estimated slopes in the two tables, though we doubt this is the case. 
All told, we do not necessarily expect the estimated error rates to be exactly those appearing in our minimax lower bounds. Nonetheless, our simulations do seem to suggest that our lower bounds are near optimal. 

As mentioned in the beginning of this section, one of our goals is to see how the estimation errors grow when $\kappa_n$ grows with $n$, since in this setting there is a gap between our minimax results and the best known upper bound on subspace recovery.
When we vary $\kappa_n$, we see in Table~\ref{table:kappa} that the estimation error rates hew closely to our lower bounds, rather than approaching the upper bound in Equation~\eqref{eq:main:sub-upper}, in agreement with our conjecture in Section~\ref{subsec:upper}.

\begin{table}[h!]
    \centering
    \begin{tabular}{ccccc}
    \toprule
     \bfseries $\rho_n$ & \multicolumn{1}{p{3cm}}{\centering \bfseries{latent \\ rate }} & \multicolumn{1}{p{3cm}}{\centering \bfseries{latent \\ lower }} &\multicolumn{1}{p{3cm}}{\centering \bfseries{subspace \\ rate }} &\multicolumn{1}{p{3cm}}{\centering \bfseries{subspace \\ lower }}\\
     \hline
     $0.2$            & $-0.465~(\pm0.009)$ & $-0.447$ & $-0.952~(\pm0.009)$ & $-0.947$\\
     $n^{-1 / 2}$  & $-0.443~(\pm0.009)$ & $-0.447$ & $-0.688~(\pm0.009)$ & $-0.697$\\
     $n^{-2 / 3}$  & $-0.435~(\pm0.009)$ & $-0.447$ & $-0.596~(\pm0.009)$ & $-0.614$ \\
     $n^{-3 / 4}$ & $-0.436~(\pm0.009)$ & $-0.447$ & $-0.559~(\pm0.009)$ & $-0.572$\\
     $n^{-4 / 5}$ & $-0.433~(\pm0.009)$ & $-0.447$ & $-0.529~(\pm0.009)$ & $-0.547$\\
     $n^{-5 / 6}$ & $-0.432~(\pm0.009)$ & $-0.447$ & $-0.510~(\pm0.009)$ & $-0.531$\\
     $n^{-1}$    & $-0.418~(\pm0.009)$ & $-0.447$ & $-0.420~(\pm0.009)$ & $-0.447$\\
    \bottomrule
    \end{tabular}
    \caption{Error rates for different choices of sparsity $\rho_n$. The ``latent rate'' and ``subspace rate'' columns show simulated estimation error rates for latent positions and subspaces using the ASE method when $\kappa_n$ is set to be 6 and we vary the growth rate of $\rho_n$. The values in the brackets correspond to a $95\%$ confidence interval.  The first column $\rho_n$ shows the growth rate of the sparsity factor $\rho_n$ up to constants, whose exact choices of constants are given in Equation~\eqref{eq:rhovals}. The third  column ``latent lower'' and the last columns ``subspace lower'' give the corresponding error rate lower bounds for latent position estimation and subspace estimation. }
    \label{table:rho}
\end{table}

\begin{table}[h!]
\centering
\begin{tabular}{ccccc}
\toprule
 \bfseries $\kappa_n$ & \multicolumn{1}{p{3cm}}{\centering \bfseries{latent \\ rate }} & \multicolumn{1}{p{3cm}}{\centering \bfseries{latent \\ lower }} &\multicolumn{1}{p{3cm}}{\centering \bfseries{subspace \\ rate }} &\multicolumn{1}{p{3cm}}{\centering \bfseries{subspace \\ lower }}\\
 \hline
 $n^{1/10}$ &    $-0.4072~(\pm0.0096)$ & $-0.3974$ & $-0.8551~(\pm0.0096)$ & $-0.8474$\\
 $n^{1 / 5}$ &   $-0.3471~(\pm0.0098)$ & $-0.3474$ & $-0.7439~(\pm0.0098)$ & $-0.7475$\\
 $ n^{3 / 10}$ & $-0.2974~(\pm0.0097)$ & $-0.2974$ & $-0.6436~(\pm0.0098)$ & $-0.6474$ \\
 $ n^{2 / 5}$ &  $-0.2416~(\pm0.0100)$ & $-0.2474$ & $-0.5390~(\pm0.0100)$ & $-0.5478$\\
 $ n^{1 / 2}$ &   $-0.2072~(\pm0.0095)$ & $-0.1974$ & $-0.4562~(\pm0.0096)$ & $-0.4486$\\
 $ n^{3 / 5}$ &      $-0.1524~(\pm0.0098)$ & $-0.1474$ & $-0.3567~(\pm0.0099)$ & $-0.3528$\\
 $n^{7/10}$ &      $-0.0957~(\pm0.0101)$ & $-0.0974$ & $-0.2513~(\pm0.0102)$ & $-0.2545$\\
 $n^{4/5}$ &       $-0.0471~(\pm0.0097)$ & $-0.0474$ & $-0.1536~(\pm0.0100)$ & $-0.1563$\\
\bottomrule
\end{tabular}
\caption{Error rates for different choices of condition number $\kappa_n$. The ``latent rate'' and ``subspace rate'' columns show simulated estimation error rates for latent positions and subspaces using the ASE method  when $\rho_n$ is set to be $0.9$ and we vary $\kappa_n$. The values in the brackets correspond to a $95\%$ confidence interval. The first column $\kappa_n$ shows the growth rate of $\kappa_n$ up to constants, whose exact choices of constants are given in Equation~\eqref{eq:kappavals}. The third  column ``latent lower'' and the last columns ``subspace lower'' give the corresponding error rate lower bounds for latent position estimation and subspace estimation. }
\label{table:kappa}
\end{table}

\section{Discussion} \label{sec:disco}

We have presented minimax lower bounds for estimation error of the latent positions and singular subspaces in the generalized random dot product graph and more general low-rank network models. We addressed the identifiability that arises due to the use of the indefinite inner product in the GRDPG model. To account for this nonidentifiability, we defined a distance on the equivalence classes of latent positions. This distance includes as special case a commonly used distance defined for the well-studied RDPG model.
To derive our minimax lower bounds, we constructed packing sets of singular subspaces for probability matrices by stacking Hadamard matrices. We divided our analysis into two parts based on different regimes of the condition number $\kappa = \lambda_1/\lambda_d$ of these probability matrices. 

When $\kappa = O(1)$, we proved minimax lower bounds that hold for sparse GRDPG models with a bounded latent position dimension $\kappa > 3d$. We note that this bound on $d$ can be relaxed to $\kappa > (1+\epsilon) d$ for any constant $\epsilon > 0$; we have used $3$ here for the sake of simplicity. The resulting lower bounds show that the adjacency spectral embedding \citep{SusTanFisPri2012} for estimating the latent positions is minimax optimal up to logarithmic factors. We provided examples to show that the assumption $\kappa > (1+\epsilon) d$ is not a stringent condition under both the GRDPG model and the RDPG model. 

In the regim where $\kappa = \omega(1)$, we established minimax lower bounds that also hold for growing latent dimension $d$, as long as $\kappa > 3d$. Here again, the constant $3$ can be relaxed to $1+\epsilon$ for any constant $\epsilon > 0$. Under this regime, we are not aware of any matching upper bound for latent position estimation or subspace estimation. The best upper bound currently known to us has a gap of $O\left(\kappa \sqrt{\mu/\log n}\right)$ compared to our bound. To evaluate how close our lower bounds are compared to the actual performance of the adjacency spectral embedding, we conducted simulations under different regimes of $\kappa$. The results are in agreement with our lower bounds.

In our future work, we would like to relax the assumption on $\kappa$. The main difficulty is that constructing packing sets for singular subspaces of probability matrices with small $\kappa$ is nontrivial, as it requires a careful combinatorial analysis of the positive and negative patterns of Hadamard matrices or other construction techniques. In addition, we would like to close the theoretical gap between the upper bounds and lower bounds when $\kappa = \omega(1)$. As suggested by our simulation results, we conjecture that in the regime where the condition number is allowed to grow, the existing upper bounds are not sharp. A tighter upper bound requires a more careful study of how noise perturbs singular subspaces and singular values of probability matrices. Lastly, low-rank matrices with a growing rank $d$ are a less studied regime, yet this provides a more realistic model for many real world networks \citep{sansford2023implications}. Future work should investigate the estimation error when $d \geq \kappa$. 

\bibliographystyle{apalike}
\bibliography{biblio}

\newpage

\appendix

\section{Proof of Lemma~\ref{lem:setup:equiv}}
\label{sec:apx:equiv}

\label{sec:apx:setup:equiv:proof}
\begin{proof}[Proof of Lemma~\ref{lem:setup:equiv}]
By definition of our equivalence relation, if $\mX \sim \mY$, then there exists $\mQ \in \Opq$ such that $\mY = \mX\mQ$, so that, expanding our definition of $\mP_{\mX}$,
 \begin{equation*}
 \mP_{\mY} = \mY \mI_{p,q} \mY^T
 = \mX \mQ \mI_{p,q} \mQ^T \mX^T
 = \mX \mI_{p,q} \mX^T
 = \mP_{\mX}.
\end{equation*}

Conversely, suppose that $\mP_{\mX} = \mP_{\mY}$.
    Write $\mX = (\mX_1, \mX_2)$, where $\mX_1 \in \R^{n \times p}$ has its $p$ columns corresponding to the ``positive'' part of $\mP_{\mX}$ and $\mX_2$ corresponds to the $q$ negative eigenvalues of $\mP_{\mX}$.
Writing $\mY = (\mY_1, \mY_2)$ similarly, since $\mP_\mX = \mP_\mY$, we have 
\begin{equation*}
    \mX_1\mX_1^{T} - \mX_2\mX_2^{T}
    = \mX\mI_{p,q}\mX^{T}
    = \mP_{\mX}
    = \mP_\mY
    = \mY\mI_{p,q}\mY^{T}
    = \mY_1\mY_1^{T} - \mY_2\mY_2^{T}
\end{equation*}
Rearranging, we have
\begin{equation*}
    [\mY_1~ \mX_2]\begin{bmatrix}
        \mY_1^T\\
        \mX_2^T
    \end{bmatrix} = [\mX_1~ \mY_2]\begin{bmatrix}
        \mX_1^T\\
        \mY_2^T
    \end{bmatrix} ,
\end{equation*}
and it follows that $[\mY_1, \mX_2]^T$ and $[\mX_1, \mY_2]^T$ have the same null space and thus $[\mY_1, \mX_2]$ and $[\mX_1, \mY_2]$ span the same column space.
    As a result, there exists a matrix $\mGamma \in \R^{d\times d}$ such that 
    \begin{equation} \label{eq:Gamma:mxeq}
    [\mY_1~ \mX_2] = [\mX_1~ \mY_2]\mGamma.
    \end{equation}
    Writing $\mGamma$ in block matrix form,
    \begin{equation*}
    \mGamma = \begin{bmatrix}
        \mGamma_{11} & \mGamma_{12}\\
        \mGamma_{21} & \mGamma_{22}\\
    \end{bmatrix}
    \end{equation*}
    where $\mGamma_{11} \in \R^{p\times p}$, $\mGamma_{12}, \mGamma_{21}^T \in \R^{p\times q}$ and $\mGamma_{22} \in \R^{q\times q}$.
    Rearranging Equation~\eqref{eq:Gamma:mxeq}, we have
    \begin{equation} \label{eq:setup:xy-rel}
    \begin{aligned}
        \mX_1 \mGamma_{11} &= \mY_1 - \mY_2 \mGamma_{21} \\
        \mX_2 - \mX_1 \mGamma_{12}&= \mY_2 \mGamma_{22} .
    \end{aligned}
    \end{equation}
    
    Writing Equation~\eqref{eq:setup:xy-rel} in matrix form, we have 
    \begin{equation} \label{eq:setup:mat-form}
        [\mX_1~ \mX_2] \begin{bmatrix}
        \mGamma_{11} & -\mGamma_{12}\\
        \mzero & \mI_q
    \end{bmatrix} = [\mY_1~ \mY_2] \begin{bmatrix}
        \mI_{p} & \mzero \\
        -\mGamma_{21} & \mGamma_{22}
    \end{bmatrix}.
    \end{equation}

    We note that $\mGamma_{22}$ is invertible, since otherwise there exists a nonzero vector $\va \in \R^q$ such that $\mGamma_{22} \va = 0$, from which it would follow that $\mX_2 \va - \mX_1 \mGamma_{12} \va = \mzero$, which contradicts the fact that $\mX$ has full column rank.
    Since $\mGamma_{22}$ is invertible, we can invert the matrix on the right-hand side, and rearranging Equation~\eqref{eq:setup:mat-form},
    it follows that there exists a matrix $\mQ \in \R^{d\times d}$ such that $\mY = \mX\mQ$.
To see that $\mQ \in \Opq$, note that since $\mP_\mX = \mP_{\mY}$, we have $\mX\left(\mI_{p, q} - \mQ \mI_{p,q} \mQ^T\right) \mX^T = \mzero$.
Since $\mX$ has full column rank, we must have $\mI_{p, q} - \mQ \mI_{p,q} \mQ^T  = 0$ and therefore $\mQ \in \Opq$.
\end{proof}

\section{Example: Equation~\eqref{eq:ind-dist} is not a distance}
\label{sec:apx:dist}

In Section~\ref{sec:setup}, we made a first attempt at defining a distance on the set $\calX^{(p,q)}_n$ according to Equation~\eqref{eq:ind-dist}, which we restate here for the sake of convenience:
\begin{equation*}
\inf_{\mQ_1, \mQ_2 \in \Opq} \left\|\mX \mQ_1 - \mY \mQ_2\right\|_{\tti}.
\end{equation*}
We stated in the text that this quantity fails to be a distance.
We illustrate that point here by constructing a triple of points in $\calX^{(1,1)}_2$ for which the triangle inequality appears to fail.

We have $n=2$, $p=1$ and $q=1$. By Proposition 6.1 and 6.2 in \cite{gallier2020differential}, any $\mQ \in \bbO_{1,1}$ is of the form 
\begin{equation*}
\mQ(\alpha) \mGamma = \begin{bmatrix}
    \cosh{\alpha} & \sinh{\alpha}\\
    \sinh{\alpha} & \cosh{\alpha}\\
\end{bmatrix} \mGamma,
\end{equation*}
where $\alpha \in \R$ and $\mGamma$ is one of the matrices
\begin{equation*}
\begin{bmatrix}
    1 & 0 \\
    0 & 1
\end{bmatrix}, 
\begin{bmatrix}
    -1 & 0  \\
    0 & 1
\end{bmatrix},
\begin{bmatrix}
    1 & 0 \\
    0 & -1
\end{bmatrix}
\text{ or }
\begin{bmatrix}
    -1 & 0 \\
    0 & -1
\end{bmatrix}.
\end{equation*}
For $\mX, \mY \in \R^{2\times 2}$, write
\begin{equation*}
\mX = \begin{bmatrix}
    x_{11} & x_{12} \\
    x_{21} & x_{22}
\end{bmatrix}, \mY = \begin{bmatrix}
    y_{11} & y_{12} \\
    y_{21} & y_{22}
\end{bmatrix},
\end{equation*}
observe that we have
\begin{equation*} \begin{aligned}
\inf_{\mQ_1, \mQ_2} \|\mX\mQ_1 - \mY\mQ_2\|_{\tti}
&= \inf_{\alpha_1, \alpha_2, \mGamma_1, \mGamma_2}
	\left\|\mX\mQ(\alpha_1)\mGamma_1
		- \mY\mQ(\alpha_2)\mGamma_2\right\|_{\tti} \\
&= \inf_{\alpha_1, \alpha_2, \mGamma_1, \mGamma_2}
  \left\|\mX\mQ(\alpha_1)- \mY\mQ(\alpha_2)\mGamma_2\mGamma_1 \right\|_{\tti} \\
&= \inf_{\alpha_1, \alpha_2, \mGamma}
	\left\|\mX\mQ(\alpha_1)- \mY\mQ(\alpha_2)\mGamma \right\|_{\tti}. \\
\end{aligned} \end{equation*}
Since
\begin{equation*} \begin{aligned}
    \begin{bmatrix}
    \cosh{\alpha} & \sinh{\alpha}\\
    \sinh{\alpha} & \cosh{\alpha}\\
    \end{bmatrix} 
    \begin{bmatrix}
    1 & 0 \\
    0 & -1 \\
    \end{bmatrix} &=  
    \begin{bmatrix}
    \cosh{\alpha} & -\sinh{\alpha}\\
    \sinh{\alpha} & -\cosh{\alpha}\\
    \end{bmatrix}
= \begin{bmatrix}
    \cosh(-\alpha) & \sinh(-\alpha)\\
    -\sinh(-\alpha) & -\cosh(-\alpha)\\
    \end{bmatrix}\\
&= \begin{bmatrix}
    1 & 0 \\
    0 & -1 \\
    \end{bmatrix}
    \begin{bmatrix}
    \cosh(-\alpha) & \sinh(-\alpha)\\
    \sinh(-\alpha) & \cosh(-\alpha)\\
    \end{bmatrix}
\end{aligned} \end{equation*}
and a similar commutative property holds for $-\mI_{1,1}$, we have 
\begin{equation} \label{eq:XQYGQ}
    \inf_{\mQ_1, \mQ_2} \|\mX\mQ_1 - \mY\mQ_2\|_{\tti} = \inf_{\alpha_1, \alpha_2, \mGamma} \left\|\mX\mQ(\alpha_1)- \mY\mGamma \mQ(\alpha_2)\right\|_{\tti}.
\end{equation}

Suppose that the first columns of $\mX$ and $\mY$ are strictly positive and that 
\begin{equation*}
    x_{i1}^2 - x_{i2}^2,  y_{i1}^2 - y_{i2}^2> 0
    ~~~\text{ for }~~~i=1,2.
\end{equation*}
Then neither $\mGamma = -\mI_{1,1}$ nor $\mGamma = -\mI_{2}$ will be the minimizer of the quantity on the right-hand side of Equation~\ref{eq:XQYGQ}. To see this, we notice that $(\mX\mQ(\alpha_1))_{i1} > 0$ and $(\mY\mQ(\alpha_2))_{i1} > 0$ hold for $i=1,2$ and any $\alpha_1, \alpha_2 \in \R$. It follows that for $\mGamma = -\mI_{1,1}$ or $-\mI_2$, $(\mY\mGamma\mQ(\alpha_2))_{i1}$ is always strictly negative, and we always have $\left\|\mX\mQ(\alpha_1)- \mY\mGamma \mQ(\alpha_2)\right\|_{\tti} > \left\|\mX\mQ(\alpha_1)- \mY\mGamma(-\mI_{1,1}) \mQ(\alpha_2)\right\|_{\tti} $ by flipping the term $(\mY\mGamma\mQ(\alpha_2))_{i1}$ to be strictly positive. Therefore, we need only to consider
\begin{equation}
\label{eq:dist:f}
\begin{aligned}
    f(\mX, \mY) &= \min\left\{\inf_{\alpha_1, \alpha_2} \left\|\mX\mQ(\alpha_1)- \mY\mQ(\alpha_2)\right\|_{\tti}, \inf_{\alpha_1, \alpha_2}\left\|\mX\mQ(\alpha_1)- \mY\mI_{1,1}\mQ(\alpha_2)\right\|_{\tti}\right\}\\
    &= \min\{\inf_{\alpha_1, \alpha_2}g(\mX, \mY, \alpha_1, \alpha_2), \inf_{\alpha_1, \alpha_2}g(\mX, \mY\mI_{1,1}, \alpha_1, \alpha_2)\},
\end{aligned}
\end{equation}
where we define 
\begin{equation*}
g(\mX, \mY, \alpha_1, \alpha_2) = \left\|\mX\mQ(\alpha_1)- \mY\mQ(\alpha_2)\right\|_{\tti}.
\end{equation*}

Now consider three matrices
\begin{equation*}
\mX = \begin{bmatrix}
    1.9 & 1.2\\
    4 & -3.8
\end{bmatrix}, 
\mY = \begin{bmatrix}
    12.7 & -9.8\\
    4.1 & -0.9
\end{bmatrix}, 
\text{ and }
\mZ = \begin{bmatrix}
    0.03 & -0.02\\
    2.3 & -1.9
\end{bmatrix}.
\end{equation*}
Noting that if we divide $\mX$, $\mY$ and $\mZ$ by a sufficiently large constant $C$, then they become valid latent positions for probability matrices.
Hence, if the triangle inequality does not hold for $\mX$, $\mY$ and $\mZ$, then Equation~\eqref{eq:ind-dist} is also not a valid distance when we restrict the matrices in its arguments to be latent positions of probability matrices.

For our choice of $\mX$, $\mY$ and $\mZ$, we use gradient descent to find the approximate values of $f(\mX, \mY), f(\mX, \mZ)$ and $f(\mY, \mZ)$.
We find that $f(\mX, \mY) \approx 2.7324, f(\mX, \mZ) \approx 1.2291, f(\mY, \mZ) \approx 7.8288$ and from this approximation, we have $f(\mX, \mY) + f(\mX, \mZ) < f(\mY, \mZ)$.
Finding the value of $f(\mX, \mY)$ turns out to be a nonconvex optimization problem, and we have no guarantee of finding the global minimum with gradient descent methods.
To better understand the landscape of the optimization problem, we provide contour plots of $g(\mX, \mY, \alpha_1, \alpha_2)$, $g(\mX, \mY\mI_{1,1}, \alpha_1, \alpha_2)$, $g(\mX, \mZ, \alpha_1, \alpha_2)$, $g(\mX, \mZ\mI_{1,1}, \alpha_1, \alpha_2)$, $g(\mZ, \mY, \alpha_1, \alpha_2)$, and $g(\mZ, \mY\mI_{1,1} \alpha_1, \alpha_2)$ as functions of $\alpha_1$ and $\alpha_2$ in Figures~\ref{fig:xyplot},~\ref{fig:xzplot} and~\ref{fig:yzplot}.

\begin{figure}
    \centering
    \begin{subfigure}[t]{0.45\textwidth}
        \centering
        \includegraphics[width=\textwidth]{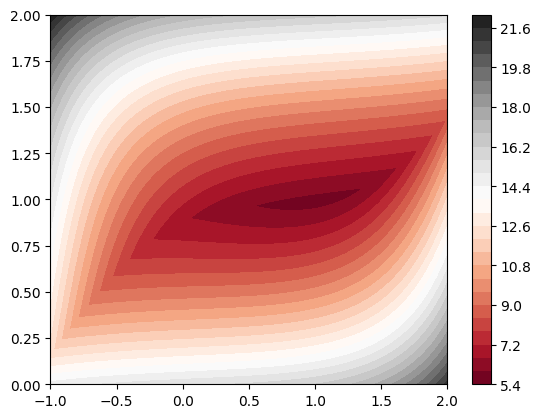}
    \end{subfigure}%
    ~ 
    \begin{subfigure}[t]{0.45\textwidth}
        \centering
        \includegraphics[width=\textwidth]{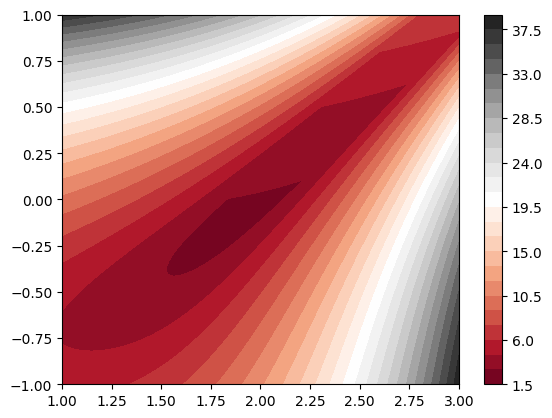}
    \end{subfigure}
    \caption{The left subplot shows a contour plot of $\|\mX\mQ(\alpha_1) - \mY\mQ(\alpha_2)\|_{\tti}$ as a function of $\alpha_1$ and $\alpha_2$, and the right subplot shows a contour plot of $\|\mX\mQ(\alpha_1) - \mY\mI_{1,1}\mQ(\alpha_2)\|_{\tti}$. Regions with a darker red color corresponds to smaller function values, and darker black color corresponds to larger function values. The $x$-axis corresponds to $\alpha_1$ and the $y$-axis corresponds to $\alpha_2$.}
    \label{fig:xyplot}
\end{figure}

\begin{figure}
    \centering
    \begin{subfigure}[t]{0.45\textwidth}
        \centering
        \includegraphics[width=\textwidth]{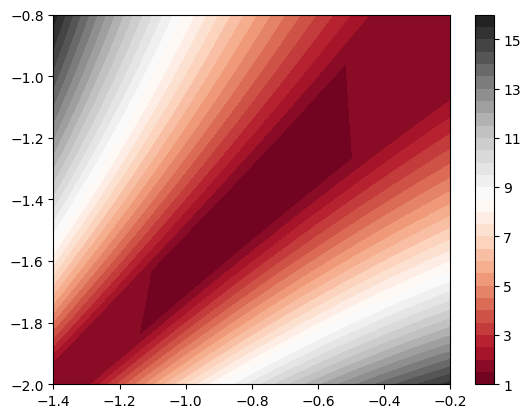}
    \end{subfigure}%
    ~ 
    \begin{subfigure}[t]{0.45\textwidth}
        \centering
        \includegraphics[width=\textwidth]{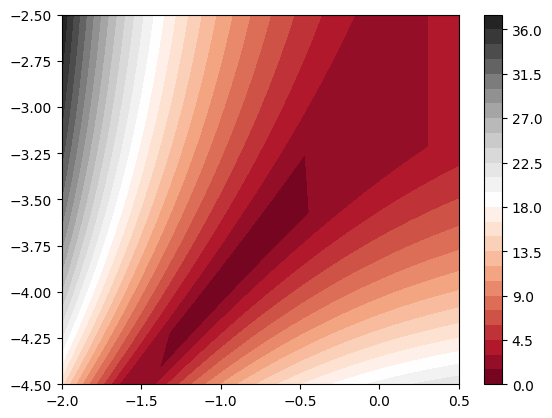}
    \end{subfigure}
    \caption{The left subplot shows a contour plot of $\|\mX\mQ(\alpha_1) - \mZ\mQ(\alpha_2)\|_{\tti}$ as a function of $\alpha_1$ and $\alpha_2$, and the right subplot shows a contour plot of $\|\mX\mQ(\alpha_1) - \mZ\mI_{1,1}\mQ(\alpha_2)\|_{\tti}$. Regions with a darker red color corresponds to smaller function values, and darker black color corresponds to larger function values. The $x$-axis corresponds to $\alpha_1$ and the $y$-axis corresponds to $\alpha_2$.}
    \label{fig:xzplot}
\end{figure}

\begin{figure}
    \centering
    \begin{subfigure}[t]{0.45\textwidth}
        \centering
        \includegraphics[width=\textwidth]{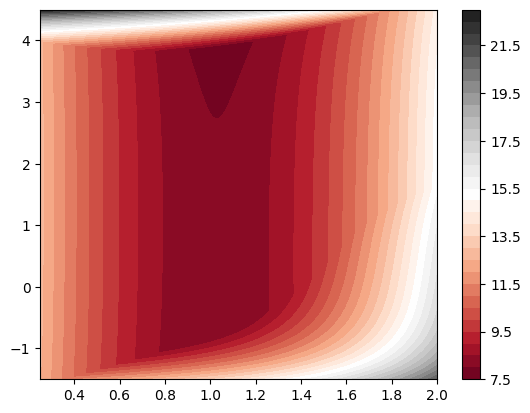}
    \end{subfigure}%
    ~ 
    \begin{subfigure}[t]{0.45\textwidth}
        \centering
        \includegraphics[width=\textwidth]{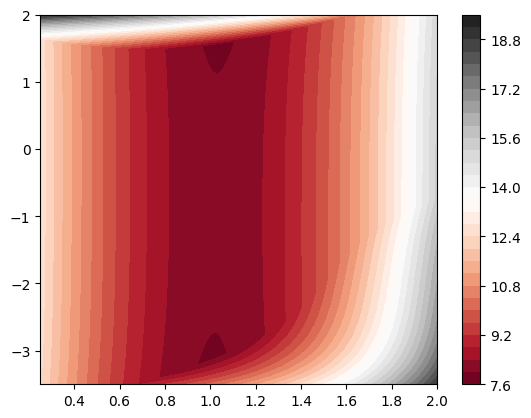}
    \end{subfigure}
    \caption{The left subplot shows a contour plot of $\|\mY\mQ(\alpha_1) - \mZ\mQ(\alpha_2)\|_{\tti}$ as a function of $\alpha_1$ and $\alpha_2$, and the right subplot shows a contour plot of $\|\mY\mQ(\alpha_1) - \mZ\mI_{1,1}\mQ(\alpha_2)\|_{\tti}$. Regions with a darker red color corresponds to smaller function values, and darker black color corresponds to larger function values. The $x$-axis corresponds to $\alpha_1$ and the $y$-axis corresponds to $\alpha_2$.}
    \label{fig:yzplot}
\end{figure}

The contour plots in Figures~\ref{fig:xyplot} through \ref{fig:yzplot} suggest that the global minimizers of all six of these functions lie in bounded regions.
We provide an intuitive argument to show this.
Noting that 
\begin{equation*} \begin{aligned}
    \begin{bmatrix}
    \cosh{\alpha} & \sinh{\alpha} \\
    \sinh{\alpha} & \cosh{\alpha} 
\end{bmatrix} &= 
\begin{bmatrix}
    \frac{1}{\sqrt{2}} & \frac{1}{\sqrt{2}} \\
    \frac{1}{\sqrt{2}} & -\frac{1}{\sqrt{2}}
\end{bmatrix}
\begin{bmatrix}
    e^{\alpha} & 0 \\
    0 & e^{-\alpha}
\end{bmatrix}
\begin{bmatrix}
    \frac{1}{\sqrt{2}} & \frac{1}{\sqrt{2}} \\
    \frac{1}{\sqrt{2}} & -\frac{1}{\sqrt{2}}
\end{bmatrix}\\
&=: \mW^* \begin{bmatrix}
    e^{\alpha} & 0 \\
    0 & e^{-\alpha}
\end{bmatrix} \mW^*,
\end{aligned} \end{equation*}
we denote $\Tilde{\mX} := \mX\mW^*$ and $\Tilde{\mY} := \mY\mW^*$ and have 
\begin{equation*} \begin{aligned}
g(\mX, \mY, \alpha_1, \alpha_2) &= \left\|\Tilde{\mX} \begin{bmatrix}
    e^{\alpha_1} & 0\\
    0 & e^{-\alpha_1}
\end{bmatrix} - \Tilde{\mY}\begin{bmatrix}
    e^{\alpha_2} & 0\\
    0 & e^{-\alpha_2}
\end{bmatrix}\right\|_{\tti}.
\end{aligned} \end{equation*}
Letting $t_1 = e^{\alpha_1}$ and $t_2 = e^{\alpha_2}$, it follows that
\begin{equation*}
[ g(\mX, \mY, \alpha_1, \alpha_2) ]^2
= \max_{i=1, 2}\left\{(t_1\Tilde{x}_{i1} - t_2\Tilde{y}_{i1})^2+\Big(\frac{1}{t_1}\Tilde{x}_{i2} - \frac{1}{t_2}\Tilde{y}_{i2}\Big)^2\right\}.
\end{equation*}
Clearly, if $(\alpha_1, \alpha_2) \to (-\infty, \infty)$ or $(\alpha_1, \alpha_2) \to (\infty, -\infty)$, then $g(\mX, \mY, \alpha_1, \alpha_2) \to \infty$.
On the other hand, we have 
\begin{equation}
\label{eq:g:lower1}
[ g(\mX, \mY, \alpha_1, \alpha_2) ]^2 \geq \max_{i=1, 2}\left\{(t_1\Tilde{x}_{i1} - t_2\Tilde{y}_{i1})^2\right\},
\end{equation}
and
\begin{equation}
\label{eq:g:lower2}
[ g(\mX, \mY, \alpha_1, \alpha_2) ]^2 \geq \max_{i=1, 2}\left\{\Big(\frac{1}{t_1}\Tilde{x}_{i2} - \frac{1}{t_2}\Tilde{y}_{i2}\Big)^2\right\}.
\end{equation}
Therefore, if $\tilde{x}_{11}\tilde{y}_{21} \neq \tilde{x}_{21}\tilde{y}_{11}$, then $g(\mX, \mY, \alpha_1, \alpha_2) \to \infty$ when $(\alpha_1, \alpha_2) \to (\infty, \infty)$.
Similarly, if $\tilde{x}_{12}\tilde{y}_{22} \neq \tilde{x}_{22}\tilde{y}_{12}$, then  $g(\mX, \mY, \alpha_1, \alpha_2) \to \infty$ when $(\alpha_1, \alpha_2) \to (-\infty, -\infty)$.
Thus, combining all the cases, it follows that $g(\mX, \mY, \alpha_1, \alpha_2)$ is coercive if 
\begin{equation} \label{eq:coer-cond}
    \tilde{x}_{1i}\tilde{y}_{2i} \neq \tilde{x}_{2i}\tilde{y}_{1i} \quad \text{ for } i=1,2.
\end{equation}
One can verify that $\mX, \mY, \mZ, \mY\mI_{1,1}$ and $\mZ\mI_{1,1}$ satisfy the condition in Equation~\eqref{eq:coer-cond}.
Hence, we indeed have that the global minimizers of all the functions plotting in Figures~\ref{fig:xyplot} through~\ref{fig:yzplot} lie in bounded regions.

For a more careful characterization of these bounded regions, we use polar coordinates to represent $\tilde{\mX}, \tilde{\mY}$ and $(t_1, t_2)$. For $i, j \in \{1, 2\}$, define
\begin{equation*} \begin{aligned}
(\tilde{x}_{ij}, \tilde{y}_{ij}) &= r_{ij}(\cos \theta_{ij}, \sin \theta_{ij}) \\
(t_1, t_2) &= s_1(\cos \psi, \sin \psi) \\
(1/t_1, 1/t_2) &= s_2(\cos \phi, \sin \phi)
\end{aligned} \end{equation*}
for $\psi, \phi \in [0, \pi/2]$.
Noting that by our choice of $\mX$, since $\mX_{i1} > |\mX_{i2}|$ holds for $i = 1, 2$, we have $\tilde\mX = \mX\mW^* > 0$ holds elementwise.
The same holds for $\mY$, $\mZ$, $\mY\mI_{1,1}$ and $\mZ\mI_{1,1}$.
Hence, we can restrict $\theta_{ij} \in [0, \pi/2]$.
Assuming that $r_{ij} \geq l_j > 0$ for $i, j \in \{1,2\}$, we have  
\begin{equation*} \begin{aligned}
    \max_{i=1, 2}\left\{(t_1\Tilde{x}_{i1} - t_2\Tilde{y}_{i1})^2\right\}
    &= s_1^2 \max
    \left\{r_{11}^2 \cos^2(\theta_{11} - \psi),
    ~r_{21}^2 \cos^2(\theta_{21} - \psi)
    \right\}\\ 
    &\geq s_1^2 l_1^2 \max\{\cos^2(\theta_{11} - \psi), ~\cos^2(\theta_{21} - \psi)\}.
\end{aligned} \end{equation*}
Note that for $\alpha, \beta, \psi \in [0, \pi/2]$, if $\alpha \neq \beta$, then the minimum of 
$$
\max\{\cos^2{(\alpha - \psi)}, \cos^2{(\beta - \psi)}\} = \frac{1}{2} +\frac{1}{2}\max\{-\cos{(2\psi-2\alpha)}, -\cos{(2\psi-2\beta)}\}
$$
occurs at $\psi = (\beta+\alpha)/2$.
Therefore, we have 
$$
\max_{i=1, 2}\left\{(t_1\Tilde{x}_{i1} - t_2\Tilde{y}_{i1})^2\right\} \geq \frac{s_1^2 l_1^2}{2} \left(1 - \cos (\theta_{11} - \theta_{21}) \right),
$$
and similarly,
$$
\begin{aligned}
    \max_{i=1, 2}\left\{(\Tilde{x}_{i2}/t_1 - \Tilde{y}_{i2}/t_2)^2\right\} &\geq \frac{s_2^2 l_2^2}{2} \left(1 - \cos (\theta_{12} - \theta_{22}) \right).
\end{aligned}
$$
Therefore, it follows from Equations~\eqref{eq:g:lower1} and~\eqref{eq:g:lower2} that 
\begin{equation}
\label{eq:g:lower}
    g(\mX, \mY, \alpha_1, \alpha_2)^2 \geq \min\Big\{\frac{s_1^2 l_1^2}{2} (1 - \cos (\theta_{11} - \theta_{21})), \frac{s_2^2 l_2^2}{2} (1 - \cos (\theta_{12} - \theta_{22}))\Big\}.
\end{equation}
Based on Equation~\eqref{eq:g:lower}, one can verify that as long as $s_i \geq 1000$ or $|\alpha_i| \geq 7$ for either $i=1$ or $i=2$, then each of $g(\mX, \mY, \alpha_1, \alpha_2)$, $g(\mX, \mY\mI_{1,1}, \alpha_1, \alpha_2)$, $g(\mX, \mZ, \alpha_1, \alpha_2)$, $g(\mX, \mZ\mI_{1,1}, \alpha_1, \alpha_2)$, $g(\mZ, \mY, \alpha_1, \alpha_2)$, and $g(\mZ, \mY\mI_{1,1} \alpha_1, \alpha_2)$ is all greater than $8$. 
Further providing bounds for $f(\mX, \mY), f(\mX, \mZ)$ and $f(\mY, \mZ)$ within $|\alpha_i| \leq 7$ for both $i = 1$ and $i=2$ would show that the triangle inequality does not hold.
Rather than providing exact bounds, we evaluate each function on a $1000$-by-$1000$ grid in $[-7, 7]\times[-7, 7]$, and the minima on this grid do not get lower than the values found by the results provided by gradient descent.
Combining the approximations provided by gradient descent and the contour plots in Figures~\ref{fig:xyplot} through \ref{fig:yzplot}, our results suggest that Equation~\eqref{eq:ind-dist} fails to obey the triangle inequality for certain triples of points, and hence is not a distance.

\section{Proof of Theorem~\ref{thm:main:main}}
\label{sec:apx:mainproof}

Here, we give a detailed proof of Theorem~\ref{thm:main:main}, drawing on a number of technical results that can be found in later sections of this appendix.
Our main tool is \cite{Tsybakov2009} Theorem 2.7, which we restate here for ease of reference.
\begin{theorem}[\cite{Tsybakov2009}, Theorem 2.7]
\label{thm:main:tsybakov}
Let $\Theta$ be a set of parameters endowed with a semi-distance $\delta$, let $M \geq 2$ and suppose that $\Theta$ contains elements $\theta_0, \theta_1, \ldots, \theta_M$ such that:
    \begin{itemize}
        \item[(i)] $\delta\left(\theta_j, \theta_k\right) \geq 2 s>0, \quad \forall 0 \leq j<k \leq M$
        \item[(ii)] Letting $P_0,P_1,\dots,P_M$ be probability measures associated to respective parameters $\theta_0,\theta_1,\dots,\theta_M$, it holds for all $j=1,2,\dots,M$ that $P_j \ll P_0$
        and
\begin{equation*}
\frac{1}{M} \sum_{j=1}^M \KL\left(P_j \| P_0\right) \leq \alpha \log M,
\end{equation*}
with $0<\alpha<1 / 8$.
\end{itemize}
Then we have
\begin{equation*}
\inf _{\hat{\theta}} \sup_{\theta \in \Theta} \E_\theta\left[ \delta(\hat{\theta}, \theta)\right] \geq c_\alpha s,
\end{equation*}
where $\inf_{\hat{\theta}}$ denotes the infimum over all estimators and $c_{\alpha}>0$ is a constant depending only on $\alpha$.
\end{theorem}

\begin{proof}[Proof of Theorem~\ref{thm:main:main}]
To apply Theorem~\ref{thm:main:tsybakov}, we must construct a collection of elements of $\calP(\kappastar,\lambdastar,p,q)$ whose pairwise distances as measured by $\dtilde$ are lower-bounded, but whose pairwise KL-divergences are close (i.e., pairs of these elements give rise to similar distributions over the set of $n$-vertex networks).
We break our proof into two separate cases, based on the growth rate of $\kappastar$.
These two different regimes require slightly different constructions, owing to the different spectral structures they imply.
We first consider the case where $\kappastar= O(1)$.

For latent dimension $d$, let $k_0$ be such that $2^{k_0-1} < d \le 2^{k_0}$, and define $m = \lfloor n/2^{k_0} \rfloor$
Lemma~\ref{lem:main:kappaconst:packing_distance}, proved in Section~\ref{subsec:apx:kappaconst:packing}, ensures the existence of a collection of $M = 2^{k_0}m = \Omega(n)$ matrices
\begin{equation} \label{eq:results:kappaconst:packingset}
\calU = \{\mU_i : i = 0,1,2,\dots, M\} \subset \Stiefel_d(\R^n)
\end{equation}
such that, with any $\lambda_1 \leq n/3$ and $\lambda_d = \lambda_1/\kappa$, for 
\begin{equation*}
\mLambda = \diag \left( \lambda_1, \lambda_1/\kappa,
        \dots,  \lambda_1/\kappa \right),
\end{equation*}
it holds for all $i \in [M] \cup \{0\}$ and all $j \in [M]$ not equal to $i$,
\begin{equation*}
    \min_{\mW \in \bbO_d \cap \Opq}
    \left\| \mU_i\mLambda^{1/2} - \mU_j\mLambda^{1/2} \mW \right\|_{\tti}
    \ge C \sqrt{\frac{\kappa (\log n \wedge \lambda_d)}{n}}
\end{equation*}
for a suitably-chosen constant $C > 0$.

Taking $\lambda_d = \lambdastar$ and $\lambda_1 = \kappastar\lambdastar$, our assumption that $3 \kappastar \lambdastar \le n$ implies
$\lambda_1 \leq n/3$ and $\kappa = \kappastar$.
Therefore, the pairs $(\mU, \mLambda)$, where $\mU \in \calU$, are indeed elements of $\calP(\kappastar,\lambdastar,p,q)$.
Thus, the set of matrices in Equation~\eqref{eq:results:kappaconst:packingset} are a $2s$-packing set of $ \calP(\kappastar, \lambdastar, p, q)$ under the $(\tti)$-norm, where $s = C\sqrt{(\log n \wedge \lambdastar) \kappastar / n }$ for suitably chosen constant $C > 0$.

Writing $\tmLambda = \mLambda^{1/2} \mI_{p,q} \mLambda^{1/2}$ for ease of notation and taking $\mP_i = \mU_i \tmLambda \mU_i^T$ for all $i \in [M]\cup\{0\}$, we note that $(\mU_i,\mLambda)$ induces a distribution over $n$-vertex networks via $\mP_i$.
Lemma~\ref{lem:main:kappaconst:KL_bound}, proved in Section~\ref{subsec:apx:kappaconst:KL}, upper bounds the KL divergences between these distributions over networks as
\begin{equation*}
\KL\left( \mP_i \| \mP_0 \right)
\le
\frac{1}{10} \log n~~~\text{ for all }~~~i\in [M]
\end{equation*}
for all suitably large $n$.
Averaging over $i \in [M]$,
\begin{equation*}
\frac{1}{M} \sum_{j=1}^M \KL( \mP_i \| \mP_0 )
\le \frac{1}{10} \log n.
\end{equation*}
Thus, applying Theorem~\ref{thm:main:tsybakov} with $s = C\sqrt{(\log n \wedge \lambdastar) \kappastar / n }$ and $\alpha = 1/10$, our result holds for the setting where $\kappastar = O(1)$.

In the setting where $\kappastar = \omega(1)$, we use a different construction, but our proof largely parallels the argument given above.
Let $\kappa=\kappastar$, set $\lambda_2=\lambda_3=\dots=\lambda_d = \lambdastar$ and $\lambda_1 = \lambdastar\kappastar$, by assumption, we have $\lambda_1 \leq n/3$. Define the matrix $\mLambda = \diag(\lambda_1,\lambda_2,\dots,\lambda_d) \in \R^{d \times d}$. 
By Lemma~\ref{lem:kappagrowing:packing}, proven in Section~\ref{subsec:apx:kappagrowing:packing}, there exists a collection
\begin{equation*}
\calU = \{ \mU_i : i=1,2,\dots, \lfloor n/2 \rfloor \} 
\subset \Stiefel_d(\R^n)
\end{equation*}
such that for any pair of indices $0 \le i < j \le \lfloor n/2 \rfloor$ and any $\zeta_d \le 1/\sqrt{640 d}$,
\begin{equation*}
\min_{\mW \in \Od \cap \Opq}
\left\|\mU_i \mLambda^{1/2} \mW - \mU_j \mLambda^{1/2}\right\|_{2, \infty}
\geq \frac{\zeta_d \sqrt{d-1}}{2}\sqrt{\frac{\kappa(\lambda_d \wedge \log n)}{n}}.
\end{equation*}
Since $\kappa = \kappastar$ and $\lambda_d = \lambdastar$ by construction,
the set $\calU$ constitutes a $2s$-packing set for $\calP(\kappastar,\lambdastar,p,q)$ with $s = C\sqrt{ (\log n\wedge \lambdastar) \kappastar / n}$ for $C>0$ chosen suitably small.

It remains for us to upper bound the KL divergence between the distributions induced by $\mLambda$ and the elements of $\calU$.
Recall our notation $\tmLambda = \mLambda^{1/2} \mI_{p,q} \mLambda^{1/2}$ and $\mP_i = \mU_i \tmLambda \mU_i^T$ for $i=0,1,2,\dots,\lfloor n/2 \rfloor$.
Applying Lemma~\ref{lem:kappagrowing:KL}, proven in Section~\ref{subsec:apx:kappagrowing:KL}, for any $i =1,2,\dots,\lfloor n/2 \rfloor$,
\begin{equation*}
\left\| \mP_i - \mP_0 \right\|_F^2
\le \frac{1}{80} \frac{ \lambda_1 (\lambda_d \wedge \log n)}{n}.
\end{equation*}
Lemma~\ref{lem:main:kappagrowing:P0exists}, also proven in Section~\ref{subsec:apx:kappagrowing:KL}, ensures that $\mP_0$ has entries bounded by
\begin{equation*}
    \frac{ \lambda_1 }{ 3n } \le \mP^{(0)}_{ij} \le \frac{2}{3}~~~\text{for all}~~~i,j.
\end{equation*}
Thus, applying Lemma~\ref{lem:Zhou2021}, for any $i=1,2,\dots,\lfloor n/2 \rfloor$,
\begin{equation*}
\KL( \mP_i \| \mP_0 )
\le
\frac{9 n}{ \lambda_1} \left\| \mP_i - \mP_0 \right\|_F^2
\le \frac{9}{80} \log n.
\end{equation*}
Averaging over our packing set,
\begin{equation*}
\frac{1}{\lfloor n/2 \rfloor}
\sum_{i=1}^{ \lfloor n/2 \rfloor} \KL( \mP_i \| \mP_0 )
\le \frac{9}{80} \log n.
\end{equation*}
Applying Theorem~\ref{thm:main:tsybakov} with $s = C\sqrt{ (\log n\wedge \lambdastar) \kappastar / n}$ and $\alpha = 9/80 < 1/8$ establishes our result in the regime where $\kappastar = \omega(1)$, completing the proof.
\end{proof}

\section{Theorem~\ref{thm:main:main}: Constant condition number}
\label{sec:apx:kappaconst}

Here, we prove Theorem~\ref{thm:main:main} in the regim where $\kappa = \lambda_1/\lambda_d = O(1)$.
Recall that we have $\mP = \mU\mLambda^{1/2}\mI_{p,q}\mLambda^{1/2}\mU^T$, where $\mLambda$ is a diagonal matrix with positive on-diagonal entries $\lambda_1 \geq \lambda_2 \geq \ldots \geq \lambda_d$.
By assumption in Theorem~\ref{thm:main:main}, $\kappa \geq 3d$, so the latent position dimension $d$ may be considered bounded throughout this section.

\subsection{Constructing a Packing Set}
\label{subsec:apx:kappaconst:packing}

We begin by constructing our collection of elements of $\calP(\kappastar,\lambdastar,p,q)$ and establishing a lowerbound on their pairwise distances under $\dtildetti$.
As we mentioned in Section \ref{sec:results}, our construction makes use of Hadamard matrices to construct a collection of matrices with orthonormal columns, which will correspond to the singular subspaces of a collection of probability matrices.
We will then show that this collection of singular subspaces, multiplied by a diagonal matrix of suitable eigenvalues, constitute the representatives of equivalence classes that form a packing set over $\calP(\kappastar,\lambdastar,p,q)$.
In Section~\ref{subsec:apx:kappaconst:KL}, we establish that the KL divergences of their associated probability matrices are suitably bounded.  

Recall that a Hadamard matrix of order $n$ is an $n$-by-$n$ matrix whose entries are drawn from $\{-1,1\}$ and whose rows are mutually orthogonal.
In particular, Hadamard matrices have the useful property that if $\mH$ is a Hadamard matrix of order $n$, then the matrix
\begin{equation*}
\begin{bmatrix}
\mH & \mH \\
\mH & -\mH
\end{bmatrix} \in \{-1,1\}^{2n \times 2n}
\end{equation*}
is a Hadamard matrix of order $2n$.
According to Sylvester's construction \citep{seberry2005on}, we can construct Hadamard matrices of order $2^k$ recursively by
\begin{equation*}
\mH_1 = \begin{bmatrix}
1
\end{bmatrix}, \quad
\mH_2 = \begin{bmatrix}
1 & 1 \\
1 & -1
\end{bmatrix},
\end{equation*}
and
\begin{equation*}
\mH_{2^{k+1}}=\begin{bmatrix}
\mH_{2^{k}} & \mH_{2^{k}} \\
\mH_{2^{k}} & -\mH_{2^{k}}
\end{bmatrix}
\end{equation*}
for any integer $k \geq 0$.
We write $\mH_n$ to denote a Hadamard matrix of order $n = 2^k$ for integer $k \ge 0$ constructed in this manner.

\begin{lemma} \label{lem:kappaconst:U0exists}
Under the conditions of Theorem~\ref{thm:main:main}, suppose that $\kappastar = O(1)$.
There exists a matrix $\mU_0 \in \R^{n \times d}$ with orthonormal columns such that
\begin{equation} \label{eq:U0:max}
    \max_{j\in[n], k\in[d]} \left|\mU^{(0)}_{jk}\right| \leq \frac{1}{\sqrt{n-r}},
\end{equation}
\begin{equation} \label{eq:U0:max-i}
    \frac{1}{\sqrt{n}} \leq \max_{k\in[d]} \left|\mU^{(0)}_{ik}\right| \leq \frac{1}{\sqrt{n-r}}
    ~~~\text{ for all } i \in [2^{k_0}m],
\end{equation}
and
\begin{equation} \label{eq:U0:l2}
    \sqrt{\frac{d}{n}} \leq \sqrt{\sum_{k=1}^d \left(\mU^{(0)}_{ik}\right)^2} \leq \sqrt{\frac{d}{n-r}}
    ~~~\text{ for all } i \in [2^{k_0}m].
\end{equation}
\end{lemma}
\begin{proof}
For a given latent space dimension $d = p + q$, we let $k_0 > 0$ be the integer such that $2^{k_0-1} < d \leq 2^{k_0}$.
Let $\mH_{2^{k_0}, d}$ denote the matrix obtained by retaining only the first $d$ columns of the Hadamard matrix $\mH_{2^{k_0}}$.
By construction, $\mH_{2^{k_0}, d}$ is a $2^{k_0}\times d$ matrix with orthogonal columns.
Assume that $n = 2^{k_0} m + r$ where $m>0$ is an integer and $r$ is a remainder term such that $0 \leq r < 2^{k_0} < 2d$.
To obtain an $n\times d$ matrix with orthonormal columns, we first stack $m$ copies of $\mH_{2^{k_0}, d}$ together vertically to obtain a matrix $\mJ_m \in \R^{m 2^{k_0} \times d}$ given by
\begin{equation*}
\mJ_m = \left. \begin{bmatrix}
    \mH_{2^{k_0}, d} \\
    \mH_{2^{k_0}, d} \\
    \vdots\\
    \mH_{2^{k_0}, d} 
\end{bmatrix}\right\} m \text{ copies of } \mH_{2^{k_0}, d}.
\end{equation*}
Defining
$\mK_r = \begin{bmatrix} \onevec_{r} & \mzero\\
\end{bmatrix} \in \R^{r \times d}$,
we construct a matrix $\mU_0 \in \R^{n \times d}$ with orthonormal columns by stacking $\mJ_m$ and $\mK_r$ and rescaling their columns.
For $i \in [2^{k_0}]$ and $j \in [d]$, let $h_{i, j}$ be the $(i, j)$ entry of $\mH_{2^{k_0}, d}$.
Noting that $h_{i,1} = 1$ for all $i \in [2^{k_0}]$, our construction of $\mU_0$ is then given by
\begin{equation} \label{eq:par1:base_U}
    \mU_0 = \begin{bmatrix}
    \frac{1}{\sqrt{n}} & \frac{h_{12}}{\sqrt{n-r}} & \frac{h_{13}}{\sqrt{n-r}} & \ldots & \frac{h_{1 d}}{\sqrt{n-r}} \\
    \vdots & \vdots & \vdots & \cdots & \vdots \\
    \frac{1}{\sqrt{n}} & \frac{h_{2^{k_0},2}}{\sqrt{n-r}} & \frac{h_{2^{k_0}, 3}}{\sqrt{n-r}} & \ldots & \frac{h_{2^{k_0}, d}}{\sqrt{n-r}} \\
    \vdots & \vdots & \vdots & \cdots & \vdots \\
    \frac{1}{\sqrt{n}} & \frac{h_{1,2}}{\sqrt{n-r}} & \frac{h_{1,3}}{\sqrt{n-r}} & \ldots & \frac{h_{1, d}}{\sqrt{n-r}} \\
    \vdots & \vdots & \vdots & \cdots & \vdots \\
    \frac{1}{\sqrt{n}} & \frac{h_{2^{k_0},2}}{\sqrt{n-r}} & \frac{h_{2^{k_0}, 3}}{\sqrt{n-r}} & \ldots & \frac{h_{2^{k_0}, d}}{\sqrt{n-r}} \\
    \onevec_{r}/\sqrt{n} &  & \mzero &
\end{bmatrix}.
\end{equation}
Noting that $|h_{i,j}| = 1$ for all $i \in [2^{k_0}]$ and $j \in [d]$, Equations~\eqref{eq:U0:max}-\eqref{eq:U0:l2} all follow from the construction in Equation~\eqref{eq:par1:base_U}. 
\end{proof}

To form our packing set, we will construct a collection of matrices that are far from $\mU_0$ (and from one another) in $(\tti)$-distance, but yield similar distributions over networks as measured by KL-divergence.
We will do this by selectively modifying one row of $\mU_0$ at a time.
Toward this end, Lemma~\ref{lem:constructxi} ensures the existence of a collection of vectors from which we will construct these perturbed versions of $\mU_0$.

\begin{lemma} \label{lem:constructxi}
Under the conditions of Theorem~\ref{thm:main:main}, suppose that $\kappastar = O(1)$.
Let $\mU_0$ be the matrix guaranteed by Lemma~\ref{lem:kappaconst:U0exists} and let $\lambda_1 \ge \lambda_2 \ge \cdots \ge \lambda_d > 0$ be arbitrary.
For each $i \in [n]$, let $\vu_i \in \bbR^d$ denote the $i$-th row of $\mU_0$.
For latent space dimension $d$, let $k_0$ be such that $2^{k_0-1} < d \le 2^{k_0}$ and let $m = \lfloor n/2^{k_0} \rfloor$.
For $n$ sufficiently large, for each $i \in [2^{k_0}m]$, there exists a vector $\vx_i$ such that
\begin{equation} \label{eq:main:sgn}
    \vx_i^T \vu_i \geq 0,
\end{equation}
\begin{equation} \label{eq:main:cos}
\left|\frac{\vx_i^T \vu_i}{\left\|\vx_i\right\|_2\left\|\vu_i\right\|_2}\right| < \frac{\sqrt{3}}{2},
\end{equation}
and for any constant $c_0>0$,
\begin{equation} \label{eq:main:abs}
    \left|\vx_{i,\ell}\right|
    = \frac{c_0}{\lambda_\ell} \sqrt{\frac{\lambda_1 (\lambda_d \wedge \log n)}{n d}}, ~~~ \ell \in [d],
\end{equation}
\end{lemma}
\begin{proof}
Fix $i \in [2^{k_0} m]$.
We will construct $\vx_i \in \bbR^d$ satisfying Equations~\eqref{eq:main:sgn},~\eqref{eq:main:cos} and~\eqref{eq:main:abs}.
Toward this end, consider the vector
\begin{equation*}
\vy = c_0 \sqrt{\frac{\lambda_1 (\lambda_d \wedge \log n)}{n d}}
	\left( \lambda_1^{-1}, \lambda_2^{-1}, \dots,
		\lambda_d^{-1} \right)^T \in \bbR^d,
\end{equation*}
where $c_0>0$ is any constant of our choosing.
Define $\va_i \in \bbR^d$ by
\begin{equation} \label{eq:def:va}
\va_{i, \ell}
= \frac{1}{\|\vy_i\|_2}\sign( \vu_{i, \ell}) \left|\vy_{i,\ell} \right|
~~~~~~ \ell \in [d],
\end{equation}
where for each $i \in [2^{k_0}m]$, we denote the $i$-th row of $\mU_0$ as $\vu_i \in \bbR^d$.
By Lemma~\ref{lem:tech:clever}, there exists a vector $\vz_i \in \R^d$ such that
\begin{equation} \label{eq:vavz:ub}
|\va_i^T \vz_i| < \sqrt{2/3}.
\end{equation}
We define $\vx_i \in \bbR^d$ by $\vx_{i, \ell} = \sign(\vz_{i, \ell}) \left|\vy_{i, \ell}\right|$, and Equation~\eqref{eq:main:abs} is satisfied trivially.

Define vector $\vw \in \bbR^d$ according to
\begin{equation*} 
\vw_{i, \ell} = \frac{ \sign(\vz_{i, \ell}) \left|\vu_{i, \ell} \right|}
			{\|\vu_i\|_2}
\end{equation*}
for $\ell \in [d]$.
Substituting and applying the triangle inequality,
\begin{equation*} 
\left| \frac{ \vx_i^T \vu_i }{ \| \vx_i \|_2 \| \vu_i \|_2 } \right|
\le \left| \va_i^T (\vw_i - \vz_i ) \right| + \left| \va_i^T \vz_i \right|
\le \left\| \vw_i - \vz_i \right\|_2 + \sqrt{\frac{2}{3}},
\end{equation*}
where the second inequality follows from Cauchy-Schwarz, the fact that $\| \va_i \| = 1$ and Equation~\eqref{eq:vavz:ub}.
Plugging in the definitions of $\vw_i$ and $\vz_i$,
\begin{equation} \label{eq:wz:bound}
\left\| \vw_i - \vz_i \right\|_2
=
\sqrt{\sum_{\ell = 1}^d
	\left(\frac{\left|\vu_{i, \ell} \right|}{\|\vu_i\|_2}
		- \frac{1}{\sqrt{d}}\right)^2}.
\end{equation}
From Equations~\eqref{eq:U0:max-i} and \eqref{eq:U0:l2}, we have 
\begin{equation*}
\left(\sqrt{\frac{n-r}{n}} - 1\right)\cdot{\frac{1}{\sqrt{d}}} \leq \frac{|\vu_{i, \ell}|}{\|\vu_i\|_2 } - \frac{1}{\sqrt{d}}\leq \left(\sqrt{\frac{n}{n-r}} - 1\right)\cdot{\frac{1}{\sqrt{d}}}
\end{equation*}
for any $\ell \in [d]$.
As a result, Equation~\eqref{eq:wz:bound} is further bounded by
\begin{equation*} \begin{aligned}
\left\| \vw_i - \vz_i \right\|_2
&\leq \max\left\{\left(1-\sqrt{\frac{n-r}{n}}\right),
		\left(\sqrt{\frac{n}{n-r}} - 1\right) \right\} \\
&\leq \max\left\{\frac{r}{n}, \frac{r}{2(n-r)}\right\},
\end{aligned} \end{equation*}
where the second inequality follows from the fact that for any $x \in [0,1]$, we have 
\begin{equation*}
1 - \sqrt{1 - x} \leq x \quad \text{ and } \quad \sqrt{1 + x} - 1 \leq \frac{x}{2}.
\end{equation*}
Hence, choosing an $n$ sufficiently large, for example, $n \geq 42 d \geq 21 r$, and it follows that
\begin{equation*}
\left| \frac{ \vx_i^T \vu_i }{ \| \vx_i \|_2 \| \vu_i \|_2 } \right|
\le
\frac{1}{21}
+ \sqrt{\frac{2}{3}} \le \frac{\sqrt{3}}{2}.
\end{equation*}

To see that $\vx_i$ can be chosen to satisfy Equation~\eqref{eq:main:sgn}, simply note that if $\vx_i^T \vu_i < 0$, we may replace $\vx_i$ with $-\vx_i$ without violating Equations~\eqref{eq:main:cos} and~\eqref{eq:main:abs}.
\end{proof}

With Lemma~\ref{lem:constructxi} in hand, we are ready to construct perturbations of the matrix $\mU_0$ guaranteed by Lemma~\ref{lem:kappaconst:U0exists}.
Our packing set argument in Theorem~\ref{thm:main:tsybakov} requires that the matrices $\{\mU_i : i=0,1,2,\dots,2^{k_0}m \}$ be suitably well separated in $(\tti)$-distance.
Lemma~\ref{lem:main:kappaconst:packing_distance} establishes that this is the case.

\begin{lemma}
\label{lem:main:kappaconst:packing_distance:prelim}
Under the setting of Theorem~\ref{thm:main:main}, suppose that $\kappa = O(1)$.
For latent space dimension $d$, let $k_0$ be such that $2^{k_0-1} < d \le 2^{k_0}$ and let $m = \lfloor n/2^{k_0} \rfloor$.
Let $\mLambda = \diag(\lambda_1,\lambda_2,\dots,\lambda_d)$ for $\lambda_1 \ge \lambda_2 \ge \cdots \ge \lambda_d > 0$ obeying $\kappa = \lambda_1 / \lambda_d$.
For all suitably large $n$, there exists a collection of matrices $\{ \mU_i : i =0,1,2,\dots,2^{k_0} m \}$ such that for all $i \in [2^{k_0}m]$ and $j \in [2^{k_0}m] \cup \{0\}$ not equal to $i$, 
\begin{equation} \label{eq:main:packing_distance}
    \min_{\mW \in \Od\cap \Opq} \|\mU_i\mLambda^{1/2} \mW - \mU_j \mLambda^{1/2}\|_{\tti} \geq \frac{1}{4} \|\vx_i^T \mLambda^{1/2} \|_2,
\end{equation}
where for each $i \in [2^{k_0}m]$, $\vx_i$ is the vector guaranteed by Lemma~\ref{lem:constructxi}.
\end{lemma}
\begin{proof}
For each $i \in [2^{k_0}m]$, define the matrix
\begin{equation} \label{eq:def:G}
\mG_i = \mU_0 + \ve_i\vx_i^T \in \R^{n \times d} ,
\end{equation}
Then, denoting the SVD of $\mG_i$ by  
$\mG_i = \tmU_{i} \tmSig_i \tmV^T_i$, define for each $i \in [2^{k_0}]$, the matrix
\begin{equation*}
    \mU_i = \tmU_i \tmV_i^T \in \R^{ n \times d }.
\end{equation*}

For the sake of simplicity, we prove Equation~\eqref{eq:main:packing_distance} for $i, j \in [2^{k_0} m]$.
When either $i=0$ or $j=0$, the proof follows a similar idea.
For a fixed $\mW \in \Od\cap \Opq$, by the triangle inequality, we have 
\begin{equation*} \begin{aligned}
\|\mU_i \mLambda^{1/2} \mW - \mU_j \mLambda^{1/2} \|_{2, \infty}
    &\geq \|\mG_i \mLambda^{1/2} \mW - \mG_j \mLambda^{1/2} \|_{2, \infty} \\
    &~~~- \|\mU_i\mLambda^{1/2} - \mG_i\mLambda^{1/2}\|_{2,\infty}
    - \|\mU_j\mLambda^{1/2} - \mG_j\mLambda^{1/2}\|_{2,\infty}. 
\end{aligned} \end{equation*}
Thus, to obtain our desired lower bound on $\|\mU_i \mLambda^{1/2} \mW - \mU_j \mLambda^{1/2} \|_{2, \infty}$, it will suffice to prove
\begin{enumerate}[label=(\emph{\roman*}), ref=(\emph{\roman*})]
\item an upper bound for every $\|\mU_i\mLambda^{1/2} - \mG_i\mLambda^{1/2}\|_{2,\infty}$ and \label{subgoal:packing:i}
\item a lower bound for every $\|\mG_i \mLambda^{1/2} \mW - \mG_j \mLambda^{1/2} \|_{2, \infty}$. \label{subgoal:packing:ii}
\end{enumerate}

To establish Item~\ref{subgoal:packing:i}, note that by our definitions and using basic properties of the $(\tti)$-norm and the operator norm,
\begin{equation*} \begin{aligned}
\|\mU_i\mLambda^{1/2} - \mG_i\mLambda^{1/2}\|_{2,\infty}
    &= \|\tmU_i \tmV_i^T \mLambda^{1/2} - \tmU_i \tmSig_i \tmV_i^T \mLambda^{1/2}\|_{2, \infty}\\
    &\leq \left\|\tmU_i\right\|_{2, \infty} \left\|\left(\tmSig_i - \mI_d\right)\tmV_i^T\mLambda^{1/2}\right\| \\
    &\leq \sqrt{\lambda_1} \left\|\tmU_i\right\|_{2, \infty} \left\|\tmSig_i - \mI_d\right\| .
    \end{aligned} \end{equation*}
By our choice of $\vx_i$ in Equation~\eqref{eq:main:abs}, we have
\begin{equation}
\label{eq:xi-l2-bound}
\|\vx_i\|^2_2 \leq \frac{c_0^2 \kappa}{n} \frac{\lambda_d \wedge \log n}{\lambda_d} \leq \frac{c_0^2\kappa}{n},
\end{equation}
where $c_0 > 0$ is a constant of our choosing.
For $n > c_0^2\kappa$, we have $\|\vx_i\|_2 < 1$, so that Lemmas~\ref{lem:eigenbound} and \ref{lem:Utilde} apply, and it follows that
\begin{equation*}
\|\mU_i\mLambda^{1/2} - \mG_i\mLambda^{1/2}\|_{2,\infty}
    \leq 2\sqrt{\lambda_1} \sqrt{\frac{d}{n-r} + \frac{\|\vx_i\|_2}{1 - \|\vx_i\|_2}} \left(\sqrt{\frac{d}{n-r}} + \frac{1}{4} \|\vx_i\|_2\right)\|\vx_i\|_2.
\end{equation*}
Observing that $\| \mLambda^{1/2} \vx_i \| \ge \sqrt{\lambda_d} \| \vx_i \|$,
\begin{equation*}
    \|\mU_i\mLambda^{1/2} - \mG_i\mLambda^{1/2}\|_{2,\infty}
\leq 2\sqrt{\kappa} \sqrt{\frac{d}{n-r} + \frac{\|\vx_i\|_2}{1 - \|\vx_i\|_2}} \left(\sqrt{\frac{d}{n-r}} + \frac{1}{4} \|\vx_i\|_2\right)\left\|\mLambda^{1/2}\vx_i\right\|_2.
\end{equation*}
Hence, in order to show that
\begin{equation}
\label{eq:subgoal:i}
    \|\mU_i\mLambda^{1/2} - \mG_i\mLambda^{1/2}\|_{2,\infty} \leq \frac{1}{8} \left\|\mLambda^{1/2}\vx_i\right\|_2,
\end{equation}
which will suffice for our upper bound in Item~\ref{subgoal:packing:i}, it suffices to have 
\begin{equation*} \begin{aligned}
 \sqrt{\frac{d}{n-r} + \frac{\|\vx_i\|_2}{1 - \|\vx_i\|_2}} \left(\sqrt{\frac{d}{n-r}} + \frac{ \|\vx_i\|_2 }{4} \right) 
     &\leq \frac{3d}{2(n-r)} + \frac{\|\vx_i\|_2}{1 - \|\vx_i\|_2} + \frac{\|\vx_i\|_2^2}{32}
     \leq \frac{1}{16\sqrt{\kappa}}
\end{aligned} \end{equation*}
where the first inequality uses the fact that $a(b+c) \leq a^2 + \frac{1}{2}b^2 + \frac{1}{2}c^2$.
This can be satisfied by requiring
\begin{equation*}
\frac{d}{n-r} \leq \frac{1}{96\sqrt{\kappa}}, \quad \frac{\|\vx_i\|_2}{1-\|\vx_i\|_2} \leq \frac{1}{48\sqrt{\kappa}} \quad \text{ and } \|\vx_i\|_2^2 \leq \frac{2}{3\sqrt{\kappa}},
\end{equation*}
where the last inequality is satisfied once $n \geq \max\{(96\sqrt{\kappa} + 2)d, 2500c_0^2 \kappa^2\}$.

Turning our attention to Item~\ref{subgoal:packing:ii}, by construction, we have for any $\mW \in \Od\cap \Opq$,
\begin{equation} \label{eq:append:tti}
    \begin{aligned}
    \|\mG_i\mLambda^{1/2}\mW - \mG_j\mLambda^{1/2}\|_{2, \infty}
= \max &\left\{\vphantom{\max_{\substack{\ell\in[n]\\ \ell\neq i,j}}}
	\left\|\mW^T \mLambda^{1/2} \left(\vu_i + \vx_i\right) 
		-  \mLambda^{1/2}\vu_i\right\|_2,\right.\\
&~~~\left. \left\| \mW^T \mLambda^{1/2} \left( \vu_j + \vx_j\right) 
		-  \mLambda^{1/2} \vu_j \right\|_2, \right. \\
&~~~\left. \max_{\substack{\ell\in[n]\\ \ell\neq i,j}}
	\left\| \mW^T \mLambda^{1/2} \vu_\ell
		- \mLambda^{1/2} \vu_{\ell}\right\|_2 
\right\}.
\end{aligned} \end{equation}
If $\| \mLambda^{1/2} \vx_i\|_2 \geq 2\left\| \mW^T \mLambda^{1/2} \vu_i - \mLambda^{1/2} \vu_i \right\|_2$, then trivially
\begin{equation*} 
\left\| \mW^T \mLambda^{1/2} \left(\vu_i + \vx_i\right)
	- \mLambda^{1/2} \vu_i \right\|_2
\geq \| \mLambda^{1/2} \vx_i \|_2 
	- \| \mW^T \mLambda^{1/2} \vu_i - \mLambda^{1/2} \vu_i \|_2
    \geq \frac{1}{2}\| \mLambda^{1/2} \vx_i\|_2.
\end{equation*}

Otherwise, suppose that $\| \mLambda^{1/2} \vx_i\|_2 < 2\left\|\mW^T \mLambda^{1/2} \vu_i - \mLambda^{1/2} \vu_i \right\|_2$.
When $n \geq 8 d$, we have $m \geq 3$.
Note that $n\ge 8d$ holds eventually, since $3d \le \kappa = O(1)$ by assumption.
From our construction and using the fact that $m\geq 3$, we can always find an $\ell \in [2^{k_0}m]$ distinct from $i$ and $j$ such that $\vu_{\ell} = \vu_i$ and  
\begin{equation*}
\left\|\mW^T \mLambda^{1/2}\vu_\ell - \mLambda^{1/2} \vu_\ell \right\|_2
= \left\| \mW^T \mLambda^{1/2} \vu_i - \mLambda^{1/2} \vu_i \right\|_2
\geq \frac{1}{2} \|\mLambda^{1/2} \vx_i\|_2.
\end{equation*}
Thus, combining the two cases with Equation~\eqref{eq:append:tti}, 
\begin{equation*}
\|{\mG_i} \mLambda^{1/2}\mW - {\mG_j} \mLambda^{1/2}\|_{2, \infty} \geq \frac{1}{2} \| \mLambda^{1/2} \vx_i\|_2.
\end{equation*}
Combining this with the upper bound in Equation~\eqref{eq:subgoal:i}, we have
\begin{equation*} \begin{aligned}
\|\mU_i \mLambda^{1/2} \mW - \mU_j \mLambda^{1/2} \|_{2, \infty}
&\geq \frac{ \|\mLambda^{1/2} \vx_i\|_2 }{ 2 }
	- \frac{ \| \mLambda^{1/2} \vx_i\|_2 }{ 8 }
	- \frac{ \| \mLambda^{1/2} \vx_i\|_2 }{ 8 } 
\geq \frac{1}{4} \| \mLambda^{1/2} \vx_i\|_2. 
\end{aligned} \end{equation*}
Noting that the right-hand side does not depend on $\mW$, minimizing over $\mW\in \Od \cap \Opq$ completes the proof.
\end{proof}

\begin{lemma} \label{lem:main:kappaconst:packing_distance}
Under the conditions of Theorem~\ref{thm:main:main}, suppose that $\kappa = O(1)$.
For latent space dimension $d$, let $k_0$ be such that $2^{k_0-1} < d \le 2^{k_0}$ and let $m = \lfloor n/2^{k_0} \rfloor$.
Let $\mLambda = \diag(\lambda_1,\lambda_2,\dots,\lambda_d)$ for $\lambda_1 \ge \lambda_2 \ge \cdots \ge \lambda_d > 0$ obeying $\kappa = \lambda_1 / \lambda_d$.
For all suitably large $n$, there exists $\mLambda = \diag(\lambda_1,\lambda_2,\dots,\lambda_d) \in \R^d$ and a collection of matrices $\{ \mU_i : i =0,1,2,\dots,2^{k_0} m \}$ such that for all distinct $i, j \in [2^{k_0}m]\cup \{0\}$,
\begin{equation*}
    \min_{\mW \in \Od\cap \Opq} \|\mU_i\mLambda^{1/2} \mW - \mU_j \mLambda^{1/2}\|_{\tti}
    \geq \frac{ c_0 }{ 8} \
    \sqrt{ \frac{ (\lambda_d \wedge \log n)  \kappa }{ n }},
\end{equation*}
where $c_0 > 0$ is as in Lemma~\ref{lem:constructxi}.
\end{lemma}
\begin{proof}
Let $\lambda_1$ be such that $\lambda_1 \le n/2$ and
\begin{equation*} 
\lambda_2=\lambda_3 =\cdots=\lambda_d = \lambda_1/\kappa,
\end{equation*}
and set $\mLambda = \diag( \lambda_1,\lambda_2,\dots,\lambda_d) \in \R^{d\times d}$.
By Lemma~\ref{lem:main:kappaconst:packing_distance:prelim}, there exists a collection of matrices $\{ \mU_i : i =0,1,2,\dots,2^{k_0} m \} \subset \Stiefel_d(\R^n)$ such that for all $i \in [2^{k_0}m]$ and $j \in [2^{k_0}m] \cup \{0\}$ not equal to $i$, 
\begin{equation} \label{eq:kappaconst:finalLB}
\min_{\mW \in \Od\cap \Opq} \|\mU_i\mLambda^{1/2} \mW - \mU_j \mLambda^{1/2}\|_{\tti} \geq \frac{1}{4} \|\vx_i^T \mLambda^{1/2} \|_2,
\end{equation}
where, letting $\vu_i \in \R^d$ denote the $i$-th row of $\mU_0$, $\vx_i$ satisfies
\begin{equation*}
\left| \vx_{i,\ell} \right|
= \frac{ c_0 }{ \lambda_\ell}
    \sqrt{ \frac{ \lambda_1 (\lambda_d \wedge \log n)}
            { n d } }.
\end{equation*}
Expanding and plugging in our choice of $\mLambda$,
\begin{equation*} 
\|\vx_i^T \mLambda^{1/2} \|_2^2
= \sum_{j=1}^d \vx_{i,j}^2 \lambda_j
= \frac{ c_0^2 }{ nd } \sum_{j=1}^d 
    \frac{ \lambda_1 (\lambda_d \wedge \log n)}
            { \lambda_j  }
= \frac{ c_0^2 }{ nd }
  \left( 1 + (d-1)\kappa \right)
        (\lambda_d \wedge \log n).
\end{equation*}
Using the fact that $\kappa=1$ when $d=1$ and that $(d-1) \ge 1$ otherwise, it follows that, taking square roots and using the fact that $d$ is a constant,
\begin{equation*}
\|\vx_i^T \mLambda^{1/2} \|_2
\ge c_0
  \sqrt{\frac{1}{d} + \left(1-\frac{1}{d}\right)\kappa}
        \sqrt{\frac{\lambda_d \wedge \log n}{n}}
\ge \frac{ c_0 }{ 2}
    \sqrt{ \frac{ (\lambda_d \wedge \log n)  \kappa }{ n }}.
\end{equation*}
Plugging this lower-bound into Equation~\eqref{eq:kappaconst:finalLB}, it follows that
\begin{equation*}
\min_{\mW \in \Od\cap \Opq} \|\mU_i\mLambda^{1/2} \mW - \mU_j \mLambda^{1/2}\|_{\tti} \geq
\frac{ c_0 }{ 8} \
    \sqrt{ \frac{ (\lambda_d \wedge \log n)  \kappa }{ n }},
\end{equation*}
completing the proof.
\end{proof}

Lemma~\ref{lem:main:kappaconst:packing_distance} guarantees the existence of a collection elements of $\Stiefel_d(\R^n)$ that are well separated in $(\tti)$-norm after right-multiplication by some $\mLambda = \diag( \lambda_1,\lambda_2,\dots,\lambda_d) \in \R^d$.
To construct a collection of valid expected adjacency matrices from this collection of $d$-frames, we must choose $\mLambda$ so that $\mU \mLambda^{1/2} \mI_{p,q} \mLambda^{1/2} \mU^T$ has all entries between $0$ and $1$ for every $\mU$ in our collection.
Lemma~\ref{lem:main:kappaconst:P0exists} ensures that this is possible.

\begin{lemma} \label{lem:main:kappaconst:P0exists}
Under the conditions of Theorem~\ref{thm:main:main}, suppose that $\kappa = O(1)$.
Let $\mU_0 \in \Stiefel_d(\R^n)$ be the matrix guaranteed by Lemma~\ref{lem:kappaconst:U0exists}.
There exist $\lambda_1 \ge \lambda_2 \ge \cdots \ge \lambda_d > 0$ such that, letting $\mLambda = \diag( \lambda_1,\lambda_2,\dots,\lambda_d) \in \R^{d \times d}$, the matrix
\begin{equation} \label{eq:def:P0:kappaconst}
\mP_0 = \mU_0 \mLambda^{1/2}\mI_{p,q}\mLambda^{1/2} \mU_0^T
\end{equation}
obeys
\begin{equation}
    \label{eq:P0-bounds}
    \frac{\lambda_1}{3n} \leq \mP^{(0)}_{ij} \leq \frac{2}{3}.
\end{equation}
for all $i,j \in [n]$ and all $n$ sufficiently large. 
\end{lemma}
\begin{proof}
We will choose $\lambda_1 \ge \lambda_2 \ge \cdots \ge \lambda_d > 0$ so that 
\begin{equation} \label{eq:main:eigen}
    \sum_{j=2}^d \lambda_j \leq \frac{\lambda_1}{1+\epsilon}  
~~~\text{ and }~~~
    \sum_{j=1}^d \lambda_j \leq \frac{n}{1+\epsilon},
\end{equation}
for any constant $\epsilon > 0$. 
So long as $\kappa = \lambda_1/\lambda_d \geq (1+\epsilon)d$, we can satisfy Equation~\eqref{eq:main:eigen} by taking
$\lambda_1 \leq n/(2+\epsilon)$ and $\lambda_d = \lambda_{d-1} =\ldots = \lambda_2 = \lambda_1/\kappa$, so that 
\begin{equation*}
\sum_{j=2}^d \lambda_j = \frac{d-1}{\kappa} \lambda_1 \leq \frac{\lambda_1}{1+\epsilon} 
\end{equation*}
and 
\begin{equation*}
    \sum_{j=1}^d \lambda_1 \leq \frac{2+\epsilon}{1+\epsilon} \lambda_1 \leq \frac{n}{1+\epsilon}. 
\end{equation*}

To find lower and upper bounds for each entry of $\mP_0$, we unroll the definition in Equation~\eqref{eq:def:P0:kappaconst} to write
\begin{equation} \label{eq:P0LB:kappaconst}
\mP^{(0)}_{ij} \geq \frac{1}{n} \lambda_1 - \frac{1}{n - r}\sum_{l = 2}^d \lambda_\ell
\geq \left(\frac{1}{n} - \frac{1}{(1+\epsilon)(n-r)}\right) \lambda_1
\geq \frac{\epsilon\lambda_1}{(1+2\epsilon)n},
\end{equation}
where the last inequality holds for $n$ sufficiently large.

To upper bound the entries of $\mP_0$, we have
\begin{equation*}
\mP^{(0)}_{ij} \leq \frac{1}{n} \lambda_1 + \frac{1}{n - r}\sum_{l = 2}^d \lambda_\ell
\leq \frac{n}{(1+\epsilon)(n-r)} \leq \frac{1+\epsilon}{1+2\epsilon},
\end{equation*}
where the second inequality holds for sufficiently large $n$.
Combining this with Equation~\eqref{eq:P0LB:kappaconst} and taking $\epsilon = 1$,
\begin{equation*}
    \frac{\lambda_1}{3n} \leq \mP^{(0)}_{ij} \leq \frac{2}{3},
\end{equation*}
as we set out to show.
\end{proof}

Our perturbation of $\mU_0$ to obtain $\mU_i$ comes from $\vx_i$ (guaranteed by Lemma~\ref{lem:constructxi}), whose norm $\|\vx_i\|_2$ is of order $O(n^{-1/2})$ as established in Equation~\eqref{eq:xi-l2-bound}.
As a result, it is not hard to imagine that this perturbation should not change the entries of $\mP_0$ too much.
Lemma \ref{lem:i:kappaconst:prob} shows that this is indeed the case.

\begin{lemma} \label{lem:i:kappaconst:prob}
Under the setting of Theorem~\ref{thm:main:main}, suppose that $\kappa = O(1)$.
For latent space dimension $d$, let $k_0$ be such that $2^{k_0-1} < d \le 2^{k_0}$ and define $m = \lfloor n/2^{k_0} \rfloor$.
Let $\{ \mU_i : i \in [2^{k_0}m] \}$ be the collection of $d$-frames guaranteed by Lemma~\ref{lem:main:kappaconst:packing_distance}, and let $\mLambda \in \R^{d \times d}$ be the diagonal matrix guaranteed by Lemma~\ref{lem:main:kappaconst:P0exists}.
For sufficiently large $n$, letting $\tmLambda = \mLambda^{1/2} \mI_{p,q} \mLambda^{1/2} $ it holds for all $i \in [2^{k_0}m]$ that $\mP_i = \mU_i \tmLambda \mU_i^T$ has all entries strictly bounded between $0$ and $1$.
\end{lemma}
\begin{proof}
Our strategy will be to show that for any $i \in [m2^{k_0}]$, the matrix $\mU_i$ is sufficiently close to $\mU_0$ entry-wise.
From this fact, it will follow that the entries of $\mP_i$are all close to $\mP_0$.
Toward this end, we begin by recalling from the proof of Lemma~\ref{lem:main:kappaconst:packing_distance} that the matrices $\mU_i$ are defined according to
\begin{equation*} 
\mU_i - \mU_0 = \mU_{i} - \mG_i + \mG_i - \mU_0
    = \tmU_{i} \left(\mI_d - \tmSig_i\right) \tmV_i^T + \ve_i \vx_i^T,
\end{equation*}
where $\mG_i$ is as defined in Equation~\eqref{eq:def:G} and $\mG_i = \tmU_{i} \tmSig_i \tmV^T_i$ is its SVD.
Define $\mR = \tmU_{i} \left(\mI_d - \tmSig_i\right) \tmV_i^T$.
By Lemma~\ref{lem:tech:eigen}, we have for $j \in [n]$ and $k \in [d]$,
\begin{equation*} \begin{aligned}
\left|\mR_{jk}\right|
&= \left|\sum_{\ell = 1}^2 \tmU^{(i)}_{j\ell}
    \left(1 - \sqrt{1 + \Tilde{\sigma}_{i\ell}}\right) \tmV^{(i)}_{k\ell}\right|
  \leq 2\max_{\ell \in [d]}\left\{\left|\tmU_{j\ell}^{(i)}\right|\right\}\left\|\tmSig - \mI_d\right\| \\
&\leq 2 \|\tmU_i\|_{2,\infty} \left\|\tmSig - \mI_d\right\|.
\end{aligned} \end{equation*}
From Equation~\eqref{eq:xi-l2-bound}, we have $\|\vx_i\|_2 \leq c_0\sqrt{\kappa/n}$, and for suitably large $n$ we may ensure that $\| \vx_i \|_2 < 1/2$ so that  Lemmas~\ref{lem:eigenbound} and \ref{lem:Utilde} apply.
If, in addition, we have $n \geq 4d \geq 2r$, then it follows that
\begin{equation*} \begin{aligned}
\left|\mR_{jk}\right|
&\leq \sqrt{\frac{d}{n-r} + \frac{\|\vx_i\|_2}{1 - \|\vx_i\|_2}}\left(\left\|\vx_i\right\|_2^2+4 \sqrt{\frac{d}{n-r}}\left\|\vx_i\right\|_2\right)\\
&\leq 2\sqrt{\frac{d}{n} + \|\vx_i\|_2}\left(\left\|\vx_i\right\|_2+8 \sqrt{\frac{d}{n}}\right)\left\|\vx_i\right\|_2.
\end{aligned} \end{equation*}

Again recalling that $\| \vx_i \|_2 = O(n^{-1})$, it holds for suitably large $n$ that
\begin{equation*}
2\sqrt{\frac{d}{n} + \|\vx_i\|_2}\left(\left\|\vx_i\right\|_2+8 \sqrt{\frac{d}{n}}\right)
\leq \frac{1}{\sqrt{d}}.
\end{equation*}
Noting that $\|\vx_i\|_{\infty} \geq \|\vx_i\|_2/\sqrt{d}$, it follows that $|\mR_{jk}| \leq \|\vx_i\|_\infty$ for $n$ suitably large, and thus 
\begin{equation*} 
\max_{j\in[n], k\in[d]} \left|\mU^{(i)}_{jk} - \mU^{(0)}_{jk} \right|
\leq \left|\mR_{jk}\right| +  \|\vx_i\|_{\infty}
\leq 2 \|\vx_i\|_\infty.
\end{equation*}
From Equation~\eqref{eq:main:abs}, we have 
\begin{equation} \label{eq:kappaconst:c0choice}
\max_{j\in[n], k\in[d]} \left|\mU^{(i)}_{jk} - \mU^{(0)}_{jk}\right|
\leq 2c_0\sqrt{\frac{\kappa}{d n}}.
\end{equation}

For any $\mU \in \Stiefel_d(\R^n)$, unrolling the definition $\mP_{jk} = (\mU \tmLambda \mU^T)_{jk}$,
\begin{equation*} \begin{aligned}
\mP_{jk}
&= \sum_{\ell=1}^d \tmLambda_{\ell\ell} \mU_{j\ell} \mU_{k\ell}
\geq \mP^{(0)}_{jk} - \sum_{\ell=1}^d \lambda_{\ell} |\mU^{(0)}_{j\ell} \mU^{(0)}_{k\ell} - \mU_{j\ell} \mU_{k\ell}| \\
&\geq \left(\frac{1}{3n} - \sum_{\ell=1}^d |\mU^{(0)}_{j\ell} \mU^{(0)}_{k\ell} - \mU_{j\ell} \mU_{k\ell}| \right) \lambda_1.
\end{aligned} \end{equation*}
If each entry of $\mU$ differs from the corresponding entry of $\mU_0$ by at most $C/d\sqrt{n}$ for some constant $C > 0$, it follows from Equation~\eqref{eq:U0:max} that 
\begin{equation*} \begin{aligned}
|\mU^{(0)}_{j\ell} \mU^{(0)}_{k\ell} - \mU_{j\ell} \mU_{k\ell}|
&\leq
  \left|\mU_{j\ell} - \mU_{j\ell}^{(0)}\right|\left|\mU_{k\ell}^{(0)}\right|
  + \left|\mU_{k\ell} - \mU_{k\ell}^{(0)}\right|\left|\mU_{j\ell}^{(0)}\right|\\
&~~~~~~~~~+ \left|\mU_{j\ell} - \mU_{j\ell}^{(0)}\right|\left|\mU_{k\ell} - \mU_{k\ell}^{(0)}\right|\\
    &\leq \frac{2C}{d\sqrt{n}} \frac{1}{\sqrt{n-r}} + \frac{C^2}{d^2 n} \leq \frac{c}{d n},
\end{aligned} \end{equation*}
provided $C>0$ and $c > 0$ are chosen suitably small.

Combining the above two displays,
\begin{equation*}
\mP_{jk} \geq \left(\frac{1}{3n} - \frac{c}{n}\right)\lambda_1.
\end{equation*}
Thus, when $c > 0$ is sufficiently small, we have $\mP_{jk} > 0$.
Similarly, we have 
\begin{equation*} \begin{aligned}
\mP_{jk}
&= \sum_{\ell=1}^d \tmLambda_{\ell\ell} \mU_{j\ell} \mU_{k\ell}
\leq \mP^{(0)}_{jk} + \sum_{\ell=1}^d \lambda_{\ell} |\mU^{(0)}_{j\ell} \mU^{(0)}_{k\ell} - \mU_{j\ell} \mU_{k\ell}| \\
&\leq \frac{2}{3} + \sum_{\ell=1}^d |\mU^{(0)}_{j\ell} \mU^{(0)}_{k\ell} - \mU_{j\ell} \mU_{k\ell}| \lambda_1 \leq \frac{2}{3} + \frac{c}{n}\lambda_1 \leq \frac{2}{3}+c.
\end{aligned} \end{equation*}
It follows that $\mP_{jk} < 1$ for $c$ sufficiently small.
Hence, fixing such a $c > 0$, recalling that the constant $c_0 > 0$ in Lemma~\ref{lem:constructxi} is ours to choose,  we can pick $c_0 < C/2\sqrt{d\kappa}$ in 
Equation~\eqref{eq:main:abs} to be small enough so that the bound in Equation~\eqref{eq:kappaconst:c0choice} becomes
\begin{equation*} 
\max_{j\in[n], k\in[d]} \left|\mU^{(i)}_{jk} - \mU^{(0)}_{jk}\right|
\le \frac{ C }{ d \sqrt{n} }.
\end{equation*}
It follows that for $n$ suitably large, all entries of $\mP_i$ obtained from $\mU_i$ lie between $0$ and $1$, as we set out to show.
\end{proof}

\subsection{Bounding the KL Divergence} \label{subsec:apx:kappaconst:KL}

In addition to the lower-bound guaranteed by Lemma~\ref{lem:main:kappaconst:packing_distance}, Theorem~\ref{thm:main:tsybakov} requires an upper bound on the KL divergences between the distributions encoded by our $\mU_i$ matrices.
To control the KL divergence between $\mP_i$ and $\mP_0$, we use a basic result established by \cite{Zhou2021}, which we restate here for ease of reference.

\begin{lemma}[\cite{Zhou2021}, Lemma 3] \label{lem:Zhou2021}
Let $0 < a \le b < 1$ be such that $a \leq \mP^{(0)}_{ij} \leq b$ for any $i, j \in [n]$.
Then 
\begin{equation*} 
    \KL\left(\mP_i \| \mP_0 \right) \leq \frac{\|\mP_i - \mP_0\|_F^2}{a(1-b)},
\end{equation*}
\end{lemma}

In light of this result, it will suffice for us to bound the Frobenius norm between $\mP_0$ and each of the other elements of our packing set.

\begin{lemma} \label{lem:main:kappaconst:frob_bound}
Under the setting of Theorem~\ref{thm:main:main}, suppose that $\kappa = O(1)$.
For latent dimension $d$, let $k_0$ be such that $2^{k_0-1} < d \le 2^{k_0}$ and let $m = \lfloor n/2^{k_0} \rfloor$.
With $\mP_0$ as in Lemma~\ref{lem:main:kappaconst:P0exists}, it holds for all $i \in [2^{k_0}m]$ that
    \begin{equation*}
            \|\mP_i - \mP_0\|_F^2 \leq \frac{\lambda_1(\lambda_d\wedge\log n)}{90 n}.
    \end{equation*}
\end{lemma}
\begin{proof}
Recall that $\mG_i$ is defined in Equation~\eqref{eq:def:G} and $\mG_i = \tmU_i\tmSig_i\tmV_i^T$. We also have $\mU_i = \tmU_i\tmV_i^T$. Applying these definitions and the triangle inequality, 
\begin{equation*}
    \begin{aligned}
    \left\|\mU_i \tmLambda \mU_i^T - \mG_i \tmLambda \mG_i^T\right\|_F &= \left\|\tmU_i\tmV_i^T\tmLambda  \tmV_i \tmU_i^T - \tmU_i\tmSig_i\tmV_i^T\tmLambda  \tmV_i \tmSig_i\tmU_i^T\right\|_F\\
    &\leq \left\|\tmV_i^T\tmLambda  \tmV_i - \tmSig_i\tmV_i^T\tmLambda  \tmV_i \tmSig_i\right\|_F\\
    &= \left\|(\mI_d - \tmSig_i)\tmV_i^T\tmLambda  \tmV_i + \tmSig_i\tmV_i^T\tmLambda\tmV_i (\mI_d - \tmSig_i)\right\|_F
\end{aligned}
\end{equation*}
Applying the triangle inequality, it yields that
\begin{equation}
    \label{eq:kl_bound_pre1}
    \begin{aligned}
        \left\|\mU_i \tmLambda \mU_i^T - \mG_i \tmLambda \mG_i^T\right\|_F &\leq \left\|(\mI_d - \tmSig_i)\tmV_i^T\tmLambda  \tmV_i\right\|_F + \left\|\tmSig_i\tmV_i^T\tmLambda\tmV_i (\mI_d - \tmSig_i)\right\|_F\\
        &\leq \left\|(\mI_d - \tmSig_i)\tmV_i^T\tmLambda\right\|_F + \left\|\tmSig_i\right\|\left\|\tmV_i^T\right\|\left\|\tmLambda\tmV_i (\mI_d - \tmSig_i)\right\|_F\\
        &= \left(1 + \|\tmSig_i\|\right)\left\|(\tmSig_i-\mI_d)\tmV_i^T\tmLambda\right\|_F.
    \end{aligned}
\end{equation}
Also recalling that $\mG_i = \mU_0 + \ve_i\vx_i^T$, we have
\begin{equation}
    \label{eq:kl_bound_pre2}
    \begin{aligned}
        \|\mG_i \tmLambda \mG_i^T - \mU_0 \tmLambda \mU_0^T\|_F &= \| (\mU_0+\ve_i\vx_i^T)\tmLambda(\mU_0+\ve_i\vx_i^T)^T - \mU_0 \tmLambda \mU_0^T\|_F \\
        &\leq 2\|\vx^T_i \mLambda\|_2 + |\vx_i^T\mLambda \vx_i|\\
        &\leq 3\|\vx^T_i \mLambda\|_2 
    \end{aligned}
\end{equation}
where the last inequality holds if $\|\vx_i\|_2 < 1$.

Plugging in Equation~\eqref{eq:kl_bound_pre1} and ~\eqref{eq:kl_bound_pre2},
\begin{equation*} \begin{aligned}
\|\mP_i - \mP_0\|_F
&= \|\mU_i \tmLambda \mU_i^T - \mU_0 \tmLambda \tmU_0^T\|_F \\
&\leq \|\mU_i \tmLambda \mU_i^T - \mG_i \tmLambda \mG_i^T\|_F 
        + \|\mG_i \tmLambda \mG_i^T - \mU_0 \tmLambda \mU_0^T\|_F \\
 &< \left(1 + \|\tmSig_i\| \right)
        \left\|\left(\tmSig_i - \mI_d\right) \tmV_i^T \tmLambda\right\|_F 
        + 3\left\|\vx_i^T \mLambda\right\|_2.
\end{aligned} \end{equation*}
Applying Lemma \ref{lem:eigenbound}, if $\|\vx_i\|_2 < 1$, then we have
$$
\begin{aligned}
    \|\tmSig_i\| &\leq 1 + \|\mI_d - \tmSig_i\| \\
    &\leq 1 + \frac{1}{2}\|\vx_i\|_2^2 + 2\sqrt{\frac{d}{n-r}} \|\vx_i\|_2\\
    &< \frac{3}{2} + 2\sqrt{\frac{d}{n-r}} \|\vx_i\|_2.
\end{aligned}
$$
From Equation~\eqref{eq:xi-l2-bound}, we have $\|\vx_i\|_2 \leq c_0 \sqrt{\kappa/n}$, hence, we have
$$
\|\vx_i\|_2 \leq \frac{1}{4}\sqrt{\frac{n-r}{d}}
$$
when $n > 2d + 4c_0\sqrt{\kappa d}$. It suffices to require $n > c_0^2 \kappa$ to satisfy that $\|\vx_i\|_2 < 1$. Therefore, if $n > 2d + 4c_0 \sqrt{\kappa d} + c_0^2 \kappa$, then we have $ \|\tmSig_i\|  < 2$, and Lemma~\ref{lem:tech:refined_eigen_bound} implies
\begin{equation*} \begin{aligned}
\|\mP_i - \mP_0\|_F 
    &\leq 3 \sqrt{17\left(\|\vu_i\|_2^2 + \|\vx_i\|_2^2\right)\left(\vx_i^T \mLambda^2 \vx_i +\vu_i^T \mLambda^2 \vu_i\right)} + 3\left\|\vx_i^T \mLambda\right\|_2\\
    &\leq \left(15 \|\vu_i\|_2 + 15\|\vx_i\|_2 + 3\right)\left\|\vx_i^T \mLambda\right\|_2 + 15\left(\|\vu_i\|_2 + \|\vx_i\|_2 \right)\left\|\vu_i^T \mLambda\right\|_2.
\end{aligned} \end{equation*}
If $n > 902d$, then we have $\|\vu_i\|_2 \leq 1/30$ by the fact that $\|\vu_i \|_2^2 \leq d/(n-r)$. Furthermore, if $n \geq 900c_0^2\kappa$, then $\|\vx_i\|_2 \leq 1/30$. Hence, setting $n \geq 904d + 4c_0\sqrt{\kappa d} + 900c_0^2\kappa$ and collecting terms, we have
\begin{equation} \label{eq:PiP0:inter}
\begin{aligned}
    \|\mP_i - \mP_0\|_F 
    &\leq 4\left\|\vx_i^T \mLambda\right\|_2 + 15(\|\vu_i\|_2 + \|\vx_i\|_2)\left\|\vu_i^T \mLambda\right\|_2\\
    &\leq 4\left\|\vx_i^T \mLambda\right\|_2 + 15\left(\sqrt{\frac{d}{n-r}} + c_0\sqrt{\frac{\kappa}{n}}\right)\left\|\vu_i^T \mLambda\right\|_2
\end{aligned}
\end{equation}
By construction and Equation~\eqref{eq:U0:max-i},
\begin{equation*}
\|\vu_i^T \mLambda\|_2 = \sqrt{\sum_{\ell = 1}^d \lambda_{\ell}^2 \left(\vu_{i\ell}\right)^2} \leq \lambda_1\sqrt{\frac{d}{n - r}}.
\end{equation*}
and
\begin{equation*}
\vx_i^T \mLambda^2 \vx_i = c_0^2 \frac{\lambda_1 (\lambda_d\wedge\log n)}{n}.
\end{equation*}
Applying the above bounds to Equation~\eqref{eq:PiP0:inter}, we obtain
\begin{equation*}
\begin{aligned}
    \|\mP_i - \mP_0\|_F &\leq 4c_0\sqrt{\frac{\lambda_1(\lambda_d \wedge \log n)}{n}} + 15 \frac{d \lambda_1}{n-r} + 15c_0\lambda_1\sqrt{\frac{\kappa d}{n(n-r)}}\\
    &\leq \left(4c_0 + \frac{30d}{\sqrt{n}} \sqrt{\kappa \vee \frac{\lambda_1}{\log n}} + 15\sqrt{2}c_0 \sqrt{\frac{\kappa d}{n}\left(\kappa \vee \frac{\lambda_1}{\log n}\right)}\right)\sqrt{\frac{\lambda_1(\lambda_d \wedge \log n)}{n}},
\end{aligned}
\end{equation*}
where the last inequality holds if $n \geq 4d \geq 2r$.
To complete the proof, it suffices to have
$$
4c_0 + 30 \frac{d}{\sqrt{n}} \sqrt{\kappa \vee \frac{\lambda_1}{\log n}} + 15\sqrt{2}c_0 \sqrt{\frac{\kappa d}{n}\left(\kappa \vee \frac{\lambda_1}{\log n}\right)} \leq \frac{1}{3\sqrt{10}}, 
$$
which holds true if $c_0 \leq \frac{1}{36\sqrt{10}}$ and
$$
d \leq \max\left\{\frac{\sqrt{n}}{270\sqrt{10}} \left(\frac{1}{\sqrt{\kappa}} \wedge \sqrt{\frac{\log n}{\lambda_1}}\right), \frac{n}{364500 c_0^2 \kappa} \left(\frac{1}{\kappa} \wedge \frac{\log n}{\lambda_1}\right)\right\}.
$$
Since both $\kappa$ and $d$ are assumed to be bounded, the requirements all hold for $n$ sufficiently large.
\end{proof}

Using Equation~\eqref{eq:P0-bounds} to bound the entries of $\mP^{(0)}$ away from zero and applying Lemma~\ref{lem:Zhou2021}, we have
\begin{equation*}
\KL(\mP_i \| \mP_0) \leq \frac{9n\|\mP_i - \mP_0\|_F^2}{\lambda_1}.
\end{equation*}
Finally, the following Lemma yields our desired bound on the KL divergence for use in Theorem~\ref{thm:main:tsybakov}. 

\begin{lemma} \label{lem:main:kappaconst:KL_bound}
Under the setting of Theorem~\ref{thm:main:main}, suppose that $\kappa = O(1)$.
For latent space dimension $d$, let $k_0$ be such that $2^{k_0-1} < d \le 2^{k_0}$ and define $m = \lfloor n/2^{k_0} \rfloor$.
There exists a matrix $\mP_0 \in [0,1]^{n \times n}$ and a collection of $2^{k_0} m $ matrices $\{ \mP_i : i \in [ 2^{k_0} m ] \} \subset [0,1]^{n \times n}$ such that for all suitably large $n$,
\begin{equation*}
\KL\left( \mP_i \| \mP_0 \right) \le \frac{1}{10} \log n.
\end{equation*}
\end{lemma}
\begin{proof}
Let $\mP_0$ and $\mP_i$ be as defined in Lemmas~\ref{lem:main:kappaconst:P0exists} and~\ref{lem:i:kappaconst:prob}, respectively.
We note that by Lemma~\ref{lem:main:kappaconst:P0exists}, we can bound the elements of $\mP_0$ as
\begin{equation*}
\frac{ \lambda_1 }{ 3n } \le \mP^{(0)}_{ij} \le \frac{2}{3}~~~\text{ for all }~~~i,j \in [n].
\end{equation*}
Applying Lemma~\ref{lem:Zhou2021} followed by Lemma~\ref{lem:main:kappaconst:frob_bound}, it follows that
\begin{equation*}
\KL\left(\mP_i \| \mP_0 \right)
\leq
\frac{9n \|\mP_i - \mP_0\|_F^2}{ \lambda_1 }
\le
\frac{ \lambda_d \wedge \log n }{ 10 }
\le \frac{ 1 }{10} \log n
\end{equation*}
for all suitably large $n$, completing the proof.
\end{proof}

\section{Theorem~\ref{thm:main:main}: Growing condition number}
\label{sec:apx:kappagrowing}
When $\kappa = \omega(1)$, we require a different construction for our packing set than that used in Section~\ref{sec:apx:kappaconst}.
Were we to use the construction for the $\kappa = O(1)$ case here, we would further require addition assumptions on the growth of $\lambda_1$.
To obtain more general results, we pursue a different construction here.

\subsection{Constructing the Packing Set} \label{subsec:apx:kappagrowing:packing}

Our approach, as in the $\kappa = O(1)$ case, is to first construct a ``base'' parameter $\mU_0 \in \R^{n \times d}$ to have orthonormal columns.
We will then construct additional $d$-frames $\mU_i \in \R^{n \times d}$ by swapping pairs of rows in $\mU_0$.
To construct $\mU_0$, we take its first column to be $\onevec/\sqrt{n}$.
To construct the remaining columns, we stack columns from a $2^{k_0}\times 2^{k_0}$ Hadamard matrix $\mH_{2^{k_0}}$.

\begin{lemma} \label{lem:main:kappagrowing:P0exists}    
Under the setting of Theorem~\ref{thm:main:main}, suppose that $\kappa = \omega(1)$ and that $\kappa \ge 3d$ for all suitably large $n$.
Define $\mLambda = \diag(\lambda_1,\lambda_2,\dots,\lambda_d)$ with
\begin{equation} \label{eq:def:partii:lambdas}
\lambda_1 \leq \frac{ n }{ 3 }
~~~\text{ and }~~~
\lambda_j = \frac{ \lambda_1 }{ \kappa }
~~~\text{ for } 2 \le j \le d.
\end{equation}
There exist matrices $\mU_0 \in \Stiefel_d(\R^n)$ such that $\mP_0 = \mU_0 \mLambda^{1/2}\mI_{p,q}\mLambda^{1/2} \mU_0^T$ satisfies, for all suitably large $n$,
\begin{equation*}
    \frac{\lambda_1}{3n} \leq  \mP_{ij}^{(0)} \leq \frac{2}{3}
    ~~~\text{ for all }~~~i,j \in [n].
\end{equation*}
\end{lemma}
\begin{proof}
Define
\begin{equation} \label{eq:def:betad}
\beta_d
= \zeta_d 
    \sqrt{
    \frac{\lambda_1 (\lambda_d \wedge \log n) }{ n } },
\end{equation}
where $\zeta_d$ is a quantity depending on $n$ (via dependence on $d$) that we will specify below.
For rows $i=1,2,\dots,2^{k_0}$ and columns $j=2,3,\dots,d$, we take
\begin{equation*}
\mU^{(0)}_{ij}
=
\frac{ \beta_d h_{i,j} }{ \lambda_j }.
\end{equation*}
Letting $M_d = \lfloor n/2^{k_0+1} \rfloor \ge 2$, we take the next $2^{k_0}(M_d-1)$ rows of $\mU_0$ to be, for $i=2^{k_0}+1,2^{k_0}+2,\dots,2^{k_0} M_d$ and $j=2,3,\dots,d$,
\begin{equation*}
\mU^{(0)}_{ij}
= \frac{ \eta_d ~h_{i^*, j}}{ \sqrt{n} }
\end{equation*}
where $i^* = (i \mod 2^{k_0})$ and $\eta_d$ is a quantity, possibly dependent on $n$, to be specified below.
Finally, for $i > 2^{k_0} M_d$ and $j \ge 2$, we take $\mU^{(0)}_{ij} = 0$, so that
\begin{equation} \label{eq:part-ii-U0}
\mU_0 = \begin{bmatrix}
    \frac{1}{\sqrt{n}} & \frac{\beta_d h_{1, 2}}{ \lambda_2 } & \ldots & \frac{\beta_d h_{1, d-1}}{\lambda_{d-1}} & \frac{\beta_d h_{1, d}}{\lambda_d} \\
    \vdots & \vdots & \ddots & \vdots &\vdots \\
    \frac{1}{\sqrt{n}} & \frac{\beta_d h_{2^{k_0}, 2}}{\lambda_2} & \ldots & \frac{\beta_d h_{2^{k_0}, d-1}}{\lambda_{d-1}} & \frac{\beta_d h_{2^{k_0}, d}}{\lambda_d} \\

    \frac{1}{\sqrt{n}} & \frac{\eta_d h_{1,2} }{\sqrt{n}} & \ldots & \frac{\eta_d h_{1, d-1}}{\sqrt{n}} & \frac{\eta_d h_{1, d}}{\sqrt{n}}  \\
    \vdots & \vdots & \ldots & \vdots &\vdots \\
    \frac{1}{\sqrt{n}} & \frac{\eta_d h_{2^{k_0}, 2}}{\sqrt{n}} & \ldots & \frac{\eta_d h_{2^{k_0}, d-1}}{\sqrt{n}} & \frac{\eta_d h_{2^{k_0}, d}}{\sqrt{n}} \\
    \vdots & \vdots & \ddots & \vdots &\vdots \\
    \frac{1}{\sqrt{n}} & 0& \ldots & 0  & 0\\
    \frac{1}{\sqrt{n}} & 0  & \ldots & 0 & 0\\
    \vdots & \vdots & \ldots & \vdots &\vdots \\
    \frac{1}{\sqrt{n}} & 0 & \ldots & 0 & 0\\
\end{bmatrix}.
\end{equation}

To ensure that $\mU_0$ has orthonormal columns, we require for every $2\leq j \leq d$, 
\begin{equation} \label{eq:ii:dnorm}
    2^{k_0} \left(\frac{\beta_d^2}{\lambda_j^2} + (M_d - 1) \frac{\eta^2_d}{n}\right) = 1.
\end{equation} 
Plugging in the definition of $\beta_d$ from Equation~\eqref{eq:def:betad} and under our assumption that $\kappa = o(n)$, we require that 
\begin{equation} \label{eq:ii:norm-require}
2^{k_0} \zeta_d^2 \frac{\kappa}{n}\cdot\frac{\lambda_d \wedge \log n}{\lambda_d}
= o(1).
\end{equation}
Since $2^{k_0} = \Theta(d)$, Equation~\eqref{eq:ii:norm-require} holds when we take $\zeta_d = c/\sqrt{d}$ for any constant $c \leq \sqrt{\frac{1}{640}}$.
We then pick $\eta_d$ in such a way that $\eta_d = \sqrt{2} + o(1)$, so that
\begin{equation*}
\eta_d^2
= \frac{n}{2^{k_0}(M_d-1)}
        - o(1) =  \frac{n}{\lfloor n/2 \rfloor - 2^{k_0}} - o(1) = 2 + o(1),
\end{equation*}
ensuring that Equation~\eqref{eq:ii:dnorm} holds.

Unrolling the definition of $\mP_0 = \mU_0 \mLambda^{1/2}\mI_{p,q}\mLambda^{1/2} \mU_0^T$, it holds for all $1\leq i, j \leq 2^{k_0}$,
\begin{equation*} \begin{aligned}
    \mP^{(0)}_{ij} &\leq \sum_{\ell=1}^d \lambda_{k} \mU_{ik}^{(0)}\mU_{jk}^{(0)} = \frac{\lambda_1}{n} + (d-1)\frac{\beta_d^2 h_{i,k}h_{j,k}}{\lambda_d} \leq \frac{\lambda_1}{n} + (d-1) \frac{\beta_d^2}{\lambda_d} \\
    &= \frac{\lambda_1}{n} + (d-1) \zeta_d^2\frac{\lambda_1(\lambda_d\wedge \log n)}{\lambda_d n} \leq (1+(d-1) \zeta_d^2)\frac{\lambda_1}{n} \\
    &\leq (1 + c^2)\frac{\lambda_1}{n} \leq \frac{2}{3},
\end{aligned} \end{equation*}
where the last inequality holds by our choice of $\lambda_1 \leq n/3$ and any $0 < c \leq 1$.
To lower-bound the entries of $\mP^{(0)}$, observe that
\begin{equation*}
\begin{aligned}
    \mP_{ij}^{(0)} &\geq \frac{\lambda_1}{n} - (d-1) \zeta_d^2\frac{\lambda_1(\lambda_d\wedge \log n)}{\lambda_d n} \geq \frac{\lambda_1}{n} - (d-1) \zeta_d^2\frac{\lambda_1}{n} \\
    &\geq (1- c^2)\frac{\lambda_1}{n}. 
\end{aligned}
\end{equation*}
Choosing $c > 0$ sufficiently small, we have $\mP_{ij}^{(0)} \geq \lambda_1/3n$. 
Combining the above two displays, we conclude that
\begin{equation} \label{eq:P0bounds:block1-1}
\frac{ \lambda_1 }{ 3n }
\le \mP^{(0)}_{ij} 
\le \frac{2}{3}
~~~\text{ for }
~~~i, j \in [2^{k_0}].
\end{equation}

For $1\leq i \leq 2^{k_0} < j \leq 2^{k_0} M_d$, 
\begin{equation*}
    \begin{aligned}
    \mP^{(0)}_{ij} &\leq \frac{\lambda_1}{n} + (d-1)\frac{\beta_d\eta_d h_{i,k}h_{j,k}}{\sqrt{n}} \leq \frac{\lambda_1}{n} + (d-1) \frac{\beta_d \eta_d}{\sqrt{n}} \\
    &= \frac{\lambda_1}{n} + (d-1) \eta_d \zeta_d\frac{\sqrt{\lambda_1(\lambda_d\wedge \log n)}}{n} = \frac{\lambda_1}{n} + o(1) \leq \frac{2}{3},
\end{aligned}
\end{equation*}
where once again the last inequality holds for $n$ suitably large by our choice of $\lambda_1 \leq n/3$.
To lower-bound $\mP^{(0)}$, we observe that
\begin{equation*}
    \begin{aligned}
    \mP^{(0)}_{ij} &\geq \frac{\lambda_1}{n} - (d-1)\frac{\beta_d\eta_d h_{i,k}h_{j,k}}{\sqrt{n}} \geq \frac{\lambda_1}{n} - (d-1) \eta_d \zeta_d\frac{\sqrt{\lambda_1(\lambda_d\wedge \log n)}}{n} \\ &= \frac{\lambda_1}{n} + o\left(\frac{\lambda_1}{n} \right)
    \geq \frac{\lambda_1}{3n},
\end{aligned}
\end{equation*}
using the fact that $\kappa = \omega(1)$.
Combining the above two displays, we have
\begin{equation} \label{eq:P0bounds:block1-2}
\frac{ \lambda_1 }{ 3n }
\le \mP^{(0)}_{ij} 
\le \frac{2}{3}
~~~\text{ for }
~~~1\leq i \leq 2^{k_0} < j \leq 2^{k_0} M_d.
\end{equation}

For $2^{k_0}<i\leq j\leq 2^{k_0} M_d$, since $\kappa \geq 3d$, $\eta_d^2 = 2 + o(1)$ and $\lambda_1 \leq n/3$, we have
\begin{equation*}
\mP_{ij}^{(0)}
\leq \frac{\lambda_1}{n} + (d-1)\frac{\eta_d^2\lambda_d}{n}
\leq \left(1+\frac{d \eta_d^2}{\kappa}\right) \frac{\lambda_1}{n} \leq \frac{2\lambda_1}{n} \leq \frac{2}{3}
\end{equation*}
and
\begin{equation*}
    \mP_{ij}^{(0)} \geq \frac{\lambda_1}{n} - (d-1)\frac{\eta_d^2\lambda_d}{n} \geq  \left(1-\frac{d \eta_d^2}{\kappa}\right) \frac{\lambda_1}{n} \geq \frac{\lambda_1}{3n}. 
\end{equation*}
Combining the above two displays, we have
\begin{equation} \label{eq:P0bounds:block2-2}
\frac{ \lambda_1 }{ 3n }
\le \mP^{(0)}_{ij} 
\le \frac{2}{3}
~~~\text{ for }
~~~2^{k_0}<i\leq j\leq 2^{k_0} M_d.
\end{equation}

Finally, for $i > 2^{k_0} M_d$ and $j \in [n]$, $\mP^0_{ij} = \lambda_1/n$, since $\lambda_1 \leq n/3$, we again have 
\begin{equation} \label{eq:P0bounds:block-others}
\frac{ \lambda_1 }{ 3n }
\le \mP^{(0)}_{ij} 
\le \frac{2}{3}
~~~\text{ for }
~~~i>2^{k_0} M_d, ~j \in [n].
\end{equation}

Thus, combining Equations~\eqref{eq:P0bounds:block1-1} through \eqref{eq:P0bounds:block-others}, we have for sufficiently large $n$,  
\begin{equation*}
    \frac{\lambda_1}{3n} \leq  \mP_{ij}^{(0)} \leq \frac{2}{3}
    ~~~\text{ for all }~~~i,j \in [n],
\end{equation*}
completing the proof.
\end{proof}

As a remark, noting that we can also take $\lambda_1 \leq \frac{n}{2+\epsilon}$, so the condition $(2+\epsilon)\kappa\lambda_d \leq n$ would be suffice for our proof. The condition $\kappa \geq 3d$ can also be relaxed to $\kappa \geq (1+\epsilon) d$ for any constant $\epsilon > 0$. In order to achieve this, we would take $2^{k_0}M_d = \lfloor n / (1+\epsilon/2) \rfloor$, and we have $\eta_d^2 = 1+\epsilon/2 - o(1)$ in this case. Repeating the previous steps of unrolling the definition of $\mP_0$, we would be able to show that $\frac{\epsilon\lambda_1}{(2+2\epsilon)n} \leq \mP^{(0)}_{ij} \leq 2/(2+\epsilon)$ for $n$ sufficiently large. Furthermore, following the proof in Lemma \ref{lem:kappagrowing:packing} and \ref{subsec:apx:kappagrowing:KL}, we can construct a packing set with $\lfloor \epsilon n\rfloor$ instances, so the rest of the proof also goes through with a more careful analysis. We omit the details.  

To construct the rest of our packing set, $\{ \mU_i : i =1,2,\dots,\lfloor n/2 \rfloor \}$, we construct the $i$-th element $\mU_i$ by swapping the first row of $\mU_0$ with the $(i + \lfloor n/2 \rfloor)$-th row of $\mU_0$.
That is, for $i \in [ \lfloor n/2 \rfloor ]$, $\mU_i$ is the same as $\mU_0$ except in its first and $i + \lfloor n/2 \rfloor)$-th rows.
Lemma~\ref{lem:kappagrowing:packing} lower bounds the distance between the elements of our packing set $\mU_0, \mU_1, \dots, \mU_{\lfloor n/2 \rfloor}$ .

\begin{lemma} \label{lem:kappagrowing:packing}
Let $\mU_0 \in \Stiefel_d(\R^n)$ and $\mLambda = \diag(\lambda_1,\lambda_2,\dots,\lambda_d) \in \R^{d \times d}$ be the matrices guaranteed by Lemma~\ref{lem:main:kappagrowing:P0exists}.
There exists a collection $\{ \mU_i : i =1,2,\dots,\lfloor n/2 \rfloor \} \subset \Stiefel_d(\R^n)$, such that for any pair of indexes $0\leq i < j \leq \lfloor n/2 \rfloor$, we have
\begin{equation} \label{eq:ii:packing}
\min_{\mW \in \Od \cap \Opq}
\left\|\mU_i \mLambda^{1/2} \mW - \mU_j \mLambda^{1/2}\right\|_{2, \infty}
\geq \frac{\zeta_d \sqrt{d-1}}{2}\sqrt{\frac{\kappa(\lambda_d \wedge \log n)}{n}},
\end{equation}
where $\zeta_d$ is any quantity such that $\zeta_d \le 1/\sqrt{640 d}$.
\end{lemma}
\begin{proof}
Recalling the definition of $\mU_0$ from Equation~\eqref{eq:part-ii-U0}, for each $i \in [ \lfloor n/2 \rfloor ]$, define $\mU_i \in \Stiefel_d(\R^n)$ to have the same rows as $\mU_0$, but switching the the first and $(i+\lfloor n/2 \rfloor)$-th rows of $\mU_0$.

Define the vectors
\begin{equation*} \begin{aligned}
\vy_0 &= \left( \frac{1}{\sqrt{n}}, 0, \dots, 0 \right)^T \in \R^d \\
&\text{and} \\
\vy_1 &= \left(
    \frac{1}{\sqrt{n}},
    \frac{ \beta_d h_{1,2} }{ \lambda_2 },
    \frac{ \beta_d h_{1,3} }{ \lambda_3 },
    \dots ,
    \frac{ \beta_d h_{1,d} }{ \lambda_d }
\right)^T \in \R^d,
\end{aligned} \end{equation*}
noting that $\vy_1 \in \R^d$ is the first row of $\mU_0$.
Fix a matrix $\mW \in \Opq \cap \Od$.
Trivially lower-bounding the maximum in the $(\tti)$-norm by the maximum of two particular rows and making use of the construction of $\mU_i$ and $\mU_j$ for any distinct $i,j \in \{0,1,2,\dots,\lfloor n/2 \rfloor \}$, we have
\begin{equation} \label{eq:ii:pack-bound}
\begin{aligned}
&\left\|\mU_i \mLambda^{1/2} \mW - \mU_j \mLambda^{1/2}\right\|_{2, \infty} \\
&~~~~~~\ge
\max\left\{
  \| \mW^T \mLambda^{1/2} \vy_1 - \mLambda^{1/2} \vy_0 \|_2,
  \| \mW^T \mLambda^{1/2} \vy_0 - \mLambda^{1/2} \vy_0 \|_2 \right\}.
\end{aligned} \end{equation}
Suppose that
\begin{equation} \label{eq:supposition}
\| \mW^T \mLambda^{1/2} \vy_0 - \mLambda^{1/2} \vy_0 \|_2
\ge
\frac{\zeta_d \sqrt{d-1} }{2}
\sqrt{\frac{\kappa(\log n \wedge \lambda_d)}{n}}.
\end{equation}
Then it holds trivially that
\begin{equation*}
\left\|\mU_i \mLambda^{1/2} \mW - \mU_j \mLambda^{1/2}\right\|_{2, \infty}
\ge
\frac{\zeta_d \sqrt{d-1} }{2}
\sqrt{\frac{\kappa(\log n \wedge \lambda_d)}{n}}.
\end{equation*}
If, on the other hand, Equation~\eqref{eq:supposition} does not hold, the triangle inequality implies
\begin{equation} \label{eq:trilb} \begin{aligned}
\| \mW^T \mLambda^{1/2} \vy_1 - \mLambda^{1/2} \vy_0 \|_2
&\ge
\| \mW^T \mLambda^{1/2} \vy_1 - \mW^T \mLambda^{1/2} \vy_0 \|_2 \\
&~~~~~~- \| \mW^T \mLambda^{1/2} \vy_0 - \mLambda^{1/2} \vy_0 \|_2 \\
&= \| \mLambda^{1/2} \vy_1 - \mLambda^{1/2} \vy_0 \|_2
- \| \mW^T \mLambda^{1/2} \vy_0 - \mLambda^{1/2} \vy_0 \|_2.
\end{aligned} \end{equation}
Plugging in the definitions of $\vy_0$ and $\vy_1$ and using the fact that $\lambda_2=\lambda_3=\cdots=\lambda_d$,
\begin{equation*}
\| \mLambda^{1/2} \vy_1 - \mLambda^{1/2} \vy_0 \|_2^2
= \sum_{j=2}^d \frac{ \zeta_d^2 \lambda_1 (\lambda_d \wedge \log n) }{ \lambda_j n}
= \frac{ (d-1) \zeta_d^2 \lambda_1 (\lambda_d \wedge \log n) }
{ \lambda_d n}.               
\end{equation*}
Further, since Equation~\eqref{eq:supposition} fails to hold by assumption,
\begin{equation*}
\left\| \mW^T \mLambda^{1/2} \vy_0 - \mLambda^{1/2} \vy_0 \right\|_2^2
\le \frac{ (d-1) \zeta_d^2 }{ 4\lambda_d }
\frac{\lambda_1 (\log n \wedge \lambda_d)}
{n}.
\end{equation*}
Taking square roots and applying the above two displays to Equation~\eqref{eq:trilb},
\begin{equation*}
\| \mW^T \mLambda^{1/2} \vy_1 - \mLambda^{1/2} \vy_0 \|_2
\ge
\frac{ \zeta_d \sqrt{d-1} }{ 2 }
\sqrt{\frac{\lambda_1 (\log n \wedge \lambda_d)}
{ n \lambda_d } },
\end{equation*}
so that
\begin{equation*}
\left\|\mU_i \mLambda^{1/2} \mW - \mU_j \mLambda^{1/2}\right\|_{2, \infty}
\ge
\frac{\sqrt{d-1} \zeta_d}{2}
\sqrt{\frac{\kappa(\log n \wedge \lambda_d)}{n}}.
\end{equation*}
Note that we have shown this bound to hold whether Equation~\eqref{eq:supposition} holds or not.
Since the right-hand side of this bound does not depend on $\mW$, minimizing over $\mW \in \Opq \cap \Od$ completes the proof.
\end{proof}


\subsection{Bounding the KL Divergence} \label{subsec:apx:kappagrowing:KL}

Now, we proceed to control the KL divergence among the parameters $\mU_0, \mU_1, \ldots,$ $\mU_{\lfloor n/2 \rfloor}$.
To do this, we must again ensure that the Frobenius norms between different probability matrices are small in order to apply Lemma~\ref{lem:Zhou2021}.

\begin{lemma}  \label{lem:kappagrowing:KL}
Under the conditions of Theorem~\ref{thm:main:main}, suppose that $\kappa = \omega(1)$.
Let $\mU_0 \in \Stiefel_d(\R^n)$ and $\mLambda = \diag(\lambda_1,\lambda_2,\dots,\lambda_d) \in \R^{d \times d}$ be the matrices guaranteed by Lemma~\ref{lem:main:kappagrowing:P0exists} and let $\{ \mU_i : i \in [ \lfloor n/2 \rfloor ]$ be the packing set guaranteed by Lemma~\ref{lem:kappagrowing:packing}.
For any $1 \leq i \leq \lfloor n/2 \rfloor$, we have
\begin{equation} \label{eq:ii:KL}
        \left \| \mU_i \tmLambda \mU_i^T - \mU_0 \tmLambda \mU_0^T\right\|_F^2
        \leq \frac{1}{80} \frac{\lambda_1 (\lambda_d \wedge \log n)}{n}. 
\end{equation}
\end{lemma}
\begin{proof}
For $i \in [ \lfloor n/2 \rfloor ]$.
Adding and subtracting appropriate quantities and applying the triangle inequality,
\begin{equation*} 
\left \| \mU_i \tmLambda \mU_i^T - \mU_0 \tmLambda \mU_0^T\right\|_F
\le
\left\| \mU_i \tmLambda (\mU_i-\mU_0)^T \right\|_F
+ \left\| (\mU_i-\mU_0) \tmLambda \mU_0^T \right\|_F .
\end{equation*}
Since $\mU_0$ and $\mU_i$ are $d$-frames, basic properties of the Frobenius norm imply
\begin{equation} \label{eq:easierbound}
\left \| \mU_i \tmLambda \mU_i^T - \mU_0 \tmLambda \mU_0^T\right\|_F
\le
2\left\| (\mU_i-\mU_0)\tmLambda \right\|_F.
\end{equation}
We observe that for $i\in [\lfloor n / 2\rfloor]$, $\mU_i$ and $\mU_0$ differ in exactly two rows.
Define, as in the proof of Lemma~\ref{lem:kappagrowing:packing},
\begin{equation*} \begin{aligned}
\vy_0 &= \left( \frac{1}{\sqrt{n}}, 0, \dots, 0 \right)^T \in \R^d \\
&\text{and} \\
\vy_1 &= \left(
    \frac{1}{\sqrt{n}},
    \frac{ \beta_d h_{1,2} }{ \lambda_2 },
    \frac{ \beta_d h_{1,3} }{ \lambda_3 },
    \dots ,
    \frac{ \beta_d h_{1,d} }{ \lambda_d }
\right)^T \in \R^d,
\end{aligned} \end{equation*}
noting that $\vy_1 \in \R^d$ is the first row of $\mU_0$ by construction.

The structure of $\mU_0$ and $\mU_i$ is such that $\mU_i-\mU_0$ has all rows equal to zero except for two, so that
\begin{equation} \label{eq:frob:rowbound}
 \left\| (\mU_i-\mU_0)\tmLambda \right\|_F^2
= 2\| \tmLambda ( \vy_1 - \vy_0 ) \|_2^2.
\end{equation}
Plugging in the definitions of $\tmLambda$, $\vy_0$ and $\vy_1$,
\begin{equation*}
\left\| (\mU_i-\mU_0)\tmLambda \right\|_F^2
=  2 \zeta_d^2 (d-1)
        \frac{ \lambda_1 ( \lambda_d \wedge \log n ) }{n}.
\end{equation*}
Plugging this into Equation~\eqref{eq:frob:rowbound},
\begin{equation*}
\left\| (\mU_i-\mU_0)\tmLambda \right\|_F^2
= 2 \zeta_d^2 (d-1)
        \frac{ \lambda_1 ( \lambda_d \wedge \log n ) }{n}.
\end{equation*}
Applying this to Equation~\eqref{eq:easierbound}, we conclude that
\begin{equation*}
\left \| \mU_i \tmLambda \mU_i^T - \mU_0 \tmLambda \mU_0^T\right\|_F^2
\le 8 \zeta_d^2 (d-1)
        \frac{ \lambda_1 ( \lambda_d \wedge \log n ) }{n}.
\end{equation*}
Lemma~\ref{lem:kappagrowing:packing} guarantees $\zeta_d^2 \le 1/640 d < 1/640(d-1)$, completing the proof.
\end{proof}

\section{Technical Lemmas} \label{apx:technical}

Here we collect a number of technical lemmas related to our packing set constructions.

\begin{lemma} \label{lem:tech:eigen} 
For $i\in [2^{k_0}m]$, let $\mG_i$ be as defined in Equation~\eqref{eq:def:G}. Assume that $\|\vx_i\|_2 < 1$, then
 the singular values of $\mG_i$ are given by
\begin{equation*}
(  \sqrt{1 + \Tilde{\sigma}_{i+}}, \sqrt{1 + \Tilde{\sigma}_{i-}}, 1, \dots, 1 )^T \in \R^d,
\end{equation*}
where 
\begin{equation} \label{eq:def:sigmapm}
  \Tilde{\sigma}_{i\pm} = \vx^T_i \vu_i + \vx_i^T \vx_i \alpha_{i\pm}
\end{equation}
    and 
\begin{equation} \label{eq:roots}
    \alpha_{i \pm} = \frac{1}{2} \pm \frac{1}{2} \sqrt{1 + \frac{4\|\vu_i\|^2_2 + 4\vx^T_i \vu_i}{\|\vx_i\|_2^2}},
\end{equation}
Write the (reordered) right singular subspace matrix as $\tmV_i = (\Tilde{\calV}_i, \Tilde{\calV}^{\perp}_i)$,  where  $\Tilde{\calV}_i \in \R^{n \times 2}$ has as its columns the singular vectors corresponding to $\sqrt{1+\Tilde{\sigma}_{i\pm}}$.
Then $\Tilde{\calV}_i$ is given by
\begin{equation*} \begin{aligned}
        \Tilde{\calV}_i^T &= \begin{bmatrix}
        \frac{\alpha_{i+}}{\|\alpha_{i+} \vx_i + \vu_i\|_2} & \frac{1}{\|\alpha_{i+} \vx_i + \vu_i\|_2} \\
        \frac{\alpha_{i-}}{\|\alpha_{i-} \vx_i + \vu_i\|_2} & \frac{1}{\|\alpha_{i-} \vx_i + \vu_i\|_2} \\
    \end{bmatrix} \begin{bmatrix}
        \vx_i^T\\
        \vu_i^T
    \end{bmatrix} .
    \end{aligned} \end{equation*}
\end{lemma}
\begin{proof}
Fix $i \in [2^{k_0} m ]$.
Recalling that $\mG_i = \tmU_i \tmSig_i \tmV_i^T$ is the SVD of $\mG_i$,
\begin{equation*}
\tmV_i \tmSig_i^2 \tmV_i^T
=
\mG_i^T \mG_i
= \left(\mU_0 + \ve_i \vx_i^T\right)^T \left(\mU_0 + \ve_i \vx_i^T\right)
= \mI_d + \vx_i\vu_i^T + \vu_i \vx_i^T + \vx_i\vx_i^T.
\end{equation*}
We observe that for any vector $\vv\in \R^d$, if $\vv$ is orthogonal to both $\vx_i$ and $\vu_i$, we have $\mG_i^T \mG_i \vv = \vv$.
Since $\vu_i$ and $\vx_i$ are linearly independent in our construction, the subspace orthogonal to the span of $\vx_i$ and $\vu_i$ has dimension $d-2$, and it follows that $1$ appears as a singular value of $\mG_i$ with multiplicity $d-2$.
Now, suppose that $\vw = \alpha \vx_i + \vu_i$ is an eigenvector of $\mG_i^T \mG_i$, so that
\begin{equation*} 
\left(\mI_d + \vx_i \vu_i^T +\vu_i \vx_i^T + \vx_i \vx_i^T\right) \vw 
= \lambda \vw
\end{equation*}
for some $\lambda \in \R$.
One can verify that $\vw = \alpha_{i\pm} \vx_i + \vu_i$ and $\lambda = 1 + \Tilde{\sigma}_{i\pm}$ satisfy the above, with $\alpha_{i\pm}$ and $\Tilde{\sigma}_{i\pm}$ as given in Equations~\eqref{eq:roots} and~\eqref{eq:def:sigmapm}, respectively. Renormalizing appropriately yields the claimed value of $\Tilde{\calV}_i$. It remains to show that $1 + \Tilde{\sigma}_{i-} > 0$. Explicitly write down the equation for $\Tilde{\sigma}_{i-}$, we have 
\begin{equation*}
\Tilde{\sigma}_{i-} = \vx_i^T\vu_i + \frac{1}{2} \|\vx_i\|_2^2 - \frac{1}{2}\|\vx_i\|_2 \sqrt{\|\vx_i\|_2^2 + 4\vx_i^T\vu_i + 4 \|\vu_i\|_2^2}.
\end{equation*}
Since 
\begin{equation*}
\frac{1}{4} \|\vx_i\|_2^4 + \|\vx_i\|_2^2 \vx_i^T \vu_i + (\vx_i^T\vu_i)^2 \leq \frac{1}{4} \|\vx_i\|_2^4 + \|\vx_i\|_2^2 \vx_i^T \vu_i + \|\vx_i\|_2^2\|\vu_i\|_2^2,
\end{equation*}
it follows that $\Tilde{\sigma}_{i-} < 0$. Noting that
\begin{equation} \label{eq:sigma-:lowerbound} \begin{aligned}
\tilde{\sigma}_{i -}
&\geq \vx_i^T \vu_i+ \frac{1}{2} \vx_i^T \vx_i
	- \frac{1}{2} (\|\vx_i\|_2 + 2\|\vu_i\|_2) \\
&= \vx_i^T \vu_i - \|\vx_i\|_2 \|\vu_i\|_2 \\
&\geq - \|\vx_i\|_2 \|\vu_i\|_2.
\end{aligned} \end{equation}
From our construction, we have $\|\vu_i\|_2 \leq \sqrt{d/n-r}$. Since $\|\vx_i\|_2 < 1$ and we always have $d \leq n-r$, it follows that $1 + \Tilde{\sigma}_{i-} > 0$.
\end{proof}


\begin{lemma} \label{lem:eigenbound}
For any $i\in [2^{k_0}m]$, let $\mG_i$ be as defined in Equation~\eqref{eq:def:G}, with singular value decomposition $\mG_i = \tmU_i \tmSig_i \tmV_i^T$.
Recalling the vector $\vx_i \in \R^d$ from Lemma~\ref{lem:constructxi}, if $\|\vx_i\|_2 < 1$,
then
\begin{equation*}
\left\|\tmSig_i - \mI_d\right\|
\leq \frac{1}{2} \|\vx_i\|_2^2 + 2\sqrt{\frac{d}{n-r}}\|\vx_i\|_2. 
\end{equation*}
\end{lemma}
\begin{proof}
Notice that for $a \in [0, 1]$, we have 
\begin{equation}
\label{eq:sqrt-ineq}
    \sqrt{a} \geq \frac{a+1}{2}-\frac{(a-1)^2}{2}
\end{equation}
and that for any $b \in [0,1]$, substituting $a = 1-b$ into Equation~\eqref{eq:sqrt-ineq}, we have 
\begin{equation}
\label{eq:sqrt-ineq2}
\sqrt{1 - b} \geq 1 - \frac{b}{2} - \frac{b^2}{2}.
\end{equation}

Applying Lemma \ref{lem:tech:eigen} and then applying Equation~\eqref{eq:sqrt-ineq2} to $\tilde\sigma_{i-}$ we have
\begin{equation*}
\begin{aligned}
    \left\|\tmSig_i - \mI_d\right\|
&= \max \left\{\sqrt{1 + \Tilde{\sigma}_{i+}} - 1,
	1 - \sqrt{1 + \Tilde{\sigma}_{i-}}\right\}\\
&\leq \max \left\{ \frac{\Tilde{\sigma}_{i+}}{1+\sqrt{1 + \Tilde{\sigma}_{i+}}}, \frac{\tilde\sigma_{i-}^2}{2} - \frac{\tilde\sigma_{i-}}{2}\right\}\\
&\leq \frac{1}{2} \max \left\{\tilde\sigma_{i+}, \tilde\sigma_{i-}^2-\tilde\sigma_{i-}\right\}. 
\end{aligned}
\end{equation*}

From the proof of Lemma \ref{lem:tech:eigen}, we have $|\Tilde{\sigma}_{i-}| < 1$ if $\|\vx_i\|_2 < 1$ and thus, $\Tilde{\sigma}_{i-}^2 \leq - \Tilde{\sigma}_{i-}$.
Therefore,
\begin{equation} \label{eq:opnorm:checkpt}
\left\|\tmSig_i - \mI_d\right\|
\leq \frac{1}{2} \max \left\{\tilde\sigma_{i+},
			\tilde\sigma_{i-}^2-\tilde\sigma_{i-}\right\}
\le \frac{1}{2} \left( \tilde\sigma_{i+}
		+ \tilde\sigma_{i-}^2-\tilde\sigma_{i-} \right)
\le \frac{1}{2} \tilde\sigma_{i+} - \tilde\sigma_{i-}.
\end{equation}
Since
\begin{equation} \label{eq:bound:sigmaplus} \begin{aligned}
\tilde{\sigma}_{i +}
&=\vx_i^T \vu_i+\alpha_{i +} \vx_i^T \vx_i \\
&= \vx_i^T \vu_i + \frac{1}{2} \vx_i^T \vx_i
	+ \frac{1}{2} \|\vx_i\|_2 \sqrt{\vx_i^T\vx_i
	+ 4 \vx_i^T \vu_i + 4\|\vu_i\|_2^2}\\
&\leq 2\left\|\vx_i\right\|_2\left\|\vu_i\right\|_2+\left\|\vx_i\right\|_2^2,
\end{aligned} \end{equation}
applying this and Equation~\eqref{eq:sigma-:lowerbound} to Equation~\eqref{eq:opnorm:checkpt}, using the fact that $\vx_i^T\vu_i \ge 0$ and recalling the construction of $\vu_i$ from Equation~\ref{eq:par1:base_U}, we conclude that
\begin{equation*}
\left\|\tmSig_i - \mI_d\right\|
\leq 2\left\|\vx_i\right\|_2\left\|\vu_i\right\|_2
	+\frac{1}{2}\left\|\vx_i\right\|_2^2
= \frac{1}{2} \|\vx_i\|_2^2 + 2\sqrt{\frac{d}{n-r}}\|\vx_i\|_2,
\end{equation*}
completing the proof.
\end{proof}


\begin{lemma} \label{lem:Utilde}
For any $i\in [2^{k_0}m]$, recall the vector $\vx_i \in \R^d$ from Lemma~\ref{lem:constructxi} and the matrix $\mG_i$ from Equation~\eqref{eq:def:G}.
If $\|\vx_i\|_2 < 1$, then, with $\tmU_i$ as defined in 
\begin{equation*}
\left\| \tmU_i \right\|_{\tti}
\leq \sqrt{\frac{d}{n-r} + \frac{\|\vx_i\|_2}{1 - \|\vx_i\|_2}},
\end{equation*}
where $\tmU_i \in \R^{n \times d}$ is the left singular subspace of $\mU_0 + \ve_i\vx_i^T \in \R^{n \times d}$. 
\end{lemma}
\begin{proof}
For any $i \in [2^{k_0} m]$,
\begin{equation*} \begin{aligned}
\left| \| \tmU_i \|_{\tti}^2 - \| \mU_0 \|_{\tti}^2 \right|
&=
\left|
\max_{\ell \in [n]}
	\left(\tmU_i \tmU_i^T\right)_{\ell \ell}
- 
\max_{\ell^\prime \in [n]}
 \left(\mU_0 \mU_0^T\right)_{\ell^\prime \ell^\prime}
\right| \\
&\le \max_{\ell \in [n]}
	\left| \left( \tmU_i \tmU_i^T - \mU_0 \mU_0^T \right)_{\ell \ell}
	\right| \\
&= \left\| \tmU_i \tmU_i^T - \mU_0 \mU_0^T \right\|.
\end{aligned} \end{equation*}
Noting that $\| \mU_0 \|_{\tti} \leq \sqrt{ d/(n-r) }$ by construction, it follows that
\begin{equation} \label{eq:tmUi:checkpt}
\|\tmU_i\|_{\tti}^2
\le 
\|\mU_0\|_{\tti}^2
+ \left| \| \tmU_i \|_{\tti}^2  - \| \mU_0 \|_{\tti}^2 \right|
\leq  \frac{d}{n-r}
  + \left\| \mU_0 \mU_0^T - \tmU_i \tmU_i^T \right\|.
\end{equation}
We then apply Wedin's $\sin \Theta$ theorem \citep[see Theorem 2.9 in][]{chen2021spectral} to the left singular subspace of $\mU_0$ and $\mU_0 + \ve_i \vx_i^T$.
Denote the $d$-th singular value of $\mU_0$ as $\sigma_d(\mU_0)$ and the $(d+1)$-th singular value as $\sigma_{d+1}(\mU_0)$.
From our assumption that $\|\vx_i\|_2\leq 1$ and the fact that $\|\ve_i \vx_i^T\| = \|\vx_i\|_2$, we have
\begin{equation*}
\begin{aligned}
\left\|\mU_0 \mU_0^T - \tmU_i \tmU_i^T\right\|
&\leq \frac{\|\ve_i \vx_i^T\| }
	{\sigma_d(\mU_0) - \sigma_{d+1}(\mU_0) - \|\ve_i \vx_i^T\| }\\
    &= \frac{\|\vx_i\|_2}{1 - \|\vx_i\|_2}. 
\end{aligned}
\end{equation*}
Applying this bound to Equation~\eqref{eq:tmUi:checkpt} and taking square roots,
\begin{equation*}
\left\|\tmU_i\right\|_{\tti}
\leq \sqrt{\frac{d}{n-r} + \frac{\|\vx_i\|_2}{1 - \|\vx_i\|_2}},
\end{equation*}
which completes the proof.
\end{proof}


\begin{lemma} \label{lem:tech:refined_eigen_bound}
For any $i\in [2^{k_0}m]$, let $\mG_i$ be as defined in Equation~\eqref{eq:def:G}, with singular value decomposition $\mG_i = \tmU_i \tmSig_i \tmV_i^T$. If $\|\vx_i\|_2 < 1$, then writing $\tmLambda = \mLambda^{1/2} \mI_{p,q} \mLambda^{1/2}$,
\begin{equation*}
\left\|\left(\tmSig_i - \mI_d\right)\tmV_i^T \tmLambda \right\|^2_F
\le 17\left( \left\|\vu_i\right\|_2^2 +  \left\|\vx_i\right\|_2^2 \right)
	\left( \vx_i^T\mLambda^2 \vx_i + \vu_i^T \mLambda^2 \vu_i \right).
\end{equation*}
\end{lemma}
\begin{proof}
We first note that Lemma~\ref{lem:tech:eigen} ensures that
\begin{equation*}
        \tmSig_i - \mI_d
	= \diag\left(
        	\sqrt{1 + \Tilde{\sigma}_{i+}} - 1,
        	\sqrt{1 + \Tilde{\sigma}_{i-}} - 1,
		0, \dots, 0 \right) \in \R^{d \times d}
\end{equation*}
where $ \Tilde{\sigma}_{i\pm}$ are defined in Equation~\eqref{eq:def:sigmapm} and we have $\Tilde{\sigma}_{i-} < 0 < \Tilde{\sigma}_{i+}$.
Recalling $\alpha_{i\pm}$ as defined in Equation~\eqref{eq:roots}, Lemma~\ref{lem:tech:eigen} further implies that $\Tilde{\calV}_i \in \R^{n \times 2}$, given by
\begin{equation*} \begin{aligned}
    \Tilde{\calV}_i^T &= \begin{bmatrix}
    \frac{\alpha_{i+}}{\|\alpha_{i+} \vx_i + \vu_i\|_2} & \frac{1}{\|\alpha_{i+} \vx_i + \vu_i\|_2} \\
    \frac{\alpha_{i-}}{\|\alpha_{i-} \vx_i + \vu_i\|_2} & \frac{1}{\|\alpha_{i-} \vx_i + \vu_i\|_2} 
\end{bmatrix} \begin{bmatrix}
    \vx_i^T\\
    \vu_i^T
\end{bmatrix},
\end{aligned} \end{equation*}
encodes the singular vectors of $\mG_i$ corresponding to $\sqrt{1+\Tilde{\sigma}_{i\pm}}$.
Defining the quantities
\begin{equation} \label{eq:def:dpm}
d_{i\pm} = \sqrt{1 + \Tilde{\sigma}_{i\pm}} - 1,
\end{equation}
and defining
\begin{equation} \label{eq:def:mA}
\mA_i = \begin{bmatrix}
        \frac{\alpha_{i+}d_{i+}}{\|\alpha_{i+} \vx_i + \vu_i\|_2} & \frac{d_{i+}}{\|\alpha_{i+} \vx_i + \vu_i\|_2} \\
        \frac{\alpha_{i-}d_{i-}}{\|\alpha_{i-} \vx_i + \vu_i\|_2} & \frac{d_{i-}}{\|\alpha_{i-} \vx_i + \vu_i\|_2} \\
    \end{bmatrix},
\end{equation}
we have
\begin{equation*} 
\left\|\left(\tmSig_i - \mI_d\right)\tmV_i^T \tmLambda\right\|_F^2
= \left\|\mA_i \begin{bmatrix}
            \vx_i^T \tmLambda \\
            \vu_i^T \tmLambda
        \end{bmatrix}\right\|_F^2
= \tr \left(\mA_i^T \mA_i \begin{bmatrix}
            \vx_i^T \mLambda^2 \vx_i & \vx_i^T \mLambda^2 \vu_i \\
            \vx_i^T \mLambda^2 \vu_i & \vu_i^T \mLambda^2 \vu_i
        \end{bmatrix}\right).
\end{equation*}
Applying our definition of $\mA_i$ from Equation~\eqref{eq:def:mA},
\begin{equation} \label{eq:frobexpansion} \begin{aligned}
\left\|\left(\tmSig_i - \mI_d\right)\tmV_i^T \tmLambda\right\|_F^2
&=
\left(\frac{\alpha_{i+}^2 d_{i+}^2}{\|\alpha_{i+} \vx_i + \vu_{i}\|_2^2} + \frac{\alpha_{i-}^2 d_{i-}^2}{\|\alpha_{i-}\vx_i + \vu_{i}\|_2^2}\right) \vx_i^T\mLambda^2 \vx_i \\
&~~~~~~+ 2\left(\frac{\alpha_{i+} d_{i+}^2}{\|\alpha_{i+} \vx_i + \vu_{i}\|_2^2} + \frac{\alpha_{i-} d_{i-}^2}{\|\alpha_{i-} \vx_i + \vu_{i}\|_2^2}\right) \vx_i^T \mLambda^2 \vu_i \\
&~~~~~~+ \left(\frac{d_{i+}^2}{\|\alpha_{i+} \vx_i + \vu_{i}\|_2^2} + \frac{d_{i2}^2}{\|\alpha_{i-} \vx_i + \vu_{i}\|_2^2} \right) \vu_i^T \mLambda^2 \vu_i.
\end{aligned} \end{equation}

Our proof will be complete once we establish an upper bound on the $\alpha_{i\pm}$ and $d_{i\pm}$ terms and a lower-bound on $\| \alpha_{i\pm} \vx_i + \vu_i \|$.
Toward this end, rearranging the definition of $\alpha_{i+}$ and applying the triangle inequality,
\begin{equation} \label{eq:alphai+:ub}
\alpha_{i+}
= \frac{1}{2}
 + \frac{ \sqrt{ \|\vx_i\|_2^2 + 4 \vx_i^T \vu_i + 4 \| \vu_i \|^2_2 } }
	{ 2\|\vx_i\|_2}
= \frac{1}{2} + \frac{ \| \vx_i + 2\vu_i \|_2 }{2\|\vx_i\|_2}  \\
\le 1 + \frac{  \| \vu_i \|_2 }{ \|\vx_i\|_2 }.
\end{equation}
A similar argument yields
\begin{equation} \label{eq:alphai-:ub}
|\alpha_{i -}| \leq  \frac{\left\|\vu_i\right\|_2}{\left\|\vx_i\right\|_2}.
\end{equation}
Observing that the function $z \mapsto \sqrt{1+z}-1$ is upper-bounded by $z/2$ for $z \ge 0$, and recalling our definition of $d_{i\pm}$ from Equation~\eqref{eq:def:dpm} above,
Equation~\eqref{eq:bound:sigmaplus} (established in the proof of Lemma~\ref{lem:eigenbound}) implies
\begin{equation} \label{eq:di+:ub}
d_{i+}
\le \frac{  \Tilde{\sigma}_{i+} }{ 2 }
\leq \left\|\vx_i\right\|_2 \left( \left\|\vu_i\right\|_2
	+\frac{1}{2}\left\|\vx_i\right\|_2 \right).
\end{equation}
From the proof of Lemma \ref{lem:tech:eigen} and \ref{lem:eigenbound}, as long as $\|\vx_i\|_2 < 1$, we have 
\begin{equation} \label{eq:di-:ub}
\begin{aligned}
    |d_{i-}| &= 1 - \sqrt{1+\Tilde{\sigma}_{i-}} \\
    &\leq -\Tilde{\sigma}_{i-}\\
    &\leq \left\|\vx_i\right\|_2\left\|\vu_i\right\|_2. 
\end{aligned}
\end{equation}

To control the entries of $\mA$, we must also control the denominator terms,
\begin{equation*}
\left\|\alpha_{i +} \vx_i+\vu_i\right\|_2
~~~\text{ and }~~~ \left\|\alpha_{i -} \vx_i+\vu_i\right\|_2.
\end{equation*}
Expanding the square and using the fact that by construction from Equation~\eqref{eq:roots}, $\alpha_{i+} \ge 1$ and $\vx_i^T \vu_i \ge 0$,
\begin{equation} \label{eq:mA:denombound+}
\left\|\alpha_{i +} \vx_i + \vu_i\right\|_2^2
=\left\|\vu_i\right\|_2^2 + 2 \alpha_{i +} \vx_i^T \vu_i
	+\alpha_{i +}^2 \vx_i^T \vx_i \\
\geq \left\|\vu_i\right\|_2^2 + \left\|\vx_i\right\|_2^2.
\end{equation}

Expanding the definition of $\alpha_{i-}$ and rearranging,
\begin{equation} \label{eq:alpha+xu:lb}
\alpha_{i-}^2 \| \vx_i^T \|_2^2
= \left( \frac{1}{2}
		- \frac{ \| \vx_i + 2 \vu_i \|_2 }{ 2 \| \vx_i \|_2 }
	\right)^2 \| \vx_i \|_2^2 
= \frac{1}{4} \left( \| \vx_i \|_2 - \| \vx_i + 2 \vu_i \|_2 \right)^2.
\end{equation}

Again expanding the definition of $\alpha_{i-}$,
\begin{equation*}
2 \alpha_{i-} \vx_i^T \vu_i
= \left( 1 - \frac{ \| \vx_i + 2\vu_i \|_2 }{ \| \vx \|_2 } \right)
	\vx_i^T \vu_i
= \frac{ \vx_i^T \vu_i }{ \| \vx_i \|_2 }
	\left( \| \vx_i \|_2 - \| \vx_i + 2\vu_i \|_2 \right)
\end{equation*}
and it follows that
\begin{equation*} \begin{aligned}
\| \alpha_{i-} \vx_i + \vu_i \|_2^2
&= \| \alpha_{i-} \vx_i \|_2^2
	+ 2 \alpha_{i-} \vx_i^T \vu_i + \| \vu_i \|_2^2 \\
&= \frac{1}{4}\left( \| \vx_i \|_2 - \| \vx_i + 2 \vu_i \|_2 \right)^2
	+ \frac{ \vx_i^T \vu_i }{ \| \vx_i \|_2 }
     \left( \| \vx_i \|_2 - \| \vx_i + 2\vu_i \|_2 \right) + \| \vu_i \|_2^2.
\end{aligned} \end{equation*}
Using non-negativity of the square and the reverse triangle inequality,
\begin{equation*}
\| \alpha_{i-} \vx_i + \vu_i \|_2^2
\ge \| \vu_i \|_2^2
	-  \frac{ \vx_i^T \vu_i  \| \vu_i \|_2 }{ \| \vx_i \|_2 }
= \| \vu_i \|_2^2
  \left( 1 - \frac{ \vx_i^T \vu_i }{ \| \vu_i \|_2  \| \vx_i \|_2 } \right).
\end{equation*}
By construction, $\vx_i$ and $\vu_i$ obey Equation~\eqref{eq:main:cos}, from which
\begin{equation} \label{eq:alpha-xu:lb}
\| \alpha_{i-} \vx_i + \vu_i \|_2^2 \ge \frac{1}{8} \| \vu_i \|_2^2.
\end{equation}

Combining Equations~\eqref{eq:alphai+:ub},~\eqref{eq:di+:ub} and~\eqref{eq:mA:denombound+} and using the fact that $(a+b)^2 \le 2(a^2+b^2)$,
\begin{equation*}
\frac{ \left( \alpha_{i+} d_{i+} \right)^2 }
	{ \left\|\alpha_{i +} \vx_i + \vu_i\right\|_2^2 }
\le 
\frac{ \left( \|\vx_i\|_2 +  \| \vu_i \|_2 \right)^2 }
	{  \left\|\vu_i\right\|_2^2 + \left\|\vx_i\right\|_2^2 }
 \left( \left\|\vu_i\right\|_2
        +\frac{1}{2}\left\|\vx_i\right\|_2 \right)^2
\le 4  \| \vu_i \|_2^2 + \| \vx_i \|_2^2.
\end{equation*}

Similarly, combining Equations~\eqref{eq:alphai-:ub},~\eqref{eq:di-:ub} and~\eqref{eq:alpha-xu:lb},
\begin{equation*}
\frac{ \left| \alpha_{i-} d_{i-} \right|^2 }
	{ \left\|\alpha_{i -} \vx_i + \vu_i\right\|_2^2 }
\le
\frac{ 8 }{ \| \vu_i \|_2^2 }
\frac{\left\|\vu_i\right\|_2^2}{\left\|\vx_i\right\|_2^2}
\left\|\vx_i\right\|_2^2 \left\|\vu_i\right\|_2^2
\le 8\left\|\vu_i\right\|_2^2.
\end{equation*}
Combining the above two displays,
\begin{equation} \label{eq:quadratic:xterm}
\left(
  \frac{\alpha_{i+}^2 d_{i+}^2}{\|\alpha_{i+} \vx_i + \vu_{i}\|_2^2}
  + \frac{\alpha_{i-}^2 d_{i-}^2}{\|\alpha_{i-}\vx_i + \vu_{i}\|_2^2}
\right) \vx_i^T\mLambda^2 \vx_i
\le
\left( 12 \| \vu_i \|_2^2 + \| \vx_i \|_2^2 \right) \vx_i^T\mLambda^2 \vx_i.
\end{equation}

By Equations~\eqref{eq:di+:ub} and~\eqref{eq:mA:denombound+},
again using the fact that $(a+b)^2 \le 2(a^2+b^2)$,
\begin{equation*}
\frac{ d_{i+}^2 }
	{ \left\|\alpha_{i +} \vx_i + \vu_i\right\|_2^2 }
\le
\frac{ \left\|\vx_i\right\|_2^2 }
	{ \left\|\vu_i\right\|_2^2 + \left\|\vx_i\right\|_2^2 }
\left( \left\|\vu_i\right\|_2
        +\frac{1}{2}\left\|\vx_i\right\|_2 \right)^2
\le
2\left\|\vu_i\right\|_2^2 + \frac{1}{2} \left\|\vx_i\right\|_2^2,
\end{equation*}
and Equations~\eqref{eq:di-:ub} and~\eqref{eq:alpha-xu:lb} yield
\begin{equation*}
\frac{ d_{i-}^2 } { \left\|\alpha_{i -} \vx_i + \vu_i\right\|_2^2 }
\le
\frac{ 8 \left\|\vx_i\right\|_2^2 \left\|\vu_i\right\|_2^2 }{ \| \vu_i \|_2^2 }
\le 8 \left\|\vx_i\right\|_2^2.
\end{equation*}
Combining the above two displays,
\begin{equation} \label{eq:quadratic:uterm}
\left(\frac{d_{i+}^2}{\|\alpha_{i+} \vx_i + \vu_{i}\|_2^2}
	+ \frac{d_{i2}^2}{\|\alpha_{i-} \vx_i + \vu_{i}\|_2^2} \right)
	\vu_i^T \mLambda^2 \vu_i
\le \left( 2\left\|\vu_i\right\|_2^2
	+ \frac{ 17 }{ 2 } \| \left\|\vx_i\right\|_2^2 \right)
	\vu_i^T \mLambda^2 \vu_i.
\end{equation}

Combining Equations~\eqref{eq:alphai+:ub},~\eqref{eq:di+:ub} and~\eqref{eq:alpha+xu:lb},
\begin{equation*} \begin{aligned}
\frac{ \alpha_{i+} d_{i+}^2 }{ \|\alpha_{i+} \vx_i + \vu_{i}\|_2^2 }
& \le
\left( \|\vx_i\|_2 + \| \vu_i \|_2 \right)
\left( 2\left\|\vu_i\right\|_2^2 + \frac{ \left\|\vx_i\right\|_2^2 }{ 2 } \right)
\frac{ \left\|\vx_i\right\|_2 }
	{ \left\|\vu_i\right\|_2^2 + \left\|\vx_i\right\|_2^2 } \\
&\le
2 \left( \|\vx_i\|_2 + \| \vu_i \|_2 \right) \left\|\vx_i\right\|_2 \\
&\le 3 \|\vx_i\|_2^2 +  \| \vu_i \|_2^2,
\end{aligned} \end{equation*}
where we have used the fact that $2ab \le a^2 + b^2$.

Combining Equations~\eqref{eq:alphai-:ub},~\eqref{eq:di-:ub} and~\eqref{eq:alpha-xu:lb} and again using the fact that $2ab \le a^2 + b^2$,
\begin{equation*}
\frac{ \alpha_{i-} d_{i-}^2 }{ \|\alpha_{i-} \vx_i + \vu_{i}\|_2^2 }
\le
 \frac{\left\|\vu_i\right\|_2}{\left\|\vx_i\right\|_2}
\left\|\vx_i\right\|_2^2 \left\|\vu_i\right\|_2^2
\frac{ 8 }{ \| \vu_i \|_2^2 }
\le 
8 \left\|\vu_i\right\|_2 \left\|\vx_i\right\|_2
\le 4\left( \left\|\vu_i\right\|_2^2 + \left\|\vx_i\right\|_2^2 \right).
\end{equation*}
Combining the above two displays,
\begin{equation*}
2\left(\frac{\alpha_{i+} d_{i+}^2}{\|\alpha_{i+} \vx_i + \vu_{i}\|_2^2}
        + \frac{\alpha_{i-} d_{i-}^2}{\|\alpha_{i-} \vx_i + \vu_{i}\|_2^2}
	\right) \vx_i^T \mLambda^2 \vu_i
\le
2\left(7 \|\vx_i\|_2^2 + 5\| \vu_i \|_2^2 \right)
	\vx_i^T \mLambda^2 \vu_i.
\end{equation*}
Using the fact that $2 \vx_i^T \mLambda^2 \vu_i \le \vx_i^T \mLambda^2 \vx_i + \vu_i^T \mLambda^2 \vu_i$,
\begin{equation} \label{eq:quadratic:crossterm}
\begin{aligned}
2 & \left(\frac{\alpha_{i+} d_{i+}^2}{\|\alpha_{i+} \vx_i + \vu_{i}\|_2^2}
        + \frac{\alpha_{i-} d_{i-}^2}{\|\alpha_{i-} \vx_i + \vu_{i}\|_2^2}
	\right) \vx_i^T \mLambda^2 \vu_i \\
&~~~~~~\le \left(7 \|\vx_i\|_2^2 + 5\| \vu_i \|_2^2 \right)
	\left( \vx_i^T \mLambda^2 \vx_i + \vu_i^T \mLambda^2 \vu_i \right).
\end{aligned} \end{equation}
Applying this, along with Equations~\eqref{eq:quadratic:xterm} and~\eqref{eq:quadratic:uterm} to bound the right-hand side of Equation~\eqref{eq:frobexpansion},
\begin{equation*} \begin{aligned}
\left\|\left(\tmSig_i - \mI_d\right)\tmV_i^T \tmLambda\right\|_F^2
&\le
\left( 12 \| \vu_i \|_2^2 + \| \vx_i \|_2^2 \right) \vx_i^T\mLambda^2 \vx_i \\
&~~~~~~+ \left(7 \|\vx_i\|_2^2 + 5\| \vu_i \|_2^2 \right)
        \left( \vx_i^T \mLambda^2 \vx_i + \vu_i^T \mLambda^2 \vu_i \right) \\
&~~~~~~+
\left( 2\left\|\vu_i\right\|_2^2
        + \frac{ 17 }{ 2 } \| \left\|\vx_i\right\|_2^2 \right)
        \vu_i^T \mLambda^2 \vu_i \\
&\le \left( 17 \left\|\vu_i\right\|_2^2 + 8 \left\|\vx_i\right\|_2^2
	\right) \vx_i^T\mLambda^2 \vx_i
	+
	\left( 7 \left\|\vu_i\right\|_2^2 + \frac{27}{2} 
						\left\|\vx_i\right\|_2^2
		\right) \vu_i^T \mLambda^2 \vu_i .
\end{aligned} \end{equation*}
The result follows by trivially upper bounding the coefficients of $\|\vu_i\|_2^2$ and $\|\vx_i \|_2^2$.
\end{proof}

\begin{lemma} \label{lem:tech:clever}
For any vector $\va \in \R^d$ for $d \ge 2$ such that $\|\va\|_2 = 1$, there exists a vector $\vz \in \R^d$ with $|\vz_l| = 1/\sqrt{d}$ for all $l \in [d]$, such that $|\vz^T \va| \leq \sqrt{2/3}.$
\end{lemma}
\begin{proof}
For a set $S \subseteq [d]$, define $\vz \in \R^d$ according to
\begin{equation*}
\vz_\ell = \begin{cases}
		\operatorname{sign}(\va_{\ell})/\sqrt{d} &\mbox{ if } \ell \in S \\
		-\operatorname{sign}(\va_{\ell})/\sqrt{d} &\mbox{ otherwise. }
	\end{cases}
\end{equation*}
To see that $|\vz^T \va| \le \sqrt{2/3}$, note that by definition of $\vz$,
\begin{equation*} 
|\vz^T \va|
= \frac{1}{\sqrt{d}}
	\left|\sum_{l \in S} |\va_l| - \sum_{l \in S^c} |\va_l|\right|\\
\leq \frac{1}{\sqrt{d}}
     \max\left\{
		\sum_{l \in S} |\va_l| - \sum_{l \in S^c} |\va_l|,
		\sum_{l \in S^c} |\va_l| - \sum_{l \in S} |\va_l|
		\right\}.
\end{equation*}
Letting $\va_S$ denote the vector $\va$ with indices outside of $S$ set to zero, and defining $\va_{S^c}$ analogously, Jensen's inequality implies
\begin{equation} \label{eq:Jensen-bound}
|\vz^T \va|
\leq \frac{1}{\sqrt{d}}
	\max\left\{
		\sqrt{|S|}\|\va_S\|_2,
		\sqrt{|S^c|}\|\va_{S^c}\|_2
		\right\}
\end{equation}
If there exists a set $S$ is such that both $\|\va_S\|^2_2 \leq 2/3$ and $\|\va_{S^c}\|^2_2 \leq 2/3$, then the proof is complete, since then
\begin{equation*}
|\vz^T \va| \le 
\sqrt{ \frac{ 2 }{ 3 d } \max\{ |S|, |S^c| \} }
\le \sqrt{ \frac{ 2 }{ 3 } }.
\end{equation*}

Suppose, then, that no such $S$ exists.
That is, for any $S \subseteq [d]$,
either $\|\va_S\|^2_2 > 2/3$ or $\|\va_{S^c}\|^2 > 2/3$.
Observe that $\va_\ell^2 \le 1/3$ for any $\ell \in [d]$,
since if $\va_\ell^2 > 1/3$,
taking $S = \{ \ell \}$ so that $|S| = 1$, 
Equation~\eqref{eq:Jensen-bound} implies
\begin{equation*}
|\vz^T \va|
\le
\max\left\{
        \frac{ \sqrt{|S|}\|\va_S\|_2 }{ \sqrt{d} },                                     \frac{ \sqrt{|S^c|}\|\va_{S^c}\|_2 }{ \sqrt{d} } \right\}
\le
\max\left\{ \frac{1}{\sqrt{d}}, \sqrt{ \frac{2(d-1)}{3d} } \right\}
\le \sqrt{ 2/3 }.
\end{equation*}

Without loss of generality, 
suppose that $S$ is such that $\|\va_S\|^2_2 > 2/3$
and $\|\va_{S^c}\|^2  < 1/3$.
For any $\ell \in S$, consider removing $\ell$ from $S$ to obtain
$\Stilde = S \setminus \{\ell\}$.
If $\| \va_\Stilde \|^2 \le 2/3$ and $\| \va_{\Stilde^c} \|^2 \le 2/3$,
then we have contradicted our assumption.
Thus, either $\| \va_\Stilde \|^2 > 2/3$ or $\| \va_{\Stilde^c} \|^2 > 2/3$.
If the latter, then $\va_\ell^2 > 1/3$, leading to a contradiction.
Therefore,
\begin{equation*}
2/3 < \|\va_\Stilde \|^2_2 \le \| \va_S \|^2.
\end{equation*}
Note that $\Stilde$ must be non-empty, since otherwise $\| \va_\Stilde \|=0$, and therefore we can repeat our argument.
Repeating this argument enough times, we arrive at a minimal set $T \subseteq S$
such that $\| \va_T \|_2^2 > 2/3$,
and for any $\ell \in T$, $\| \va_{T \setminus \{\ell \} } \|_2^2 \le 2/3$.
If $\| \va_{T \setminus \{\ell\}} \|_2^2 \le 1/3$, we have again found $\ell \in [d]$ such that $|\va_\ell|>1/3$, a contradiction.
Therefore, 
\begin{equation*}
\frac{1}{3} \le \| \va_{T \setminus \{\ell\}} \|_2^2 \le \frac{2}{3},
\end{equation*}
and a similar bound holds for $\| \va_{T^c \cup \{\ell\}} \|_2^2$,
completing the proof.
\end{proof}

\end{document}